%% file: SimpleOrbifolds-2020-12-23.tex
\documentclass[12pt]{amsart}
\usepackage{amsmath,amsthm,amsfonts,amssymb,stmaryrd,mathrsfs,enumitem}
\usepackage{latexsym}
\usepackage{fullpage}
\usepackage{diagbox,rotating}
\usepackage{multirow}
\usepackage{enumitem}
\usepackage{prettyref}
\usepackage{a4,color,palatino,fancyhdr}
\usepackage{graphicx}
\usepackage{float}
\usepackage[small, bf, margin=90pt, tableposition=bottom]{caption}
\RequirePackage{amssymb}
\RequirePackage[T1]{fontenc}
\usepackage{hyperref}
\usepackage{amscd}

\usepackage{xypic}
\input{xy}
\xyoption{all}
\xyoption{poly}
\usepackage[all]{xy}
\usepackage{tikz}
\usepackage{cases}
\usepackage{verbatim}
\newtheorem{thm}{Theorem}[section]
 \newtheorem{cor}{Corollary}[section]
 \newtheorem{lem}{Lemma}[section]
 \newtheorem{prop}{Proposition}[section]
 \theoremstyle{definition}
 \newtheorem{defn}{Definition}[section]

 \theoremstyle{remark}
 \newtheorem{rem}{Remark}
 \newtheorem{example}{Example}[section]
 
 \numberwithin{equation}{section}
 \theoremstyle{plain}

\newtheorem*{theoremA}{Theorem A}

\newtheorem*{theoremB}{Theorem B}

\begin{document}

\newcommand{\auths}[1]{\textrm{#1},}
\newcommand{\artTitle}[1]{\textsl{#1},}
\newcommand{\jTitle}[1]{\textrm{#1}}
\newcommand{\Vol}[1]{\textbf{#1}}
\newcommand{\Year}[1]{\textrm{(#1)}}
\newcommand{\Pages}[1]{\textrm{#1}}
\newcommand{\RNum}[1]{\uppercase\expandafter{\romannumeral #1\relax}}

\title{Topology and geometry of flagness and beltness of simple orbifolds}

\author{Zhi L\"u and Lisu Wu}

\address{School of Mathematical Sciences, Fudan University\\ Shanghai\\ 200433\\ China}
\email{zlu@fudan.edu.cn}
\address{School of Mathematical Sciences, Fudan University\\ Shanghai\\ 200433\\ China}
\email{wulisuwulisu@qq.com}
\thanks{Partially supported by the grant from NSFC (No. 11971112).
}

\begin{abstract}
We consider a class of  right-angled Coxeter orbifolds, named as simple orbifolds, which are a generalization of simple polytopes. Similarly to  manifolds over simple polytopes, the topology and geometry of manifolds over simple orbifolds
are closely related to the  combinatorics and orbifold structure of simple orbifolds. We generalize the notions of flag and belt in the setting of simple polytopes into the setting of simple orbifolds.
 To describe the topology and geometry of a simple orbifold in terms of its
combinatorics, we focus on  {\em simple handlebodies} (that is,  simple orbifolds which can be obtained from  simple polytopes by gluing  some disjoint specific codimension-one faces). We prove the following two  main results in terms of combinatorics, which can be understood as ``Combinatorial Sphere Theorem" and ``Combinatorial Flat Torus Theorem" on  simple handlebodies:
\begin{itemize}
\item [(A)]  A simple handlebody  is orbifold-aspherical if and only if it is flag.

\item [(B)] There exists a rank-two free abelian subgroup in $\pi_1^{orb}(Q)$ of an orbifold-aspherical simple handlebody $Q$ if and only if  it contains an $\square$-belt.
\end{itemize}
Furthermore,  based on such two results and some results of geometry,  it is shown that the existence of  some  curvatures on a certain manifold cover (manifold double) over a simple handlebody $Q$ can be characterized  in terms of  the combinatorics of $Q$. In 3-dimensional case, together with the theory of hyperbolic 3-manifolds, we can induce a pure combinatorial equivalent description for a simple $3$-handlebody to admit a right-angled hyperbolic structure, which is a natural generalization of Pogorelov Theorem.
\end{abstract}

\maketitle
\section{Introduction} \label{Section 1}

In this paper, we consider a class of right-angled Coxeter $n$-orbifolds, named {\em simple orbifolds}, each $Q$ of which  satisfies the three conditions:
\begin{itemize}
\item [(a)]  $|Q|$ is  compact and connected with $\partial |Q|\not=\emptyset$ where $|Q|$ denotes the underlying space of $Q$;

\item [(b)] The {\em nerve} of $Q$, denoted by $\mathcal{N}(Q)$,  is  a triangulation of the boundary $\partial |Q|$, where $\mathcal{N}(Q)$ is the abstract simplicial complex with a vertex for each facet of $Q$  and a $(k-1)$-simplex for each nonempty $k$-fold intersection;

\item [(c)]  Each facet  in $Q$ is a simple polytope (note that when $n\leq 3$, this condition will be automatically omitted).
\end{itemize}
 Simple polytopes together with the natural structure of right-angled Coxeter orbifold provide the canonical examples of simple orbifolds, and  their nature has given rise to many interesting and beautiful connections among  topology, geometry,  combinatorics and so on.
 For example,  we can see these from Pogorelov Theorem, the theory of toric varieties, toric geometry, toric topology etc (see, e.g. \cite{Fu93, Dan78, Dan82, TT}).
 Although simple orbifolds form a  much wider class than simple polytopes, we can still expect such these connections in the setting of simple orbifolds. In particular, similar to  the case of simple polytopes, we wish that the  combinatorial  structure  of a simple orbifold still plays an important role in its own ways. This suggests us to pay more attention on the case of simple handlebodies.

 \vskip .2cm
A {\em simple $n$-handlebody} $Q$ is  a simple $n$-orbifold $Q$ such that its underlying space $|Q|$ is an $n$-dimensional handlebody and  it can be cut into a simple polytope $P_Q$ along some
codimension-one B-belts (named {\em cutting belts}, for the notion of $B$-belts, see \autoref{Def-B}), where an $n$-dimensional handlebody of genus $\mathfrak{g}\geq 0$ is a tubular neighborhood of the wedge sum of $\mathfrak{g}$ circles in $\mathbb{R}^n$ (of course, an $n$-dimensional handlebody of genus $0$ is exactly an $n$-ball). In other words, a simple $n$-handlebody can be obtained from a simple $n$-polytope by gluing its some specific facets.
 In  3-dimensional case, we can cancel  the restriction condition that  ``it can be cut into a simple polytope along some
codimension-one B-belts" in the above definition. This is because we can show that for a simple 3-orbifold $Q$ such that  $|Q|$ is a 3-handlebody of genus $\mathfrak{g}$, if $\mathfrak{g}>0$ then $Q$ can always be cut into a simple polytope along some
codimension-one B-belts; and if $\mathfrak{g}=0$ then $Q$ must be a simple 3-polytope (see \autoref{lemma-1}).
 An example is shown in \hyperref[special simple 3-handlebody]{Figure 1}.
       \begin{figure}[h]\label{special simple 3-handlebody}
\centering
\def\svgwidth{0.6\textwidth}
\includegraphics[width=0.75\textwidth]{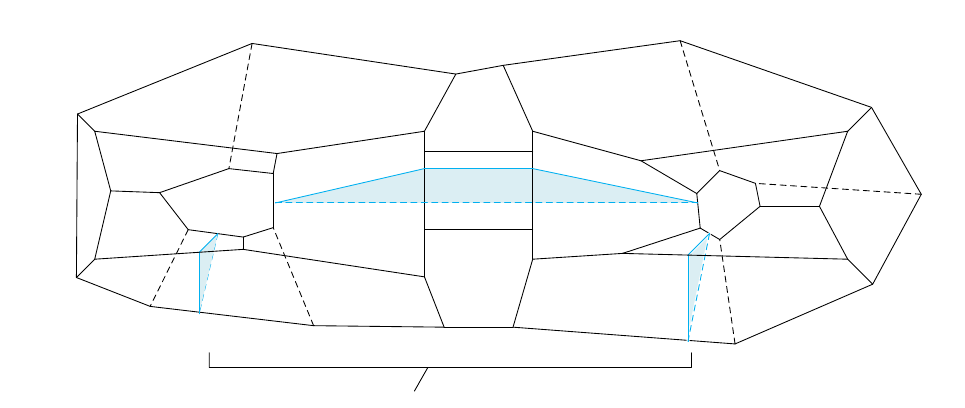}
 \put(-225,-3.5){cutting belts}
\caption{A  simple $3$-handlebody of genus $2$.}
\end{figure}

\vskip .2cm

We shall  carry out our work from the following some aspects:
\begin{itemize}
\item [(I)] 
We   generalize the notions of belt and flag in the setting of simple polytopes into the setting of simple orbifolds (see \autoref{Def-B} and  \autoref{Def-T}).
Indeed, there is a quite difference because the underlying space of a simple orbifold is not contractible in general. As we shall see, the flagness of a simple orbifold $Q$ can not be defined by  the flagness of its nerve $\mathcal{N}(Q)$ in general. Actually, its definition is given in such a different way that  $|Q|$ is aspherical and  $Q$ contains no $\Delta^k$-belt for any $k\geq 2$.
\vskip.1cm
\item [(II)] To understand the implicit   structure of a simple orbifold, we introduce the notions of ``right-angled Coxeter cells"  and `` right-angled Coxeter cellular  complex". Then we see that for a  right-angled Coxeter cellular complex $X$, its orbifold fundamental group $\pi_1^{\text{orb}}(X)$ is isomorphic to $\pi_1^{\text{orb}}(X^2)$ where $X^2$ is the 2-skeleton of $X$ (\autoref{p1}).
     For  a  simple handlebody $Q$, we can
    give an explicit right-angled Coxeter cellular  decomposition of  $Q$, so that we can
    obtain an explicit presentation of  $\pi_1^{\text{orb}}(Q)$, which is just an iterative HNN-extension over some right-angled Coxeter group. Indeed, generally $\pi_1^{\text{orb}}(Q)$ will not be the Coxeter group of $Q$, given by only reflections on facets of $Q$, and it actually contains torsion-free generators.
    In addition, we can also give an explicit right-angled Coxeter cubical cellular  decomposition of  $Q$. See \hyperref[section-61]{Subsection 7.1}.

    \vskip.1cm
\item [(III)] We will use a  ``basic construction"  from Davis~\cite[Chapter 5]{D2}, which plays an important role on our work.  This basic construction  tells us  that
each simple $n$-orbifold $Q$ can be finitely covered  by a  closed $n$-manifold $M_Q$ with an action of some 2-torus group $G$, which is called a {\em manifold double} over $Q$.
By \cite[Proposition 1.51]{ALR}, we can define the {\em orbifold homotopy group}
$$\pi_k^{orb}(Q):=\pi_k(EG\times_GM_Q).$$
By associating the Borel fibration $EG\times_GM_Q\longrightarrow BG$,
$\pi_k^{orb}(Q)\cong \pi_k(M_Q)$ for $k\geq 2$. Thus, $Q$ is orbifold-aspherical if and only if $M_Q$ is aspherical. It follows from \cite[Theorem 2.2.5]{DJS2} that a simple polytope $P$ is orbifold-aspherical if and only if $P$ is flag, so in general, a  simple $n$-handlebody $Q$ may not be orbifold-aspherical yet although $|Q|$ is aspherical. In addition, using the basic construction,
 we can also use $\pi_1^{\text{orb}}(Q)$ and $P_Q$ to construct  the orbifold universal cover $\widetilde{Q}$ of $Q$.
\vskip.1cm
\item [(IV)] Based upon  (I), (II) and (III), together with Cartan-Hadamard Theorem and the work of Gromov on nonpositive curvature~\cite[Section 4.2]{G}, we obtain the Combinatorial  Sphere Theorem on  simple handlebodies (see~\hyperref[DJS-small]{Theorem A}), which is essentially a generalization of \cite[Theorem 2.2.5]{DJS2} of Davis, Januszkiewicz and Scott for simple polytopes. Making use of Tits Theorem~\cite[Theorem 3.4.2]{D2} of Coxeter groups and  the normal form theorem of HNN-extensions~\cite[Theorem 2.1, Page 182]{LS}, we also obtain the Combinatorial  Flat Torus Theorem on  simple handlebodies (see~\hyperref[FTT]{Theorem B}).
\vskip.1cm
As further applications of our two main results,  it is shown with  some results of geometry that the existence of  some  curvatures on  manifold double  over a simple handlebody $Q$ can be characterized by the $B$-belts in $Q$. In 3-dimensional case, combining the theory of hyperbolic 3-manifolds again, we give a pure combinatorial equivalent description for a simple $3$-handlebody to admit a right-angled hyperbolic structure, which is a generalization of Pogorelov Theorem, and in particular, it  also agrees with the pure combinatorial description of Pogorelov Theorem.
\end{itemize}

Now let us state our main results as follows.

\begin{theoremA} \label{DJS-small}
Let $Q$ be a simple handlebody of dimension $n\geq 3$. Then
 $Q$ is orbifold-aspherical if and only if it is flag.
\end{theoremA}

The proof of \hyperref[DJS-small]{Theorem A} actually  involves  Cartan-Hadamard Theorem, the work of Gromov on nonpositively curved cubical complex and Davis's method in~\cite[Chapter 8]{D2}. Specifically speaking, similar to hyperbolization of polyhedra~\cite{D1}, we construct a cubical complex structure of the orbifold
universal cover $\widetilde{Q}$ such that the link of each vertex is a simplicial complex which is related to  the combinatorial structure of $Q$. Then applying a result of Gromov (\autoref{LMG}) induces that $\widetilde{Q}$ is non-positively curved if  $Q$ is flag. Now Cartan-Hadamard Theorem implies that $\widetilde{Q}$ is aspherical  if  $Q$ is flag. Moreover, by Davis's method \cite[Chapter 8]{D2}, we calculate the homology groups of  $\widetilde{Q}$ which are related to $\pi_1^{orb}(Q)$ and the combinatorial structure of $Q$.

\begin{rem}\label{Thm2-rem}\
\begin{itemize}
\item[(a)] \hyperref[DJS-small]{Theorem A} is essentially a generalization of the result of Davis, Januszkiewicz and Scott   for small covers in \cite[Theorem 2.2.5]{DJS2}.
\item[(b)] If $Q$ is a simple handlebody, then by \autoref{Def-T},   $Q$ is flag if and only if it contains no $\Delta^k$-belt for any $k\geq 2$. Furthermore, \hyperref[DJS-small]{Theorem A} tells us that if $Q$ is not flag, then there must exist an $\Delta^k$-belt for some $k\geq 2$ in $Q$, so that the pullback of the embedding $\Delta^k\hookrightarrow Q$ via the projection
    $M_Q\longrightarrow Q$ gives an (equivariant) embedding $S^k\hookrightarrow M_Q$ which represents a nontrivial element in $\pi_k(M_Q)$, as shown in the following diagram:
    $$
\xymatrix{\ar @{} [dr] |{}
S^k~ \ar@{^{(}->}[r]\ar[d] & M_Q \ar[d] \\
\Delta^k ~ \ar@{^{(}->}[r]& Q.
}
$$
\end{itemize}
\end{rem}

Our another main result is the Combinatorial  Flat Torus Theorem on  simple handlebodies.
  After a tedious proof with the use of Tits Theorem~\cite[Theorem 3.4.2]{D2} of Coxeter groups and  the normal form theorem of HNN-extensions~\cite[Theorem 2.1]{LS},
we  characterize  the rank two free abelian  subgroup $\mathbb{Z}\oplus \mathbb{Z}$ in $\pi_1^{orb}(Q)$ in terms of combinatorics of $Q$.
\begin{theoremB}\label{FTT}
Let $Q$ be a flag  simple handlebody of dimension $n\geq 3$. Then there is a rank two free abelian  subgroup $\mathbb{Z}\oplus \mathbb{Z}$ in $\pi_1^{\text{orb}}(Q)$ if and only if $Q$ contains an $\square$-belt.
\end{theoremB}

\begin{rem}
Similar to the pullback way in Remark~\autoref{Thm2-rem}~(b), the existence of an $\square$-belt in $Q$ actually means that there exists an essential embedding of 2-dimensional torus $T^2$ in $M_Q$, which is an obstacle of the existence of hyperbolic structure or negative curvature on $M_Q$.
For the Flat Torus Theorem of non-positively curved spaces, one can refer to \cite[Charter 7, Part II]{BH} or \cite{LY}.
\end{rem}

Next, as further consequences of our two main results, we discuss the topology and geometry of covering spaces over a simple handlebody. Together with some important results in geometry from \cite[Theorem 1.2]{SMY},
\cite[Proposition 4.9]{WY}, \cite[Chapter 7]{O98}, \cite[Proposition I.6.8]{D2} and \cite{KAP}, for a simple handlebody $Q$,
we  obtain some relations between the existence of some curvatures on $M_Q$ and the combinatorics  of $Q$.
\begin{cor}\label{cct}
Let $Q$ be a  simple handlebody of dimension $n\geq 2$,  $M_Q$ be the manifold double over $Q$, and $\widetilde{Q}$ be the orbifold universal  cover of $Q$. Then  we have that
\begin{enumerate}[label=(\roman*)]
\item The following statements  are equivalent.
\begin{itemize}
\item[$(1)$] $M_Q$ is non-positively curved;
\item[$(2)$] $\widetilde{Q}$ is CAT(0);
\item[$(3)$]  $Q$ is flag (this is equivalent to that $Q$ is orbifold-aspherical);
\item[$(4)$] $M_Q$ is aspherical.
\end{itemize}
\vskip.1cm
\item If $M_Q$ admits a strictly negative curvature then $Q$ is flag and contains no any $\square$-belt. In particular, if $Q$ is a simple polytope $P$, then $M_P$ admits a strictly negative curvature if and only if $P$ is flag and contains no  any $\square$-belt.
\end{enumerate}
In 3-dimensional case,
\begin{itemize}
\item[(iii)] $M_Q$ is hyperbolic if and only if it is flag and contains no any $\square$-belt. Moreover, $Q$ admits a right-angled hyperbolic structure if and only if it is flag and contains no any $\square$-belt. In this case,  $\widetilde{Q}\approx \mathbb{H}^3$.
    \vskip.1cm
\item[(iv)] When $Q$ is a simple 3-polytope,  $M_Q$ admits a positive scalar curvature if and only if  every $2$-dimensional belt in $Q$ is $\Delta^2$, or $Q$ is just a tetrahedron.
\end{itemize}
\end{cor}

\begin{rem} The proof of  \autoref{cct} will mainly be finished in \hyperref[section7]{Section 7}.
\begin{itemize}
\item
\autoref{cct}~(i)--(ii) are based on Gromov's results. See \cite{G}.
A metric space is said to be of curvature $k$  if it is locally a CAT($k$) space. A comparison theorem in \cite[Theorem 1A.6, Page-173]{BH} tells us that
a smooth manifold has curvature $\leq k$ if and only if it admits a Riemannian metric with sectional curvature $\leq k$. Hence  the curvature in the statements of \autoref{cct}~(i)--(ii)  can be replaced by  sectional curvature. See \cite[Theorem 1.2]{SMY} for simple polytopes.
\vskip.1cm
\item There are examples of closed orientable 3-manifolds that are aspherical but do not support a Riemannian metric with non-positively sectional curvature (see \cite{Leeb}).
 \vskip.1cm
\item \autoref{cct}~(iii) is the Hyperbolization Theorem on simple $3$-handlebodies. 
The Hyperbolization Theorem on general right-angled Coxeter 3-orbifolds was considered by
Otal in \cite{O98}.
 An {\em irreducible and atoroidal} $3$-manifold $Q$ with corners defined by Otal in \cite[P168]{O98} implies essentially that all involved $\Delta^2$ and $\square$ suborbifolds in $Q$ are not belts. This  is actually equivalent to saying that $Q$ is flag and  contains no any $\square$-belt. Here our statement is of more combinatorial.
\vskip.1cm
\item \autoref{cct}~(iv) is also a restatement  of a result of \cite{WY}.
A vc($k$) in \cite{WY} is equivalent to the simple $3$-polytope in \autoref{cct}~(iv).
\vskip.1cm
\item All $2$-dimensional right-angled Coxeter orbifold can be classified by their orbifold Euler numbers, see \cite{TH}. 
\vskip.1cm
\item
With a bit additional argument,   the ``simple'' condition in 3-dimensional case can be generalized to the case of a  right-angled Coxeter $3$-handlebody whose nerve is  an ideal triangulation of its boundary, where the concept of ideal triangulation can be referred to \cite[Section 2]{FST08}. In this case,  there may exist  bad $3$-handlebodies, that is, as  right-angled Coxeter orbifolds, they cannot be covered by  $3$-manifolds. So these bad orbifolds cannot admit any hyperbolic metric. See \autoref{Lemma-62}.
Although so, we can obtain that a  right-angled Coxeter $3$-handlebody with  ideal nerve is hyperbolic if and only if it is very good, flag and contains no any $\square$-belt, see
\hyperref[s-63]{Subsubsection 7.3.2} for details.
More generally, only the flag condition and no $\square$-belt condition cannot characterize the (right-angled) hyperbolicity of a ``non-simple''  right-angled Coxeter $3$-handlebody. Here a right-angled Coxeter $3$-handlebody $Q$ is ``non-simple''  if its nerve is not a triangulation or an ideal triangulation of $\partial |Q|$.
 An example is given in \hyperref[s-64]{Subsubsection 7.3.2}.
\end{itemize}
\end{rem}
This paper is organized as follows.
In \hyperref[section2]{Section 2},  we review the notions of (right-angled Coxeter) orbifolds and  manifolds with corners. We introduce the right-angled Coxeter cellular decomposition of  right-angled Coxeter orbifolds, and discuss their orbifold fundamental groups. In addition, we also introduce the theory of fundamental domain.
In \hyperref[section3]{Section 3} we generalize the notions of $B$-belts and flag from simple polytopes to simple orbifolds.
In \hyperref[section4]{Section 4}, we give a right-angled Coxeter cellular  decomposition of a  simple $n$-handlebody $Q$, so that we can explicitly give a presentation of  orbifold fundamental group $\pi^{orb}_1(Q)$. We show that this presentation of  orbifold fundamental group $\pi^{orb}_1(Q)$ is an iterative HNN-extension of some right-angled Coxeter group. Moreover, the orbifold universal cover of $Q$ is constructed by using $\pi^{orb}_1(Q)$ and the simple polytope $P_Q$ associated to $Q$.
In \hyperref[section5]{Section 5}, we review the work of Gromov, and  compute the homology groups of the manifold double and  universal cover of a simple handlebody $Q$ by  Davis' method, which are useful in the proof of \hyperref[DJS-small]{Theorem A}.
Then we prove \hyperref[DJS-small]{Theorem A}.
 In \hyperref[section6]{Section 6}, we show that the existence of a rank-two free abelian subgroup in the orbifold fundamental group of a flag simple handlebody $Q$ is characterized by an $\square$-belt in $Q$ (\hyperref[FTT]{Theorem B}).
  In \hyperref[section7]{Section 7},  applying  \hyperref[DJS-small]{Theorem A}, \hyperref[FTT]{Theorem B} and some  results of geometry,  we discuss the existence of  some  curvatures on  manifold double  over a simple handlebody $Q$ in terms of  the combinatorics of $Q$.

\section{Preliminaries}
\label{section2}

\subsection{Orbifold}
As a generalization of manifolds, an $n$-dimensional {\em orbifold} $\mathcal{O}$ is a singular space which is locally modelled on the quotient of  a finite group acting on an open subset of $\mathbb{R}^n$. For any $p\in \mathcal{O}$, there is  an {\em orbifold chart} $(U,G,\psi)$ satisfying that $U$ is an $n$-ball centered
at origin $o$ and $\psi^{-1}(p)=o$, where $\psi:U\rightarrow U/G$ is the projection map.
In particular, the origin $o$ is fixed by $G$. We called $G$ the local group at $p$.

\vskip.2cm

\begin{defn}[Thurston {\cite[Definition 13.2.2]{TH}}]
A {\em covering orbifold} of an orbifold $\mathcal{O}$ is an orbifold $\widetilde{\mathcal{O}}$ with a projection $\pi: \widetilde{\mathcal{O}}\rightarrow \mathcal{O}$, satisfying that:
\begin{itemize}
\item $\forall~x\in \mathcal{O}$ has a neighborhood $V$ which is identified with an open subset $U$ of $\mathbb{R}^n$ module a finite group $G_x$, such that each component $V_i$ of $\pi^{-1}(V)$  is homeomorphic to $U/\Gamma_i$, where $\Gamma_i<G_x$ is some subgroup;
\item $\pi|_{V_i}: V_i\rightarrow V$ corresponds to the natural projection $U/\Gamma_i\rightarrow U/G_x$.
\end{itemize}
\end{defn}
An orbifold is {\em good} (resp. {\em very good})
 if it can be   covered (resp. finitely) by a manifold. Otherwise it is {\em bad}.
Any orbifold $\mathcal{O}$ has an universal cover $\widetilde{\mathcal{O}}$, see \cite[Proposition 13.2.4]{TH}.
\vskip.2cm

In general, the {\em orbifold fundamental group} of an orbifold is defined as the deck transformation group of its universal cover, see \cite[Definition 13.2.5]{TH}. Another  equivalent definition is the use of the notion of based orbifold loops, that is, the orbifold fundamental group is defined as the homotopy classes of based orbifold loops.
 For more details, see~\cite[Section 3]{C}.
\begin{example}
Let $D^2$ be the unit disk in $\mathbb{R}^2$. A transformation $r$ on $D^2$ via $r(x,y)=(x,-y)$ gives a reflective $\mathbb{Z}_2$-action on $D^2$. The orbit space $D^2/\mathbb{Z}_2$ has a natural orbifold structure.  Any $(x,0)\in D^2/\mathbb{Z}_2$ is a singular point with local group $\mathbb{Z}_2$. Since $D^2$ is contractible,  $\pi_1^{orb}(D^2/\mathbb{Z}_2)\cong \mathbb{Z}_2$ is generated by the transformation $r$.
\vskip.1 cm
In the viewpoint of orbifold loops, any path between $(x_1,0)$ and $(x_2, y_2)$ with $y_2>0$ can be viewed as a non-trivial orbifold loop.
It is clear that $D^2/\mathbb{Z}_2\cong D^1\times D^1/\mathbb{Z}_2\simeq D^1/\mathbb{Z}_2$. Hence, $\pi_1^{orb}(D^2/\mathbb{Z}_2)\cong \mathbb{Z}_2$ is generated by a based orbifold loop $D^1/\mathbb{Z}_2$,  see \hyperref[OL]{Figure 2}.

 \begin{figure}[h]\label{OL}
\centering
\def\svgwidth{0.65\textwidth}
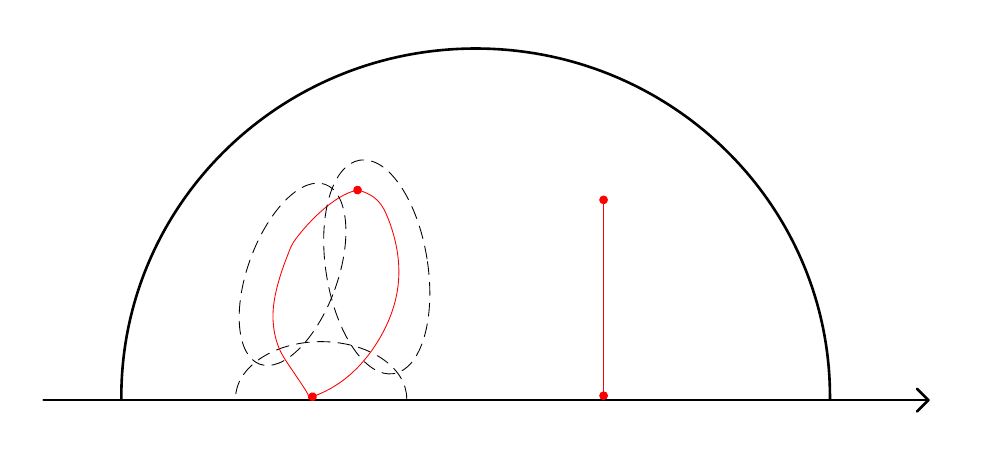
\caption{Orbifold loop}
\end{figure}
\end{example}
\begin{example}[\cite{ALR}]
If a compact Lie group $G$ acts smoothly and almost freely on a manifold $M$, or a discrete group $G$ acts properly discontinuously on a  manifold $M$, then the orbit space $M/G$ canonically inherits an orbifold structure. Here $M/G$ is called the {\em  quotient orbifold } by  $G$ acting on $M$.
\end{example}

\begin{defn}[{\cite[Proposition 1.51]{ALR}}]\label{HHG}
For a quotient orbifold $\mathcal{O}$ by a group $G$ acting on $M$, its {\em $k$-th orbifold homotopy group} is defined as the $k$-th  homotopy group of its Borel construction
$EG\times_G M$
$$
\pi_k^{orb}(\mathcal{O},x ):=\pi_k(EG\times_G M, \widetilde{x})
$$
\end{defn}

An orbifold $\mathcal{O}$ is said to be {\em orbifold-aspherical} if $\pi_k^{orb}(\mathcal{O})=0$ for all $k\geq 2$.
\begin{rem} It is well-known that for an arbitrary orbifold,  there are also several equivalent definitions of  orbifold homotopy group $\pi_k^{orb}(\mathcal{O},x)$ of $\mathcal{O}$, which can be given in the following ways:
\begin{itemize}
\item The $k$-th homotopy group of the classifying space of orbifold groupoid of $\mathcal{O}$ (see \cite[Definition 1.50]{ALR})
$$
\pi_k^{orb}(\mathcal{O},x):=\pi_k(B\mathcal{G}, \widetilde{x}).
$$
\item The $(k-1)$-th homotopy groups of  the based loop space $(\Omega(\mathcal{O},x), \widetilde{x})$ (see \cite[Definition 3.2.1]{C})
$$
\pi_k^{orb}(\mathcal{O},x):=\pi_{k-1}(\Omega(\mathcal{O},x), \widetilde{x}).
$$
\end{itemize}
\end{rem}
Let $p:M\rightarrow \mathcal{O}$ be a regular orbifold cover over a good orbifold $\mathcal{O}$, where $M$ is a manifold.  Then by orbifold covering theory~\cite{C}, 
$$
\pi_k^{orb}(\mathcal{O},x)\cong \pi_k(M,\widetilde{x})
$$
for $k\geq 2$. 
Thus we have that
\begin{cor}\label{asph}
An good orbifold is orbifold-aspherical if and only if its covering manifold is aspherical.
\end{cor}

 Assume $G$ is the deck transformation group for $p:M\rightarrow \mathcal{O}$. Then there is a probably non-split group extension
$$
1\rightarrow \pi_1(M)\rightarrow \pi_1^{orb}(\mathcal{O})\rightarrow G\rightarrow 1.
$$
\vskip.2cm
 See \cite{ALR,C,CR,Sa57} for more discussions  of orbifolds.
\subsection{Right-angled Coxeter orbifolds, manifolds with corners and their manifold covers}
Following \cite{DJ, D5}, a {\em right-angled Coxeter $n$-orbifold} $Q$ is a special $n$-orbifold locally modelled on the quotient $\mathbb{R}^n/(\mathbb{Z}_2)^n$ of the standard $(\mathbb{Z}_2)^n$-action on $\mathbb{R}^n$ by reflections across the coordinate hyperplanes.
  A {\em stratum} of codimension $k$ is the closure of a component of the subspace of $|Q|$ consisting of all points with local group $(\mathbb{Z}_2)^k$, where $|Q|$ denotes the underlying space of $Q$.
It is easy to see that  $\mathbb{R}^n/(\mathbb{Z}_2)^n$ possesses the following properties:
\begin{itemize}
\item  Topologically and combinatorially,  $\mathbb{R}^n/(\mathbb{Z}_2)^n$ is
the standard simplicial cone $\mathcal{C}^n=\{(x_1, ..., x_n)\in \mathbb{R}^n|x_i\geq 0, 1\leq i\leq n\}$ in $\mathbb{R}^n$;
\item The local group  at $x=(x_1, ..., x_n)\in \mathbb{R}^n/(\mathbb{Z}_2)^n$ is the subgroup
$(\mathbb{Z}_2)^{c(x)}$,
where  $c(x)$ is the number of those coordinates $x_i=0$ in $x$, called the {\em codimension} of $x$;
\item For $0\leq k\leq n$,  $(\mathbb{Z}_2)^k$ as a local group determines ${n\choose k}$ strata of codimension $k$, each of which is isomorphic to
  $\mathbb{R}^{n-k}/(\mathbb{Z}_2)^{n-k}$.
\end{itemize}

\vskip .2cm

Davis in \cite[Section 6]{D3} (or \cite[Chapter 10, Page 180]{D2}) defined {\em $n$-manifolds with corners}, each of which is a Hausdorff space $X$ together with a maximal atlas of local charts onto open subsets of the standard simplicial cone $\mathcal{C}^n$ such that the overlap maps are homeomorphisms of preserving codimension,  where for any chart $\varphi: U\longrightarrow \mathcal{C}^n$,
the codimension  of any $x\in U$ is defined as $c(\varphi(x))$, also denoted by $c(x)$, and it is independent of the chart. An {\em open face}  of codimension $k$ is a component of $\{x\in X| c(x)=k\}$. A {\em face}  is the closure of such a component.

\vskip .2cm
A right-angled Coxeter orbifold $Q$ naturally inherits the structure of a manifold with corners. On the other hand, since the topological and combinatorial structure of $\mathcal{C}^n$ is compatible with that of  right-angled Coxeter orbifold on $\mathbb{R}^n/(\mathbb{Z}_2)^n$, an  $n$-manifold with corners naturally admits a right-angled Coxeter orbifold structure.  Furthermore, all  strata in  a right-angled Coxeter orbifold $Q$ bijectively correspond to all faces in $Q$ as a manifold with corners. A stratum or face of codimension one is called a {\em facet}.

\vskip.2cm
In this paper we are mainly concerned with a special class of right-angled Coxeter orbifolds, i.e., simple orbifolds, as defined in \hyperref[Section 1]{Section 1}.
Let $Q$ be a simple  $n$-orbifold with facet set  $\mathcal{F}(Q)=\{F_1, ..., F_m\}$. For some $n\leq k\leq m$, the map
$$\lambda: \mathcal{F}(Q)\longrightarrow (\mathbb{Z}_2)^k$$
is called  a {\em coloring} if for each   $l$-face $f^l$ in $Q$ (so there are exactly $n-l$ facets, say $F_{i_1}, ..., F_{i_{n-l}}$, whose intersection is $f^l$ since $Q$ is simple),
$\lambda(F_{i_1}), ..., \lambda(F_{i_{n-l}})$ are independent in $(\mathbb{Z}_2)^k$. Clearly, $l$-face $f^l$ in $Q$ determines  a subgroup $G_{f^l}$ generated by $\lambda(F_{i_1}), ..., \lambda(F_{i_{n-l}})$ via $\lambda$.
Note that  each $x\in \partial |Q|$ always lies in the relative interior of a unique face $f$.
Then  there is a {\em manifold cover} of $Q$  defined as follows:
\begin{equation}\label{E1}
\mathcal{U}(Q, (\mathbb{Z}_2)^k)=Q\times (\mathbb{Z}_2)^k/\sim
\end{equation}
where
$$(x,g)\sim (y,h)
\Longleftrightarrow
\begin{cases} x=y \text{ and } g=h & \text{ if } x\in \text{Int}(|Q|)\\
x=y  \text{ and } gh^{-1}\in G_f & \text{ if } x\in f\subset \partial |Q|.
\end{cases}
$$
Essentially this is a special case of ``basic construction" of Davis~\cite[Chapter 5]{D2}.
It follows from~\cite[Proposition 10.1.10]{D2} that $\mathcal{U}(Q, (\mathbb{Z}_2)^k)$ is an $n$-dimensional closed manifold and
  naturally admits an action of $(\mathbb{Z}_2)^k$ with quotient orbifold $Q$. So a simple orbifold is a very good orbifold.
  \begin{lem}
  A simple orbifold $Q$ is the quotient orbifold of $(\mathbb{Z}_2)^k$ acting on $\mathcal{U}(Q, (\mathbb{Z}_2)^k)$.
  \end{lem}

 If $k=n$, then $\mathcal{U}(Q, (\mathbb{Z}_2)^n)$ is called {\em small manifold cover} over $Q$ which is a generalization of small covers over simple polytopes (\cite{DJ}), but it may not exist even if $Q$ is a simple polytope (also
  see~\cite[Nonexamples 1.22]{DJ}).
  However, if $k=m$,
  we can take $\lambda(F_i)=e_i$ for each facet $F_i$ of $Q$ where $\{e_1, ..., e_m\}$ is the standard basis of $(\mathbb{Z}_2)^m$, such that there always exists such $\mathcal{U}(Q, (\mathbb{Z}_2)^m)$,  called the {\em manifold double}
  (\cite[Proposition 2.4]{D5}) over $Q$,  which is a generalization of real moment angled manifolds over simple polytopes (\cite{TT}).  In this case,  for simplicity, we use $M_Q$ to replace $\mathcal{U}(Q, (\mathbb{Z}_2)^m)$ here.

\vskip .2cm
 Many  works have been carried out with respect to the  topology and geometry of $\mathcal{U}(Q, (\mathbb{Z}_2)^k)$  by associating with the topology, geometry  and combinatorics of $Q$, especially for $Q$ to be a simple polytope (e.g., see~\cite{BLA, TT, D2, D1, DJ, DJS2, GS, SMY, LM, N, WY}). In addition, there are also various relative works with other topics and  viewpoints (e.g., see~\cite{BBCG, BM, CL,  CLY, CP,  GLs,  LT}).

 \vskip .2cm
Clearly we may define the orbifold homotopy group of a simple orbifold $Q$ with $m$ facets  as follows:
  $$\pi_k^{orb}(Q):=\pi_k(E(\mathbb{Z}_2)^m\times_{(\mathbb{Z}_2)^m}M_Q)$$
where  $E(\mathbb{Z}_2)^m\times_{(\mathbb{Z}_2)^m}M_Q\longrightarrow B(\mathbb{Z}_2)^m$ is the Borel fibration with fiber $M_Q$.
Since $B(\mathbb{Z}_2)^m$ is a $K(\pi, 1)$-space, we see that if $k\geq 2$, then $\pi_k^{orb}(Q)\cong \pi_k(M_Q)$, and if $k=1$, then there is a short exact sequence $$1\rightarrow \pi_1(M_Q)\rightarrow\pi_1^{orb}(Q)\rightarrow (\mathbb{Z}_2)^m\rightarrow 1.$$
Thus $Q$ is  orbifold-aspherical if and only if $M_Q$ is aspherical (see also~\autoref{asph}).


\subsection{The right-angled Coxeter cellular  decomposition}
Now let us introduce the right-angled Coxeter orbifold cellular decomposition for right-angled Coxeter orbifolds, which will play an important role on the calculation of  the orbifold fundamental groups  of right-angled Coxeter orbifolds. The more general notion of cellular decomposition of certain orbifolds was considered as q-cellular complex (or, q-CW complex) in  \cite{BNSS17, PS10}.

\vskip .2cm
  Let $r_i: \mathbb{R}^n\rightarrow \mathbb{R}^n$ be the $i$-th standard reflection defined by  $$r_i(x_1,\cdots,x_i,\cdots,x_n)=(x_1,\cdots,-x_i,\cdots,x_n).$$
 All standard reflections in $\mathbb{R}^n$ induce a standard $(\mathbb{Z}_2)^n$-action on the closed unit $n$-ball $B^n$ with a right-angled corner $B^n/(\mathbb{Z}_2)^n$ as its orbit space. Of course, $\text{Int}B^n$ is $(\mathbb{Z}_2)^n$-equivariantly homeomorphic to $\mathbb{R}^n$.

\begin{defn}[Right-angled Coxeter cells] Let $\Gamma$ be a group generated by some standard reflections in $\mathbb{R}^n$.
Then  the quotient $B^n/\Gamma$ is called a {\em right-angled Coxeter $n$-ball}, and the quotient $\text{Int}B^n/\Gamma$ is called an {\em open  right-angled Coxeter $n$-ball}.
Note that if $\Gamma$ is not a trivial group, then the right-angled Coxeter $n$-ball $B^n/\Gamma$ is an $n$-orbifold with boundary $\partial B^n/\Gamma$.

\vskip .2cm
If $e^n$ is $\Gamma$-equivariantly homeomorphic to $\text{Int}B^n$, then the quotient $e^n/\Gamma$ is called a {\em right-angled Coxeter $n$-cell}, and its closure is call a {\em closed right-angled Coxeter $n$-cell}.
 \end{defn}

  For example, a right-angled Coxeter $1$-cell is either a connected open interval  or a semi-open and semi-closed interval whose closed endpoint gives a local group $\mathbb{Z}_2$.
  A right-angled Coxeter  $2$-cell has three kinds of possible types with local group being trivial group, $\mathbb{Z}_2$ and $(\mathbb{Z}_2)^2$ respectively, as shown in \hyperref[right-angled Coxeter cell]{Figure 3}.

 \begin{figure}[h]\label{right-angled Coxeter cell}
\centering
\def\svgwidth{0.4\textwidth}
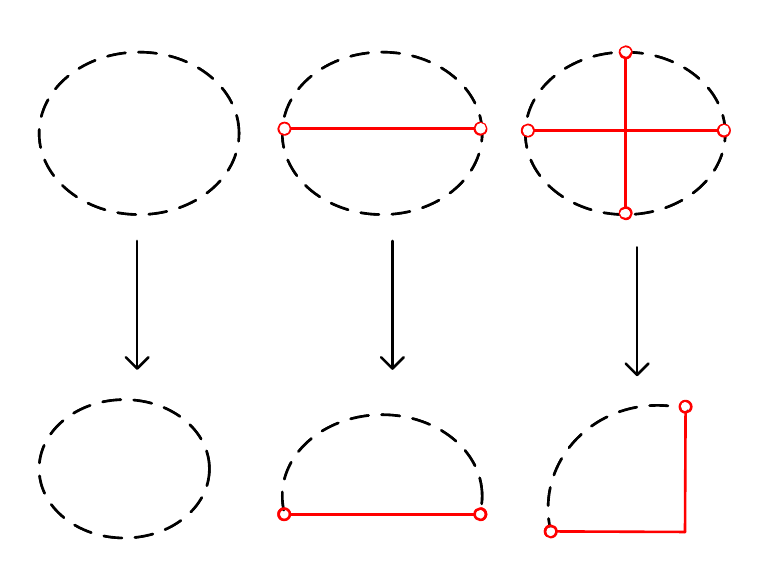
\caption{Right-angled Coxeter $2$-cells}
\end{figure}

\vskip .2cm

An $n$-dimensional {\em right-angled  Coxeter cellular complex}   (or {\em Coxeter CW complex})  $X$  can be constructed in the same way as CW complex (see \cite[Page 5]{H}).
A key point is that: The attaching map of every $n$-dimensional right-angled Coxeter cell $e_\alpha^n/\Gamma$ to the $(n-1)$-skeleton $X^{n-1}$
\begin{equation}\label{CC}
\phi_\alpha:\partial{\overline{e_\alpha^n}}/\Gamma\rightarrow X^{n-1}
\end{equation}
is required to {\bf preserve the local group of each point in $\partial{\overline{e_\alpha^n}}/\Gamma$}.
\vskip.2cm
Here  the attaching maps $\{\phi_{\alpha}\}$ of right-angled Coxeter cells with non-trivial local groups have a much stronger restriction  than those in CW complexes.
Actually, $\phi_{\alpha}$  preserving local groups implies that  singular points and non-singular points of each embedding right-angled Coxeter $n$-cell are  still singular  and non-singular in $X$, respectively. 

\begin{rem}[{ Right-angled Coxeter cubical cellular complex}]
Recall that a cubical cellular complex is a CW complex $X$ whose  cells are cubes, with
the property that for two cubes $c$, $c'$ of $X$, $c \cap c'$ is a common face of $c$ and
$c'$; in other words,  cubes are glued in $X$ via combinatorial isometries of their faces.
Similarly, a {\em right-angled Coxeter cubical cellular complex} can be defined in the same way  whose cells are all right-angled Coxeter cubical cells, that is, the orbits of standard reflections  on an $n$-cube $[-1,1]^n$. For example, the standard cubical cellular decomposition of  a simple polytope $P$ (i.e., the cone of the barycentric subdivision of $\mathcal{N}(P)$) gives a right-angled Coxeter cubical cellular  complex structure of $P$.
Of course, right-angled Coxeter cubical cellular complexes form a special class of right-angled Coxeter cellular complexes.
\end{rem}

\begin{prop}
Each simple handlebody has a finite right-angled Coxeter cellular complex structure.
\end{prop}
 \begin{proof} Let $Q$ be a    simple $n$-handlebody with the
  associated simple polytope $P_Q$. Then the standard cubical subdivision of $P_Q$ induces a right-angled Coxeter cellular decomposition of $Q$.  More details will be shown in \hyperref[section4]{Section 4}.
 \end{proof}
\begin{rem}
It should be pointed out that each  simple handlebody  still has a right-angled Coxeter cubical cellular complex structure. This can be seen in \hyperref[section-61]{Section 7.1}.
\end{rem}

In general, a right-angled Coxeter cellular complex is naturally an orbispace. Here its orbifold fundamental group is defined by the homotopy classes of based orbifold loops. For more details, see~\cite[Section 3]{C}.
Although a right-angled Coxeter cell with non-trivial local group is not contractible in the sense of orbifold,
all attaching maps $\{\phi_\alpha\}$ preserving local groups ensure that the orbifold fundamental group of a right-angled Coxeter cellular complex is isomorphic to the orbifold fundamental group of its $2$-skeleton.

\begin{prop}\label{p1}
Let $X$ be a right-angled Coxeter cellular complex. Then
$$
\pi^{\text{orb}}_1(X^2)\cong \pi^{\text{orb}}_1(X),
$$
where $X^2$ is the 2-skeleton of $X$.
\end{prop}
\begin{proof}
 The argument can be proved in a similar way as shown by Hatcher~\cite[Proposition 1.26]{H}. The only thing to note is that the local group information of each right-angled Coxeter $n$-cell can be inherited by the boundary orbifold of its closure in $X^{n-1}$.
\end{proof}

\begin{figure}[h]\label{generator-relation}
\centering
\def\svgwidth{0.55\textwidth}
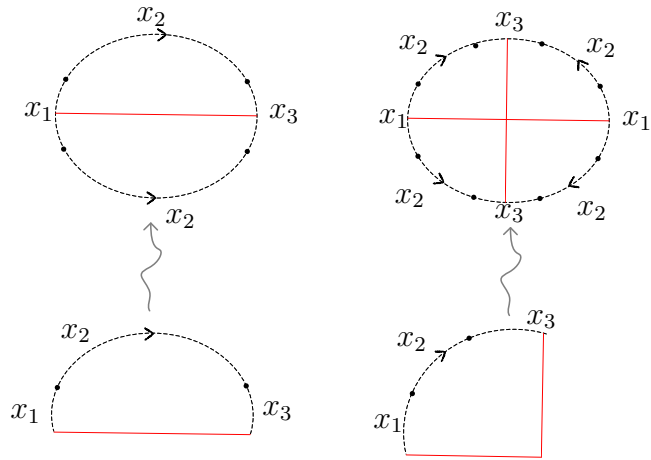
\caption{Relations determined by  right-angled Coxeter $2$-cells in the case $n=3$.}
\end{figure}
\begin{rem}
We can easily read out the generators and relations of $\pi^{\text{orb}}_1(X)\cong \pi^{\text{orb}}_1(X^2)$ from the $2$-skeleton of a right-angled Coxeter cellular complex $X$.  Let us look at  a right-angled Coxeter $2$-cell with non-trivial local group in $X$.
Assume that   the boundary of a right-angled Coxeter $2$-cell with non-trivial local group   consists of $x_1, x_2,\cdots, x_n$, where each $x_i$ is a closed oriented orbifold loop in $X$, and only one endpoint of $x_1$ and $x_n$ has non-trivial local group.
Regard these closed orbifold loops as generators. Then $x_1^2=x_n^2=1$. Moreover,
the right-angled  Coxeter $2$-cells with local group $\mathbb{Z}_2$ give a relation $x_1x_2\cdots x_n\cdot x^{-1}_{n-1}\cdots x_2^{-1}=1$, while the right-angled Coxeter $2$-cells with local group $\mathbb{Z}^2_2$ give a relation $(x_1x_2\cdots x_n\cdot x^{-1}_{n-1}\cdots x_2^{-1})^2=1$. This can intuitively be seen from \hyperref[generator-relation]{Figure 4} when $n=3$.
\end{rem}

 \begin{example}Let $P$ be a simple polytope with facet set $\mathcal{F}(P)$. Regard $P$ as a right-angled Coxeter orbifold. The standard cubical subdivision of $P$ is a right-angled Coxeter cellular decomposition of $P$.
 Calculating the orbifold fundamental group of  $P$ by  the $2$-skeleton of its right-angled Coxeter cellular  decomposition,   $\pi^{\text{orb}}_1(P)$ can be represented by the right-angled Coxeter group $W_P$ of $P$:
$$
\pi^{\text{orb}}_1(P)\cong W_P=\langle s_F, F\in \mathcal{F}(P) | s_F^2=1,~\text{for all }F;~ (s_Fs_{F'})^2=1, ~\text{for}~ F\cap F'\neq\emptyset\rangle$$
\end{example}

\subsection{Group action and fundamental domain (\cite[Page 64]{D2} or \cite[Page 159-161]{VS})}
  Suppose that a discrete group $G$ acts properly on a connected topological space $X$. A closed subset $D\subset X$ is a {\em fundamental domain} for the $G$-action on $X$ if each $G$-orbit intersects with $D$ and if for each point $x$ in the interior of $D$, $G(x)\cap D=\{x\}$. In other words, $\{gD | g\in G\}$ forms a locally finite cover for $X$, such that no two of $\{gD | g\in G\}$ have common interior points. Such $\{gD,g\in G\}$ is called a {\em decomposition} for $X$ so
   $$X=\bigcup_{g\in G} gD$$
    and each $gD$ is called a {\em chamber} of $G$ on $X$.
\vskip .2cm
Throughout the following,  the fundamental domain of $G$ acting on $X$ will be taken as a simple convex polytope $D$.
Then each $g\in G$ gives a self-homeomorphism of $X$
$$\phi_g:X\longrightarrow X$$
by mapping chamber $hD$ to $g\cdot hD$ for any $h\in G$.
If  two chambers $gD$ and $hD$ have a nonempty intersection which includes some facets of  $gD$ and $hD$, then there is a homeomorphism $\phi_{hg^{-1}}$ that maps $gD$ to $hD$.  Hence, for two facets $F$ and $F'$ from $gD$ and $hD$, respectively, that are glued together in $X$, naturally we can assign $hg^{-1}$ and $gh^{-1}$ to $F$ and $F'$, respectively.
This means that the action of $G$ on $X$ gives a {\em characteristic map} on the facets set of $D$:
$$\lambda: \mathcal{F}(D)\longrightarrow G.$$
For each facet $F$ of $D$, $\lambda(F)\in G$  is called a {\em coloring} on $F$. Each $\lambda(F)\in G$ naturally determines a self-homeomorphism $\phi_{\lambda(F)}\in \text{Homeo}(X)$, which is called an {\em adjacency transformation} on $X$ with respect to $F$.
Such $\phi_{\lambda(F)}$ maps each chamber into adjacent chamber such that the facet $F$  is contained in the intersection of those two chambers.
  Each adjacency transformation has an inverse adjacency transformation corresponding to  a facet
  $F'$ of $D$. Of course, $F=F'$ is allowed. In this case, we call $F$ a {\em mirror} of $X$ associated with $G$, and the corresponding adjacency transformation is called a {\em reflection} of $X$ with respect to $F$.
   \begin{rem}
It should be pointed out that  two adjacency transformations  determined by different facets of $D$ are viewed as being different, although they may correspond to the same self-homeomorphism of $X$. The inverse adjacency transformation of an adjacency transformation determined by a facet $F$ is exactly determined by another facet $F'$  which is identified with $F$ in $X$.
\end{rem}
 All inverse adjacency transformations give an equivalence relation $\sim$ on $\mathcal{F}(D)\times G$, where
$(F,g)\sim (F',h)$  if and only if
\begin{equation}\begin{cases}\label{C1}
 \lambda(F)\cdot g=h\\
 \lambda(F')\cdot h=g.\\
\end{cases}\end{equation}
In other words, if two chambers $gD$ and $hD$ are attached together by identifying a facet $F$ of $gD$ with a facet $F'$ of $hD$ in $X$, then $\lambda(F)\cdot \lambda(F')=1$, which  gives a  {\em pair relation} for $G$.
When $F$ is a mirror, the pair relation is $\lambda(F)^2=1$.

\begin{rem}
The equivalence relation $\sim$ on $\mathcal{F}(D)\times G$ gives an  equivalence relation $\sim'$ on $\mathcal{F}(D)$
via the projection $\mathcal{F}(D)\times G\longrightarrow \mathcal{F}(D)$ as follows: for $F, F'\in \mathcal{F}(D)$,
$$F\sim' F'\Longleftrightarrow (F,g)\sim (F',h) \text{ for } g, h\in G \text{ satisfying (\ref{C1})}.$$
Thus, we can obtain a quotient orbifold $D/\sim'$ by attaching some facets on the boundary of $D$ via the equivalence relation $\sim'$ on $\mathcal{F}(D)$.
\end{rem}

\vskip .2cm

On the contrary, giving a simple polytope $D$ and a characteristic map  satisfying (\ref{C1}),  we can construct a space $X$ with $G$-action in the following way:
\begin{equation}
X=D\times G/\sim
\end{equation}
where  the equivalence relation is defined in (\ref{C1}).

\vskip .2cm

The construction of $X$ gives a natural polyhedral cellular decomposition of $X$, denoted by $\mathcal{P}(X)$. The dual complex of  $\mathcal{P}(X)$ is denoted by $\mathcal{C}(X)$. If each codimension-$k$ face of $D$ in $X$ intersects with exactly $2^k$ chambers, then each cell of $\mathcal{C}(X)$ is a cube, which is exactly one induced by the standard cubical decomposition of the simple polytope $D$. Furthermore, if $\mathcal{C}(X)$ is a cubical complex, then the link of each vertex in $\mathcal{C}(X)$ is a simplicial complex which is exactly the boundary complex of the dual of $D$. The $1$-skeleton of $\mathcal{C}(X)$ is exactly the Cayley graph of $G$ with generator set consisting  adjacency transformations determined by all facets of $D$.
Therefore, one has that
\begin{lem}[{\cite[Page 160]{VS}}]
The group $G$ is generated by all adjacency transformations.
\end{lem}

 For simplifying notation, denote $\lambda(F_i)=s_i$ or $s_{F_i}$ for each $F_i\in \mathcal{F}(D)$.
Then for each $g\in G$, $\phi_g$ can be decomposed into the composition of some adjacency transformations:
$$g=s_{i_1}s_{i_2}\cdots s_{i_k}$$
The relations with form $s_{i_1}s_{i_2}\cdots s_{i_k}=1$ except pair relations is called {\em Poincar$\acute{\text{e}}$ relations}.
\begin{lem}[{\cite[Page 161]{VS}}]
The  Poincar$\acute{\text{e}}$ relations together with the pair relations form a set of relations of group $G$.
\end{lem}

For each codimension $2$ face of $D$, there is a Poincar$\acute{\text{e}}$ relation with form $s_ks_{k-1}\cdots s_1=1$ (alternatively,  $s'_1\cdots s'_k=1$,
where $s'_i=(s_i)^{-1}$ for each $i$).

\vskip .2cm

Define a group $G_D$ with generators consisting of all adjacency transformations determined by $\mathcal{F}(D)$ and relations  formed by all pair relations and Poincar$\acute{\text{e}}$ relations determined by all codimension $2$ faces in $D$.
\begin{equation}\label{E3}
\begin{split}
G_D=\langle s_i, \text{for } F_i\in \mathcal{F}(D) \mid & s_is_j=1,  \text{for } F_i\sim' F_j;\\
&s_{i_1}s_{i_2}\cdots s_{i_k}=1, \text{for each codim-$2$~ face~ in}~ D\rangle
\end{split}
\end{equation}

For the sake of preciseness,  suppose again that each codimension-$k$ face of $D$ in $X$ intersects with exactly $2^k$ chambers.  Then the cubical subdivision of $D$ induces a right-angled Coxeter cellular decomposition for the quotient orbifold $X/G$.
It is not difficult to see that $D/\sim'$ is isomorphic to $X/G$ as orbifolds.
 According to \autoref{p1},  $G_D$ is isomorphic to the orbifold fundamental group of the quotient space $X/G$.  Therefore,
 we have that
\begin{lem}\label{lmA3}
The orbifold fundamental group of $D/\sim'\cong X/G$   is isomorphic to $G_D$.
\end{lem}

There is a natural quotient map $\lambda_*: G_D\longrightarrow G$, and the image of $\lambda_*$ on each adjacency transformation $s_F$ is the coloring on corresponding facet $F$. Then the fundamental group of $X$ is isomorphic to the kernel of $\lambda_*$.
\begin{prop}\label{general}
Let $G$ be a discrete group which acts properly discontinuously on a   manifold $X$. Suppose $X$ is  decomposed into
$X=\bigcup_{g\in G} gD=D\times G/\sim$, where $D$ is a  simple convex polytope and each codimension-$k$ face of $D$ in $X$ intersects with exactly $2^k$ chambers. Let $G_D$ be the group defined as in (\ref{E3}), and $\lambda_*$ be the quotient map from $G_D$ to $G$ induced by the characteristic map $\lambda: \mathcal{F}(D)\longrightarrow G$. Then there is a short exact group sequence
$$1\longrightarrow \pi_1(X) \longrightarrow G_D\overset{\lambda_*}{\longrightarrow} G \longrightarrow 1$$
 which is induced by an orbifold covering $\pi:X \longrightarrow X/G$.
\end{prop}
\begin{proof} Refer to Chen (\cite[Page 40-49]{C}). Here it is only necessary to show that $G_D\cong \pi^{orb}_1(X/G)$, which is exactly \autoref{lmA3}.
\end{proof}

Given a simple convex polytope $D$ and a discrete group $G$, assume that there exists a characteristic  map $\lambda:\mathcal{F}(D)\longrightarrow G$ such that $X=D\times G/\sim$ is a  $G$-manifold, where $(F,g)\sim (F',h)$ for any $F, F'\in \mathcal{F}(D), g,h\in G$ if and only if  (\ref{C1}) holds.
Then, we have the following result.

\begin{cor}\label{cor3}
Under the assumption of \autoref{general},
 $X$ is simply-connected if and only if $G\cong G_D$.
\end{cor}
\begin{example}
Let $P$ be a square with faces $F_1,F_2,F_3,F_4$ colored by $e_1, e_2 ,e_1, e_2$ respectively, where $e_1, e_2$ are generators of $(\mathbb{Z}_2)^2$.
Then $X=P\times (\mathbb{Z}_2)^2/\sim\cong T^2$ is a small cover over $P$ (\cite{DJ}),
and $G_P=\langle s_1,s_2,s_3,s_4\mid s_i^2=1; (s_1s_2)^2=(s_2s_3)^2=(s_3s_4)^2=(s_4s_1)^2=1\rangle$ is the right-angled Coxeter group determined by $P$. Then
$\pi_1(X)\cong \ker \lambda_*=\mathbb{Z}^2$ is a normal subgroup in $G_P$ generated by Poincar\'e relations $s_1s_3$ and $s_2s_4$.
\end{example}

\subsection{Right-angled Coxeter group and HNN extension}

In this subsection, we refer to \cite[Chapter 3]{D2} and \cite[Chapter 4]{LS}.

\vskip.2cm
Let $w=s_1s_2\cdots s_m$ be a word in a right-angled Coxeter group $W=\langle ~S\ | R~\rangle$.
An {\em elementary operation} on $w$ is one of the following two types of operations:
\begin{itemize}
\item[(i)] Length-reducing: Delete a subword of $ss$;
\item[(ii)] Braid (commutation): Replace a subword of the form $st$ with $ts$, if $(st)^2=1$ in the relations set $R$ of $W$.
\end{itemize}
A word is {\em reduced} if it cannot be shorten by a sequence of elementary operations.

\begin{thm}[Tits~{\cite[Theorem 3.4.2]{D2}}]\label{Tits}
Two reduced words $x,y$ are the same in a right-angled Coxeter group  if and only if one of both $x$ and $y$ can be transformed into the other one by a sequence of elementary operations of type (ii).
\end{thm}

\begin{defn}[Higman-Neumann-Neumann Extension~{\cite[Page 179]{LS}}]
Let $G$ be a group with presentation $G=\langle S\ | R \rangle$, and let $\phi: A\longrightarrow B$ be an isomorphism between two subgroups of $G$. Let $t$ be a new symbol out of $S$. Then the HNN extension of $G$  relative to $\phi$ is defined as
$$G*_{\phi}=\langle S, t | R, t^{-1}gt=\phi(g),  ~g\in A\rangle.$$
\end{defn}

Let $\omega=g_0t^{\epsilon_1}g_1t^{\epsilon_2}\cdots g_{n-1}t^{\epsilon_n}g_n$ ($n\geq 0$)
be an expression in $G*_{\phi}$, where each $g_i$ is an element in $G$ (probably $g_i$ may be taken as the unit element $1$ in $G$), and $\epsilon_i$ is either number $1$ or $-1$.
Then $\omega$ is said to be {\em $t$-reduced} if there is no consecutive subword $t^{-1}g_it$ or  $tg_jt^{-1}$ with $g_i\in A$ and $g_j\in B$, respectively.

\vskip.2 cm

A {\em normal form} of an element in $G*_{\phi}$ is a word  $\omega=g_0t^{\epsilon_1}g_1t^{\epsilon_2}\cdots g_{n-1}t^{\epsilon_n}g_n\ (n\geq 0)$ where
\begin{itemize}
\item[(i)] $g_0$ is an arbitrary element of $G$;
\item[(ii)] If $\epsilon_i=-1$, then $g_i$ is a representative of a coset of $A$ in $G$;
\item[(iii)] If $\epsilon_i=+1$, then $g_i$ is a representative of a coset of $B$ in $G$;
\item[(iv)] There is no consecutive subword $t^{\epsilon}1t^{-\epsilon}$.
\end{itemize}

\begin{thm}[The Normal Form Theorem for HNN Extensions, {\cite[Theorem 2.1, Page 182]{LS}}]\label{TT2}
Let $
G*_{\phi}=\langle G,t \mid t^{-1}gt=\phi(g), g\in A\rangle$ be an HNN extension. Then there are two equivalent statements:
\begin{itemize}
\item[(I)] The group $G$ is embedded in $G*_\phi$ by the map $g\mapsto g$. If  $\omega=g_0t^{\epsilon_1}g_1\cdots t^{\epsilon_n}g_n=1$ in $G*_{\phi}$, then $\omega$ is not reduced;
\item[(II)] Every element $\omega$ of $G*_{\phi}$ has a unique representation  $\omega=g_0t^{\epsilon_1}g_1\cdots t^{\epsilon_n}g_n$ which is a normal form.
\end{itemize}
\end{thm}

 A {\em $t$-reduction} of $\omega=g_0t^{\epsilon_1}g_1\cdots t^{\epsilon_n}g_n$ is one of the following two operations.
\begin{itemize}
\item replace a subword of the form $t^{-1}gt$, where $g\in A$, by $\phi(g)$;
\item replace a subword of the form $tgt^{-1}$, where $g\in B$, by $\phi^{-1}(g)$.
\end{itemize}

A finite number of $t$-reductions  leads from  $\omega=g_0t^{\epsilon_1}g_1\cdots t^{\epsilon_n}g_n$ to a normal form.

\section{ Flagness and beltness of simple orbifolds and simple handlebodies}\label{section3}
\subsection{$B$-belts and flagness}
Assume that $Q$ is a  simple $n$-orbifold with nerve $\mathcal{N}(Q)$.  Denote $Q^*$ as the {\em dual} of $Q$,  whose facial structure is given by $\mathcal{N}(Q)$.


\begin{defn}[$B$-belts]\label{Def-B}
Let $i: B\hookrightarrow Q$ be  an embedding closed simple $k$-suborbifold whose underlying space is a $k$-ball.
We say that $i(B)$ is an {\em $B$-belt} of $Q$ if
\begin{itemize}
\item $i$ preserves codimensions, i.e., $i$ maps each codimension-$d$ face $f$ of $B$ to a codimension-$d$ face $F_f$ of $Q$;
\item The intersection $\cap f_\alpha=\emptyset$ for some facets $f_\alpha$ in $B$ if and only if either $\cap F_{f_\alpha}= \emptyset$  or
$\cup F_{f_\alpha}$ cannot deformatively retract onto $B$ in $|Q|$.
\end{itemize}
\end{defn}

\begin{rem}
The orbifold embedding $i: B\hookrightarrow Q$  preserving codimension is equivalent to that $i$ restricting on the local group of each point in $B$  induces an identity.
The statement that  $\cup F_{f_\alpha}$ cannot deformatively retract onto $B$ in $|Q|$ is equivalent to that there is at least a hole in the area surrounded by $\{F_{f_\alpha}\}$ and $B$.
\vskip.1cm
A simple polytope $P$ itself  is a belt.
For $B$-belt in a simple polytope $P$, the intersection $\cap f_\alpha=\emptyset$ for some facets $f_\alpha$ in $B$ if and only if  $\cap F_{f_\alpha}= \emptyset$. And each $B$-belt is {\em $\pi$-injective} in the sense of  \autoref{proposition-61}, which is an analogue of $\pi_1$-injective surface in a $3$-manifold.

\end{rem}

A $2$-dimensional $B$-belt in a simple $3$-handlebody $Q$ is a $k$-gon.  Traditionally, such a $B$-belt is also called a {\em $k$-belt} of $Q$.
In the case of dimension three, any simple $3$-polytope except tetrahedron has a $2$-dimensional $B$-belt.
\vskip.2cm

Next, we want to generalize the definition of flagness to simple orbifolds in terms of $B$-belts defined above. Recall that a simplicial complex $K$ with vertex set $V$  is a {\em flag} complex if every finite subset of $V$, which is pairwise joined by edges, spans a simplex.
Let $X$ be a cubical complex equipped with a   piecewise Euclidean structure. Then Gromov Lemma \cite{G} tells us that $X$ is non-positively curved if and only if the link of each vertex in  $X$ is a flag simplicial complex. Furthermore, by Cartan-Hadamard theorem, a non-positively curved space is aspherical.
\vskip.2cm
A simple polytope $P$ is {\em flag} if the boundary complex  of its dual is a flag simplicial complex. Let $M\rightarrow P$ be a small cover or a real moment angled manifold over $P$. Then we know from \cite[Theorem 2.2.5]{DJS2} that $M$ is aspherical  if and only if $P$ is flag. Equivalently, $P$, as a right-angled Coxeter orbifold, is orbifold-aspherical if and only if it is flag.

\vskip.2cm

Naturally, the flagness of  a simple orbifold $Q$ should still be closely related to orbifold-asphericity of $Q$. The right-angled Coxeter orbifold structure of $Q$ induces a facial structure of $Q$ which can be carried by the nerve $\mathcal{N}(Q)$ of $Q$ as a manifold with corners. The combinatorial  obstruction of orbifold-asphericity of $Q$ contains some quotient orbifolds of $S^k$ for $k\geq 2$ by reflective actions of $(\mathbb{Z}_2)^l$ for $1\leq l\leq k+1$. Moreover, since $Q$ is a simple orbifold, $S^k/(\mathbb{Z}_2)^{k+1}\cong \Delta^{k}$ is the unique possible combinatorial  obstruction of orbifold-asphericity of $Q$.
Notice that the orbifold-asphericity of $Q$ is determined by both $\mathcal{N}(Q)$ and $|Q|$.  Hence,  a natural conjecture arises as follows:

\vskip.2cm
\noindent {\bf Conjecture:} {\em A simple orbifold $Q$ is orbifold-aspherical if and only if  its underlying space $|Q|$ is aspherical as a manifold and $Q$ contains no  $\Delta^k$-belt for any $k\geq 2$.}
\begin{rem} Davis' results in \cite[Theorem 9.1.4]{D2} and \cite[Theorem 3.5]{D3} tell us that the conjecture is true  in the cases where $|Q|$ is acyclic or $Q$ has a corner structure defined in \cite[Section 3.1]{D3}. Our \hyperref[DJS-small]{Theorem A} also proves the case that $Q$ is a simple handlebody. All these  provide the support evidence of the conjecture.
\end{rem}



\begin{defn}\label{Def-T}
Let $Q$ be a simple orbifold.
If its underlying space $|Q|$ is aspherical, then $Q$  is said to be {\em flag}  if  it contains no  $\triangle^k$-belt  for any $k\geq 2$.
\end{defn}

A {\em simple handlebody} $Q$ is a simple polytope or  there exist finitely many disjoint  $B$-belts of codimension one, named
{\em cutting belts},  such that $Q$ can be cut open into a simple polytope $P_Q$ along those cutting belts.
Here the cutting operation is similar to a hierachy of Haken $3$-manifolds (or  $3$-orbifolds).
Generally a simple $3$-handlebody is not a Haken $3$-orbifold except that it is flag.
Refer to \cite{WF68,FR} for Haken $3$-manifolds and generalized Haken manifolds.

\vskip.2cm

Let $Q$ be  a simple handlebody. We see that some vertices $F_1, F_2, \cdots, F_k$ of $\mathcal{N}(Q)$ span a simplex $\triangle^{k-1}$ in $\mathcal{N}(Q)$ if and only if the associated vertices  span a simplex  in $\mathcal{N}(P_Q)$, and they span an {\em empty simplex} (that is, $\partial \triangle^{k-1}\subset \mathcal{N}(Q)$ but $\triangle^{k-1}$ itself is not in $\mathcal{N}(Q)$) whose interior is contained in the interior of $Q^*$ if and only if associated vertices  span an  empty simplex in $\mathcal{N}(P_Q)$. Specifically, those empty simplices correspond to some $\triangle^k$-belts in  $Q$.
 Hence, we have the following result.

\begin{lem}
A simple handlebody $Q$ is  flag if and only if the associated simple polytope $P_Q$ is flag (in other words,   $\mathcal{N}(P_Q)$ is a flag simplicial complex).
\end{lem}

\begin{rem}
Notice that a flag simple handlebody defined above may contain an empty simplex whose interior cannot be embedded in its dual $Q^*$,  as shown in ~\hyperref[f4]{Figure 5} for three pairwise intersected  faces $F_1,F_2,F_3$ in a flag simple solid torus. Therefore, the statement that $\mathcal{N}(Q)$ is a flag simplicial complex is not equivalent to that $Q$ is a flag simple handlebody.
\begin{figure}[h]\label{f4}
\centering
\def\svgwidth{0.45\textwidth}
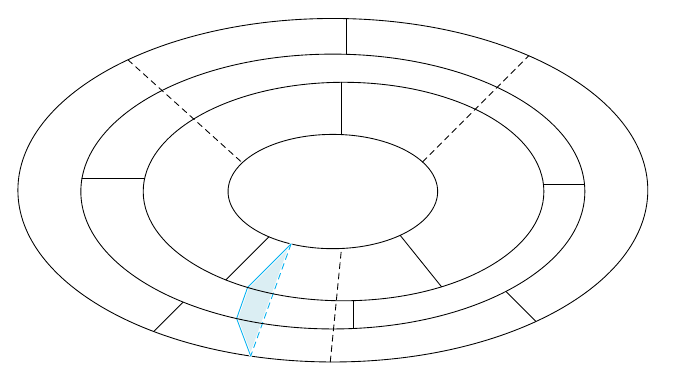
\caption{A flag simple 3-handlebody whose nerve is not a flag simplicial complex.}
\end{figure}
\end{rem}

\subsection{$\square$-belts in a simple handlebody}
\begin{defn}\label{Def-S}
Let $Q$ be a flag simple handlebody, and $Q^*$ be its dual. By  $\square$-belt we mean a quadrilateral-belt in $Q$.
The dual of an $\square$-belt in $Q$ is said to be an $\square$ in  $Q^*$.
\end{defn}

\begin{rem}\label{sphere}
(1) In  Gromov's paper~\cite[Section 4.2]{G},
{\em Siebenmann's no $\square$-condition} for a  flag simplicial complex $K$ means no {\em empty square}  in $K$, where an empty square in $K$ must make sure that neither pair of opposite vertices  is connected by an edge,   which is a special case in our definition.

\vskip .2cm



(2) A prismatic $3$-circuit~\cite{RHD} in  a  simple $3$-polytope $P^3$ determines an $\Delta^2$-belt in $P^3$.
 If there is no  prismatic $3$-circuit in $P^3$, then $P^3$ is a flag polytope or a tetrahedron.  Similarly for a prismatic $4$-circuit~\cite{RHD} in a flag simple polytope, it determines an $\square$-belt in $P$ in our definition.
\end{rem}

Next, we give two lemmas as the preliminary of the proof of \hyperref[FTT]{Theorem B}.
\vskip.2cm

Let $Q$ be a flag simple handlebody, and $B_{\square}$ be an $\square$-belt in $Q$ with four ordered edges  $f_1,f_2,f_3,f_4$, any two of which have a non-empty intersection except for pairs $\{f_1, f_3\}$ and $\{f_2,f_4\}$.  Assume that each $f_i$ is contained in a facet $F_{i}$ of $Q$. Then we may claim that $\{F_{i}\mid i=1,2,3,4\}$ must be different from each other. More precisely, we have the following lemma.
\begin{lem} \label{special-belt}
Let $Q$ be a flag simple handlebody, and $B_{\square}$ be an $\square$-belt in $Q$. Then,
\begin{itemize}
\item Two adjacent edges of $B_{\square}$ cannot be contained in the same facet of $Q$;
\item Two disjoint edges of $B_{\square}$ cannot be contained in the same facet of $Q$.
\end{itemize}
\end{lem}
\begin{proof}
Assume that the four edges $\{f_1,f_2,f_3,f_4\}$ of $B_{\square}$ are contained in four ordered facets $\{F_1,F_2,F_3,F_4\}$ of $Q$, respectively.
If there are two adjacent edges of $B_{\square}$  contained in the same facet of $Q$. Without loss of generality, suppose that $F_1=F_2$. Then $f_1\cap f_2\neq \emptyset$ implies that $F_1$ has a self-intersection, which is equivalent to that there is a $1$-simplex which bounds a single vertex in $\mathcal{N}(Q)$.
This contradicts that $Q$ is simple.

\vskip.2cm

Similarly,  if there are two disjoint edges of $B_{\square}$  contained in the same facet of $Q$, then one can assume that $F_1=F_3$. This  happens only for the case where the genus of $Q$ is more than zero since $B_\square$ is an $\square$-belt in $Q$.  Thus there are some holes between $F_1$ and $B_\square$. However, $F_2$ is contractible, so this induces that $F_2\cap F_1$ is disconnected. In other words, there are two $1$-simplices  which bound the same two vertices in $\mathcal{N}(Q)$. This is  also impossible since  $Q$ is simple.
\end{proof}

  \autoref{special-belt} tells us that in a  flag simple handlebody  $Q$, an $\square$-belt  can be presented as four different vertices $\{F_1,F_2,F_3,F_4\}$ in $ \mathcal{N}(Q)$, which satisfies the following two conditions:
\begin{itemize}
\item [(I)]
$\{F_1,F_2,F_3,F_4\}$ bounds a square with its interior located in the interior of $Q^*$ and with its edges contained in $1$-skeleton of $\mathcal{N}(Q)$;
   \item [(II)]     The full subcomplex spanned by $\{F_1,F_2,F_3, F_4\}$ in $\mathcal{N}(Q)$  is either a square  or a non-square subcomplex (containing two $2$-simplices gluing along an edge). Here the latter ``a non-square subcomplex" may happen only when the genus of $Q$ is more than zero.
\end{itemize}

   \begin{example}[$\square$s in the dual of a simple handlebody]\label{example-31}
   Let $Q$ be a simple handlebody, and $Q^*$ be its dual. There are some possible cases of $\square$s and non-$\square$s in $Q^*$, listed in \hyperref[F6]{Figure 6}, where all vertices and edges  are considered  in $\mathcal{N}(Q)$.
(a) and (b) are not  $\square$  in $Q^*$, while $(c)$ and $(d)$ are. Notice that  (d) is not an empty square in $\mathcal{N}(Q)$, which is different from the case of Siebenmann's no $\square$-condition, as stated in Remark~\ref{sphere} (1).
   \end{example}
\begin{figure}[h]\label{F6}
\centering
\def\svgwidth{0.60\textwidth}
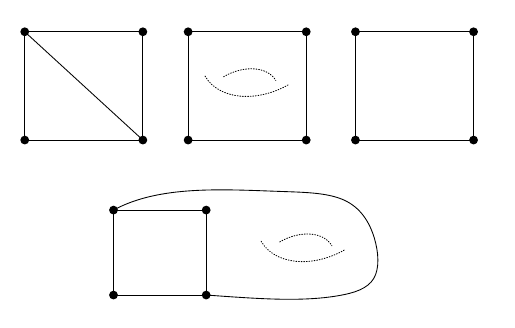
\caption{$\square$s and non-$\square$s}
\end{figure}

\begin{lem}\label{Lemma-34}
Let $B_\square$ be an $\square$-belt in a simple $n$-handlebody $Q$, and $B$ be a cutting belt of $Q$. Then
either $B_\square$ and $B$ can be separated in $Q$, or $B$ intersects transversely  with only a pair of disjoint edges of $B_\square$.
\end{lem}
\begin{proof}
Assume that the four ordered edges $f_1,f_2,f_3,f_4$ of $B_\square$ are contained in four facets $F_1,F_2,F_3,F_4$ of $Q$, respectively. Since $B_\square$ and $B$ are contractible, we see that $B$ and $B_\square$ can be separated if and only if their boundaries can be separated.

\vskip.2 cm
First we assume that $\partial B$ and $\partial B_\square$ intersect transversely, meaning that $\partial B\cap \partial B_\square$  is a set of isolated points  cyclically ordered on the boundary of $B_\square$, which is denoted by $\mathcal{V}$.  Then $\mathcal{V}$ contains at least two points if $\mathcal{V}$ is non-empty.

\vskip.2cm

Let $v$ and $v'$ be two adjacent points in $\mathcal{V}$.
Then there are the following cases:
\begin{itemize}
\item[(i)] $v$ and $v'$ are located in the same edge of $B_\square$;
\item[(ii)] $v$ and $v'$ are located in two adjacent edges of $B_\square$;
\item[(iii)] $v$ and $v'$ are located in two disjoint edges of $B_\square$.
\end{itemize}
\vskip .2cm
In the case (i), without loss of generality, suppose that $v,v'\in \text{int}(f_1)$. Now  if $v$ and $v'$ are contained in the same connected component of $F_1\cap B$ (without a loss of generality, assume that $B$ is regarded as $B_1$ of (a) in  \hyperref[F7]{Figure 7}), then we can deform the interior of $f_1$ such that $f_1\cap\partial B=\emptyset$ will not contain $v$ and $v'$.
 If $v$ and $v'$ are contained in two connected components of $F_1\cap B$, without loss of generality, assume that $B$ ia regarded as $B_2$ of (a) in \hyperref[F7]{Figure 7}. Since $B$ is an $B$-belt,  there is a hole surrounded by $f_1$ and $B$. This case is allowed (also see (b) and (c) in \hyperref[F7]{Figure 7}).

\begin{figure}[h]\label{F7}
\centering
\def\svgwidth{0.7\textwidth}
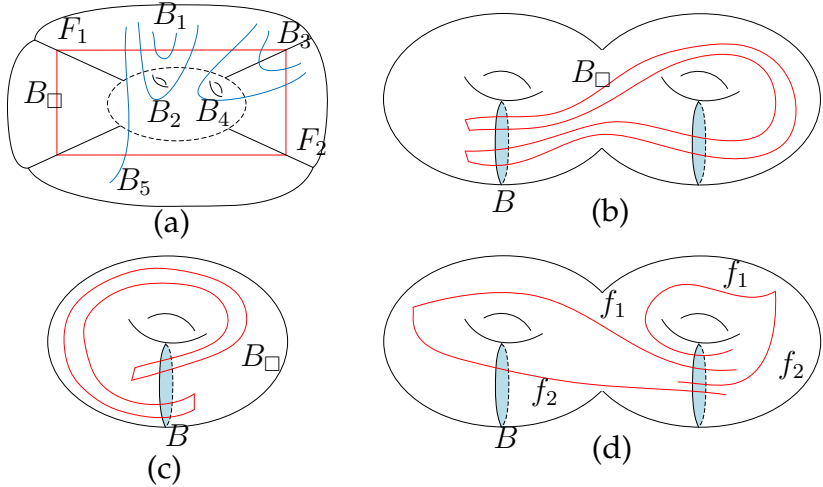
\caption{$\square$-belt and cutting belt}
\end{figure}

In the case (ii),  without loss of generality,  assume that $B$ intersects with $f_1$ and $f_2$.
 Now if $B\cap F_1\cap F_2\neq \emptyset$  (regard $B$ as $B_3$ of (a) in  \hyperref[F7]{Figure 7}),  then we can move vertex $f_1\cap f_2$ in $F_1\cap F_2$ such that $\partial B_\square\cap\partial B$ does not contain $v$ and $v'$.

 \vskip.2cm

Repeating this operation, we can assume that any two adjacent points $v$ and $v'$ in $\mathcal{V}$ cannot remove. This means that
  $B\cap F_1\cap F_2= \emptyset$ in the case (ii), so we may regard $B$ as $B_4$ of (a) in  \hyperref[F7]{Figure 7}. Then by the definition of $B$-belt, there is a hole in the area surrounded by $B, f_1,f_2$ (see (d)  in \hyperref[F7]{Figure 7}). If $|\mathcal{V}|=2$, then $B_\square$ will not be contractible. This is a contradiction. If $|\mathcal{V}|>2$,  let $v''$ be a point after $v'$ by the cyclic order of all isolated points in $\mathcal{V}$.
 If $v'$ and $v''$  belong to the same edge $f$ of $B_\square$, then there must be a hole surrounded by $f$ and $B$.
  If $v'$ and $v''$  belong to two adjacent edges $f'$ and $f''$ of $B_\square$, then there is also a hole surrounded by $f', f''$ and $B$.
  If $v'$ and $v''$  belong to two disjoint edges $f'$ and $f''$ of $B_\square$, then there  is still a hole surrounded by $f, f''$ and $B$, where $f$ is the edge containing $v$. Whichever of all possible cases above happens implies that  $\partial B_{\square}$ is not contractible in $|Q|$, but this is impossible.

\vskip.2cm

The case (iii)  is allowed, see $B_5$ of (a) in  \hyperref[F7]{Figure 7}. So the conclusion holds.
\end{proof}

 We see that if there are some cutting belts that intersect with $B_\square$, then one can do some deformations such that those cutting belts either do not intersect with $B_\square$ or intersect transversely with only a pair of disjoint edges of $B_\square$.

 \section{ The orbifold fundamental groups of  simple handlebodies}\label{section4}
\subsection{ The right-angled Coxeter cellular decomposition of simple handlebodies}
Let $Q$ be  a simple $n$-handlebody of genus $\mathfrak{g}$ with facet set $\mathcal{F}(Q)= \{F_1, ..., F_m\}$.
Then we can  cut $Q$ into  a simple polytope $P_Q$ along $\mathfrak{g}$  cutting belts $B_1, ..., B_\mathfrak{g}$, each  of which intersects transversely with some facets  of $Q$ and is a simple $(n-1)$-polytope.
Two copies of $B_i$ in $P_Q$, denoted by $B^+_i$ and $B^-_i$, respectively,  are two disjoint facets of $P_Q$.
Since they share the common belt $B_i$ in $Q$, by $B^+_i\sim B^-_i$ we denote this share between them.
   The number of facets of $P_Q$ around $B^+_i$ is the same as the number of  facets of $P_Q$ around $B^-_i$.
   In addition, each  facet $F$ of $P_Q$ around $B^+_i$  also uniquely corresponds to a    facet $F'$ of $P_Q$ around $B^-_i$
   such that $F$ and $F'$ share a common facet in $Q$, so by $F\cap B^+_i\sim F'\cap B^-_i$ we mean this share between $F$ and $F'$ via the belt $B_i$ of $Q$.

  \vskip .2cm
  Let $\mathcal{F}(P_Q)$ denote the set of all facets in $P_Q$ and $\mathcal{F}_B$ denote the set of those facets in $P_Q$, produced by cutting belts of $Q$, so $\mathcal{F}_B$ contains $2\mathfrak{g}$ facets of $P_Q$, appearing in pairs.
  \vskip .2cm
 $P_Q$ is viewed as a right-angled Coxeter orbifold with boundary consisting of  all facets in $\mathcal{F}_B$.
 By attaching all  pairs $B^+\sim B^-$  in $\mathcal{F}_B$ and  all corresponding pairs $(F, F')$  with $F\cap B^+\sim F'\cap B^-$ together, we can recover $Q$ from $P_Q$. Thus $Q$ can be regarded as a
 quotient $P_Q/\sim$, and we denote the quotient map by
 \begin{equation}\label{eq-31}q:P_Q\longrightarrow Q.\end{equation}

There is a canonical right-angled Coxeter cubical cellular decomposition $\mathcal{C}(P_Q)$ of  $P_Q$, whose cells  consist of
\begin{itemize}
\item all cubes in the standard cubical decomposition of  $P_Q$;
\item  all cubes in the standard cubical decomposition of  all boundary components  of $P_Q$ in $\mathcal{F}_B$.
 \end{itemize}
Moreover, $\mathcal{C}(P_Q)$ induces a right-angled Coxeter cellular decomposition on $Q$ by attaching some cubical cells of the copies of $B$-belts.
 Let $c$ be a $k$-cube in $\mathcal{C}(P_Q)$ and $B\in \mathcal{F}_B$.
 \begin{itemize}
\item If  $c\cap B=\emptyset$, then we may take $c$ as a right-angled Coxeter cubical  cell for $Q$. Such $c$  corresponds to a codimension $k$ face in $P_{Q}$ which is determined by $k$ facets in $\mathcal{F}(P_Q)-\mathcal{F}_B$,
 so $c$ is of the form $e^k/(\mathbb{Z}_2)^k$.

\item If $c$ is a $k$-cube in $\mathcal{C}(B^+)\subset \mathcal{C}(P_Q)$, then there is also another $k$-cube $c'\in \mathcal{C}(B^-)\subset \mathcal{C}(P_Q)$. Both $c$ and $c'$ are codimension-one faces of two  $(k+1)$-cubes in $\mathcal{C}(P_Q)$, respectively. Gluing those two  $(k+1)$-cubes by identifying $c$ with $c'$, we  obtain a right-angled Coxeter cubical  cell with form $e^{k+1}/(\mathbb{Z}_2)^k$.
\end{itemize}

\vskip .2cm

 Finally, we obtain  a right-angled Coxeter cellular decomposition of $Q$, denoted by $\mathcal{C}(Q)$,  whose cells are right-angled cubes.  Of particular note is that $\mathcal{C}(Q)$ is not cubical. This is  because there exists the cubical cell  glued by two cells $c$ and $c'$ in $\mathcal{C}(P_Q)$ as above, which has a self-intersection, namely the cone point $x_0$, as shown in \hyperref[cone point]{Figure 8}. The cone point is the only $0$-cell in $\mathcal{C}(Q)$, which will be chosen as the basepoint when we calculate the orbifold fundamental group $\pi^{\text{orb}}_1(Q)$ of $Q$.
\begin{figure}[h]\label{cone point}
\centering
\def\svgwidth{0.6\textwidth}
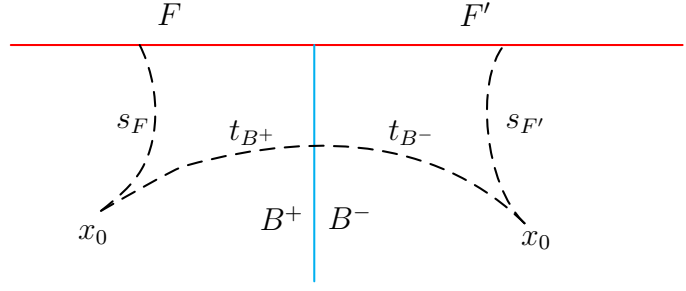
\caption{The right-angled Coxeter $2$-cell nearby $B$-belt.}
\end{figure}

\subsection{ The orbifold fundamental groups of  simple handlebodies}
Following the above notations, by \autoref{p1}, we can directly write out a presentation of orbifold fundamental group of $Q$.
\begin{prop}\label{Lemma-31}
Let $Q$ be a simple handlebody of genus $\mathfrak{g}$, and $P_Q$ be the associated simple polytope with copies of cutting belts $\mathcal{F}_B$.  Then  $\pi_1^{\text{orb}}(Q)$ has a presentation with
generators $s_{F}$ indexed by $F\in\mathcal{F}(P_Q)$, satisfying the following relations:
  \begin{itemize}
\item[$(1)$] $s_F^2=1$ for $F\in \mathcal{F}(P_Q)- \mathcal{F}_B$;
\item[$(2)$] $t_{B^+}t_{B^-}=1$ for two $B^+$ and $B^-$ with $B^+\sim B^-$ in $\mathcal{F}_B$;
\item[$(3)$] $(s_Fs_{F'})^2=1$  for $F,F'\in \mathcal{F}(P_Q)-\mathcal{F}_B$ with $F\cap F'\neq \emptyset$;
 \item[$(4)$]
 $s_{F}t_{B^+}=t_{B^+}s_{F'}$  for $B^+\sim B^-$ in $\mathcal{F}_B$ and $F,F'\in \mathcal{F}(P_Q)-\mathcal{F}_B$ with $F\cap B^+\sim F'\cap B^-$
  \end{itemize}
 where the basepoint of $\pi^{\text{orb}}_1(Q)$ is the cone point $x_0$ in the interior of $Q$.
\end{prop}

 On the other hand,  we  show here that $\pi^{\text{orb}}_1(Q)$ is actually an iterative  HNN-extension on $W(P_Q, \mathcal{F}_B)$, where $W(P_Q, \mathcal{F}_B)$
 is a right-angled Coxeter group determined by facial structure of $P_Q$ by ignoring the facets in $\mathcal{F}_B$:
  \begin{equation*}\label{W-Group}
 \begin{split}
W(P_Q, \mathcal{F}_B)  = \langle s_{F}, &\forall F\in\mathcal{F}(P_Q)-\mathcal{F}_B\mid
 s_F^2=1, \forall F\in \mathcal{F}(P_Q)- \mathcal{F}_B; \\
 &(s_Fs_{F'})^2=1,  \forall F,F'\in \mathcal{F}(P_Q)-\mathcal{F}(B), F\cap F'\neq \emptyset
 \rangle
  \end{split}
 \end{equation*}
 which can be regarded as the orbifold fundamental group $\pi^{\text{orb}}_1(P_Q)$ of $P_Q$ as a right-angled Coxeter orbifold with boundary consisting of the disjoint union of all facets in $\mathcal{F}_B$.

 \vskip .2cm

Let $B$ be a cutting belt in $Q$, and $B^+,B^{-}\in\mathcal{F}_B $ be two copies of $B$. Set $\mathcal{F}^{B^+}=\{F\in \mathcal{F}(P_Q)-\mathcal{F}_B | F\cap B^{+}\neq \emptyset \}$, and $\mathcal{F}^{B^{-}}=\{F\in \mathcal{F}(P_Q)-\mathcal{F}_B| F\cap B^{-}\neq \emptyset \}$. The associated right-angled Coxeter group $W_{B^+}$ and $W_{B^-}$ are two isomorphic groups since
$B^+$ and $B^-$ are homeomorphic as simple polytopes.

\begin{lem}\label{lemma-34}
The maps  $i_{B^+}: W_{B^+}\rightarrow W(P_Q, \mathcal{F}_B)$ and $i_{B^{-}}: W_{B^{-}}\rightarrow W(P_Q, \mathcal{F}_B)$  induced by inclusions $B^+\hookrightarrow P_Q$ and $B^-\hookrightarrow P_Q$ are monomorphisms.
\end{lem}
 \begin{proof}
 According to the definition of $B$-belt, $i_{B^+}$ and $i_{B^-}$ are obviously well-defined.  There are two group homomorphisms $j_{B^+}: W(P_Q, \mathcal{F}_B) \rightarrow W_{B^+}$ and $j_{B^-}: W(P_Q, \mathcal{F}_B) \rightarrow  W_{B^-}$, defined by reducing modulo the normal subgroups generated by facets being not in $\mathcal{F}^{B^+}$ and $\mathcal{F}^{B^-}$, such that $j_{B^+}\circ i_{B^+}=id_{W_{B^+}}$ and $j_{B^-}\circ i_{B^-}=id_{W_{B^-}}$.  The result follows from this.
\end{proof}

Thus, $W_{B^+}$ and $W_{B^-}$  can also be regarded as two isomorphic subgroups of $W(P_Q, \mathcal{F}_B)$ generated by $s_F, F\in\mathcal{F}^{B^+}$ and $s_{F'}, F'\in \mathcal{F}^{B^-}$, respectively.  Define $\phi_B:  W_{B^-}\longrightarrow W_{B^+}$ by $\phi_B(s_{F'})=s_F$ with $F'\cap B^-\sim F\cap B^+$. Then $\phi_B$ is a well-defined isomorphism.
 Furthermore, attaching  two facets on $P_Q$ corresponding to the cutting belt $B$ is equivalent to doing once HNN-extension on its orbifold fundamental group $\pi_1^{orb}(P_Q)=W(P_Q,\mathcal{F}_B)$, giving new elements $t_{B^+}, t_{B^-}$ with certain conditions in $\pi^{\text{orb}}_1(Q)$.
 By doing an induction on the genus of $Q$ and repeating the use of the normal form theorem of HNN-extension (\autoref{TT2}),
 the orbifold fundamental group of $Q$ is isomorphic to  $\mathfrak{g}$ times HNN-extensions on the right-angled Coxeter group $W(P_Q, \mathcal{F}_B)$, as shown below:
$$\xymatrix@=0.4cm{
  (Q_{\mathfrak{g}},B_{\mathfrak{g}}) \ar[rr] & & \cdots \ar[rr] && (Q_1,B_1) \ar[rr] &&Q_0=P_Q\\
  &&& \ar[rr]^{\text{Cutting}} && &&&\\
 &&& & &\ar[ll]_{\text{HNN-extension}}&&&\\
G_\mathfrak{g}=\pi^{\text{orb}}_1(Q)&& \ar[ll]\cdots&& \ar[ll] G_1 &&\ar[ll]  G_0=W(P_Q,\mathcal{F}_B)\\
  }$$
where each $Q_k$ is the simple handlebody of  genus $k$ obtained from $Q_{k+1}$ by cutting open along the $(k+1)$-th belt $B_{k+1}$, which is a right-angled Coxeter orbifold with boundary consisting of double copies of $\{B_{k+1}, \cdots, B_\mathfrak{g}\}$, and each $G_k$ is the orbifold fundamental group of $Q_{k}$ which is obtained from an HNN extension on $G_{k-1}$.

\begin{prop}\label{HNN}
Let $Q$ be a simple handlebody of genus $\mathfrak{g}$ with cutting belts $B_1, ..., B_\mathfrak{g}$. Then
$\pi^{\text{orb}}_1(Q)\cong (\cdots((W(P_Q,\mathcal{F}_B){*}_{\phi_{B_1}}){*}_{\phi_{B_2}})\cdots){*}_{\phi_{B_\mathfrak{g}}}$.
 \end{prop}

 Notice that the expression $(\cdots((W(P_Q, \mathcal{F}_B){*}_{\phi_{B_1}}){*}_{\phi_{B_2}})\cdots){*}_{\phi_{B_\mathfrak{g}}}$ in \autoref{HNN} is independent of orders of $\phi_{B_i}$. In addition, the  presentation of  $\pi^{orb}_1(Q)$ in \autoref{Lemma-31} can be simplified  by deleting all generators $t_{B^{-}}$ and relations $t_{B^{+}}t_{B^{-}}=1$, meanwhile, replaced by only all $t_B$ . Here the  group  $\pi^{orb}_1(Q)$  is called a {\em handlebody group}.
 It should be pointed out that the right-angled Coxeter group $W_Q$ determined by facial structure of  $Q$ is not a subgroup of $\pi^{\text{orb}}_1(Q)$ in general.
 Actually, $W_Q$ is the quotient group of $\pi^{\text{orb}}_1(Q)$ with respect to the normal group generated by all $t_{B}$.

  \begin{rem}
In \cite[Theorem 4.7.2]{DJS}, Davis, Januszkiewicz and Scott give a similar form. However, all generators in their paper lifted into the universal space as  homeomorphisms onto itself are involutions, i.e.,  $t_B^2=1$. Here,  with a little difference,  we require that the lifted action of $t_B$ is free.
In particular, the last relation in \autoref{Lemma-31}
 belongs to a kind of {\em Baumslag-Solitar relations}, which are related to the HNN-extension; in other words,  pasting pairs of  facets corresponding to cutting belts of the polytope $P_Q$ can be viewed as a topological explanation for the HNN-extension of their orbifold fundamental groups.
More precisely, for a cutting belt $B$, there are two copies $B^+$ and $B^-$ in $P_Q$, and the composite  map
$$W_B\cong W_{B^+}\overset{i_{B^+}}{\longrightarrow} W(P_Q,\mathcal{F}_B)\overset{i_{1}}{\longrightarrow} G_1\overset{i_{2}}{\longrightarrow}\cdots\overset{i_{g}}{\longrightarrow} G_\mathfrak{g}=\pi^{\text{orb}}_1(Q)$$
 embeds $W_{B}$ into $\pi^{\text{orb}}_1(Q)$, where $i_{k}$ is defined by $i_{k}(h)=h\in G_k$ for $h\in G_{k-1}$. Both $W_{B^{+}}$ and $W_{B^-}$ are linked in $\pi^{\text{orb}}_1(Q)$ by an isomorphism and the injectivity of $i_{k}$ is followed by the normal form theorem of HNN-extension (\autoref{TT2}).
 \end{rem}

 \subsection{The orbifold universal covers of simple handlebodies}
  Let $Q$ be a simple handlebod with cutting belts $\{B_1,\cdots,B_{\mathfrak{g}}\}$, and $P_Q$ be the simple polytope given by cutting $Q$ with the quotient map $q: P_Q\longrightarrow Q$. Let $\pi^{orb}_1(Q)$ be the orbifold fundamental group with the presentation in \autoref{Lemma-31}.
 Define a {\em characteristic map} on the facet set of $P_Q$:
 $$\lambda: \mathcal{F}(P_Q)\longrightarrow \pi^{orb}_1(Q)$$
 given by $\lambda(F)=s_F$ for $F\in \mathcal{F}(P_Q)-\mathcal{F}_B$, and $\lambda(B)=t_B$ for $B\in\mathcal{F}_B$.
 Then we construct the following space
 \begin{equation}\label{uni-cover}
\widetilde{Q}=P_Q\times \pi^{orb}_1(Q)/\sim
 \end{equation}
where $(x,g)\sim(y,h)$ if and only if
\begin{equation*}\begin{cases}\label{er}
x=y\in F\in\mathcal{F}(P_Q)-\mathcal{F}_B, gs_F=h,\\
(x,y)\in (B,B^{'}), B,B'\in \mathcal{F}_B, q(x)=q(y), t_B\cdot g=h.
 \end{cases}\end{equation*}

\vskip.2cm
  The orbit space of the action of $\pi^{orb}_1(Q)$ on $\widetilde{Q}$ is $Q$, so the polytope $P_Q$ can be viewed as the fundamental domain of $\pi^{orb}_1(Q)$ acting on $\widetilde{Q}$.
According to \autoref{cor3},
 \begin{lem}\label{Lemma-A4}
 $\widetilde{Q}$ is the universal orbifold cover of  $Q$.
 \end{lem}

\section{Proof of \hyperref[DJS-small]{Theorem A}}\label{section5}

This section is devoted to giving the  proof of \hyperref[DJS-small]{Theorem A}.

\subsection{Non-positively curved cubical complex}
A geodesic metric space $X$ is  {\em non-positively curved} if it is a locally CAT(0) space.
The Cartan-Hadamard theorem implies that non-positively curved spaces are aspherical. Cf~\cite{BH, D2, G}.

\begin{defn}[The links in  a cubical complex  {\cite[Subsection 7.15]{BH} or \cite[Page 508]{D2}}]
  Let $K$ be a cubical complex. For each vertex $v\in K$, its {\em (geometric) link}, denoted by Lk$(v)$, is a  simplicial complex defined by
  all cubes in $K$ that properly contains $v$ with respect to the inclusion.
 A $d$-cube $c$ of $K$ that properly contains $v$  determines a $(d-1)$-simplex $s(c)$ in Lk$(v)$.
\end{defn}

\begin{prop}[Gromov Lemma, also see {\cite[Corollary I.6.3]{D2}}]\label{LMG}
A piecewise Euclidean cubical complex is non-positively curved if and only if the link of its each vertex is a flag complex.
\end{prop}

\subsection{Homology groups of  manifold covers  over simple handlebodies}\label{App-1}
Let $M_Q$ and $\widetilde{Q}$ be the manifold double and orbifold universal cover over a simple handlebody $Q$ with $m$ facets, respectively. In this subsection, we discuss the homology groups of $M_Q$ and $\widetilde{Q}$.

\subsubsection{Homology groups of $M_Q$}
By Davis' \cite[Theorem 8.12]{D2},
\begin{equation}\label{hh}
H_*(M_Q)\cong \bigoplus_{g\in (\mathbb{Z}_2)^m}H_*(|Q|,\mathcal{F}_{g})
\end{equation}
where $\mathcal{F}_{g}=\bigcup_{s_i\in S(g)} F_i\subset \partial |Q|$, $F_i\in \mathcal{F}(Q)$ and $S(g)=\{s_i| l(s_i\cdot g)=l(g)-1\}$ for   a reduced word $g$ of length $l(g)$ in $(\mathbb{Z}_2)^m$.
If $Q$ is a simple handlebody of genus $\mathfrak{g}\geq 0$, then $|Q|\simeq \bigvee_{\mathfrak{g}} S^1 $. By the long exact sequence of homology groups of $(|Q|,\mathcal{F}_{g})$, if $*\geq3$ 
then
$$
H_*(|Q|,\mathcal{F}_{g})\cong H_{*-1}(\mathcal{F}_{g})\cong H_{*-1}(K_{g})
$$
where $K_g\simeq \mathcal{F}_{g}$ is the dual simplicial complex of $\mathcal{F}_{g}$, which is a subcomplex of $\mathcal{N}(Q)$. Hence for $*\geq3$ 
$$
H_*(M_Q)\cong \bigoplus_{J\subset\mathcal{F}(Q)}H_{*-1}(K_{J}).
$$
where $J$ is the set of those facets $F_i$ corresponding to all $s_i\in S(g)$.
\vskip .2cm

 For $*=1,2$, we have
 $$
 0\rightarrow H_2(|Q|,\mathcal{F}_{g})\rightarrow H_{1}(\mathcal{F}_{g})\overset{(i_g)_*}{\longrightarrow} H_{1}(|Q|)\cong \mathbb{Z}^{\mathfrak{g}}\rightarrow H_1(|Q|,\mathcal{F}_{g})\rightarrow 0.
 $$
 where $i_g: \mathcal{F}_g\longrightarrow |Q|$ is an inclusion.
 Then
  $$
H_1(M_Q)\cong \bigoplus_{g\in (\mathbb{Z}_2)^m} \text{coker} (i_g)_*
$$
and
 $$
H_2(M_Q)\cong \bigoplus_{g\in (\mathbb{Z}_2)^m} \ker (i_g)_*.
$$
\begin{rem}
The formula~(\ref{hh}) is actually the Hochster's formula in the setting of  simple handlebodies. When $Q$ is a simple polytope $P$, the Hochster's formula~\cite[Proposition 3.2.11]{TT} can also be expressed as
$$
H^{l(g)-i-1}(\mathcal{F}_g)\cong \text{Tor}_{\mathbb{Z}[v_1,\cdots,v_m]}^{-i,l(g)}(\mathcal{K}(P),\mathbb{Z})
$$
where $\mathcal{K}(P)$ is the Stanley-Reisner face ring of $P$.

\end{rem}
\subsubsection{Homology groups of $\widetilde{Q}$}
  Davis' \cite[Theorem 8.12]{D2} can not be directly applied to give the homology groups of $\widetilde{Q}$ since $\pi_1^{orb}(Q)$ is not a Coxeter group when the genus of $Q$ is more than zero.
  However, we can employee the method of Davis in \cite[Chapter 8]{D2} to calculate of the homology groups of $\widetilde{Q}$.

 \vskip .2cm

 Let $Q$ be a simple handlebody with nerve $\mathcal{N}(Q)$, and $P_Q$ be the associated simple polytope.
 Let $G=\pi^{orb}_1(Q)$ be the orbifold fundamental group of $Q$.
 We have known that $G$ is an iterative HNN-extension on a right-angled Coxeter group $W(P_Q,\mathcal{F}_B)$. Namely
$$G=\pi^{\text{orb}}_1(Q)\cong (\cdots((W(P_Q,\mathcal{F}_B){*}_{\phi_{B_1}}){*}_{\phi_{B_2}})\cdots){*}_{\phi_{B_\mathfrak{g}}}$$
where $\mathfrak{g}$ is the genus of $Q$.
For any $w\in G$, consider the following reduced normal form,
\begin{equation*}\label{Reduced-x}
w=g_0t_{1}g_1\cdots g_{m-1}t_{m}g_m
 \end{equation*}
 where each $g_i$ is reduced in $W(P_Q,\mathcal{F}_B)$, and each $t_i$ is one of $\{t^{\pm 1}_B\}$ which determines an isomorphism of $\{\phi^{\pm 1}_B\}$ on some subgroups of $\pi_1^{orb}(Q)$.
 Denote the generator set of $G$ by $\mathcal{S}=\{s_F; F\in \mathcal{F}(P_Q)-\mathcal{F}_B\}\cup \{t_B; B\in \mathcal{F}_B\}$.  For any word $w\in G$,
put
 $$
S(w)=\{s\in \mathcal{S}\mid l(ws)<l(w)\},
$$
where $l(w)$ is  the word length of the reduced normal form of $w$ in $G$ (i.e.,  the shortest length between $1$ and $w$ in the Cayley graph of $G$ associated with the generator set $\mathcal{S}$).
For each subset $T$ of $\mathcal{S}$, let $P_Q^T$ be the subcomplex of $P_Q$  defined by
$$
P_Q^T=\bigcup_{t\in T}F_t,
$$
where $F_{s_F}=F$ for $s_F\in \mathcal{F}(P_Q)-\mathcal{F}_B$ and $F_{t_{B}}=B'$ for $B\in \mathcal{F}_B$ with $B\sim B'$.

\vskip.2cm
 Let  $\widetilde{Q}=P_Q\times G/\sim$ be the universal cover of $Q$ defined as in (\ref{uni-cover}).  Then we have the following conclusion which generalizes Davis' theorem in \cite[Theorem 8.12]{D2}.
\begin{prop}\label{thm4}
The homology of $\widetilde{Q}$ is isomorphic to the following direct sum
$$
H_*(\widetilde{Q})\cong \bigoplus_{w\in G}H_*(P_Q,P_Q^{S(w)}).
$$
\end{prop}
\begin{rem}
It should be emphasized that here $P_Q$ is not used as  a mirrored space in the sense of Davis in~\cite{D2} although is is a simple polytope.
Actually, here we just put $P_Q$ and $\pi_1^{orb}(Q)$ together to construct the orbifold universal cover $\widetilde{Q}$, but
$\pi_1^{orb}(Q)$ is not a Coxeter group except that the genus of $Q$ is zero.
\end{rem}
\begin{cor}\label{nonempty}
If there is an empty $k$-simplex $\triangle^k$ in $\mathcal{N}(P_Q)$, then $H_k(\widetilde{Q})\neq 0$.
\end{cor}
\begin{proof}
 Assume that the vertex set of the empty $k$-simplex $\triangle^k$ in $\mathcal{N}(P_Q)$ is $$T=\{F_1,F_2,\cdots,F_{k+1}\}$$ which does not contain the facet in $\mathcal{F}_B$ (in fact, any facet in $\mathcal{F}_B$ is not the vertex of any empty simplex of $\mathcal{N}(P_Q)$. This is guaranteed by the definition of $B$-belt). Then $W_T\cong(\mathbb{Z}_2)^{k+1}$ is generated by $s_1,\cdots,s_{k+1}$. Let $w=s_1s_2\cdots s_{k+1}$. Regard $T$ as $\{s_1, ..., s_{k+1}\}$. Then $S(w)=T$. Moreover, $P_Q^{S(w)}=P_Q^T=\cup_{i=1}^{k+1} F_i\backsimeq \partial \triangle^k \backsimeq S^{k-1}$. Since $P_Q$ is a contractible ball,  by the long exact homology group sequence of pair $(P_Q,P_Q^T)$, we have
$$H_k(P_Q, P_Q^T)\cong H_{k-1}(P_Q^T)\cong H_{k-1}(S^{k-1})\neq 0.$$
Therefore, $H_k(\widetilde{Q})\neq 0$.
\end{proof}


\subsubsection{Proof of \autoref{thm4}}
Before we prove \autoref{thm4}, we first give some notations (cf~\cite{D2}).

\vskip.2cm

A subset $T$ of $\mathcal{S}$ is called {\em spherical} if the subgroup generated by $T$ is a finite subgroup of $G$. Each $s_F$ in a spherical subset $T$ exactly corresponds to a facet $F\in \mathcal{F}(P)-\mathcal{F}_B$, and $F\cap F'\neq\emptyset$ for any $s_F,s_{F'}$ in spherical set $T$. Let $W_T$ be the  group generated by a spherical subset $T$. Then $W_T\cong (\mathbb{Z}_2)^{\# T}$, where $\# T$ denotes the number of all elements in $T$.

\vskip.2cm

If the set $T$ is the union of a spherical set $ T_S$ and a $t_{B}$ for $B\in \mathcal{F}_B$, then
$$W_T=W_{T_S}\cup t_{B'}W_{T_S},$$
 where  $B'$ is the facet which is identified with $B$ in $Q$.

\begin{lem}
Let $G$ be the orbifold fundamental group of a simple handlebody with generator set $\mathcal{S}$.
Then, for each $w\in G$, $S(w)$ is either a spherical subset of $\mathcal{S}$ or the union of a $t_B$ and  a spherical subset in $\mathcal{S}$.
\end{lem}
\begin{proof}

Let $w=g_0t_{1}g_1\cdots g_{m-1}t_{m}g_m$ be a reduced normal form in $G$.
We might as well assume that this expression of  $w$ is a normal form  in  the opposite direction for  each $t_{B}$, that is, each $g_i$ is a representative of a coset of $W_{B_{i+1}}$ or $W_{B'_{i+1}}$ in $G$, for $i=0,\cdots,m-1$.

\vskip.2cm
It is easy to see that  for $F\in \mathcal{F}(P)-\mathcal{F}_B$, $s_F\in S(w)$ if and only if $s_F\in S(t_mg_{m})$.
If there is a $B\in \mathcal{F}_B$ such that $t_B\in S(w)$, then $g_{m}t_B=t_Bg'_{m}$ where $g'_{m}=\phi_B(g_{m})$, and  the last $t_m$ is $t_B^{-1}$. For another $t_{B'}\neq t_{B}$, it cannot reduce the length of $w$. Thus the conclusion holds.
\end{proof}

\vskip.2cm

For a spherical set $T=S(w)$, we define an element in  $\mathbb{Z}W_{T}\subset \mathbb{Z}W(P_Q,\mathcal{F}_B)$ by a formula
$$
\beta_T=\sum_{w\in W_T}(-1)^{l(w)}w.
$$
Consider a natural cellular decomposition of $P_Q$ given by its facial structure.
 Let $C_*(P_Q)$ and $C_*(\widetilde{Q})$ denote the cellular chain complexes of $P_Q$ and $\widetilde{Q}$, respectively, and let $H_*(P_Q)$ and $H_*(\widetilde{Q})$ be their respective homology groups. Since $G$ acts cellularly on $\widetilde{Q}$, $C_*(\widetilde{Q})$ is a $\mathbb{Z}(G)$-module.

\vskip.2cm
Let $T$ be a spherical set.
Multiplication by $\beta_T$ defines a homomorphism $\beta_T: C_*(P_Q)\longrightarrow C_*(W_TP_Q)$.
\begin{lem}\label{A6}
$C_*(P_Q^T)$ is contained in the kernel of $\beta_T: C_*(P_Q)\longrightarrow C_*(W_TP_Q)$.
\end{lem}
\begin{proof}
Suppose $\tau$ is a cell in $P_Q^T$.
If $T$ is a spherical set,
then $\tau$ lies in some $F\in \mathcal{F}(P_Q)-\mathcal{F}_B$ such that $s_F\in T$.  Let $\mathcal{B}$ is a subset of $W_T$
such that $W_T=\mathcal{B}\cup s_F\mathcal{B}$, then we can write $\beta_T$ as follows
$$
\beta_T=\sum_{w\in W_T}(-1)^{l(w)}w=\sum_{v\in \mathcal{B}}(-1)^{l(v)}(v-vs_F).
$$
Since $vF$ is identified with $vs_FF$ in $\widetilde{Q}$, we have that $$\beta_T\tau=\sum(-1)^{l(v)}(v-vs_F)\tau=\sum(-1)^{l(v)}(v\tau-v\tau)=0.
$$
Thus, $C_*(P_Q^T)\subset \ker\beta_T$.
\end{proof}

Hence, $\beta_T$ induces a chain map $C_*(P_Q,P_Q^T)\longrightarrow C_*(W_TP_Q)$, still denoted by $\beta_T$.

\vskip.2cm

For each $w\in G$ satisfying that $T=S(w)$ is a spherical set, we then define a map $$\rho^w=w\beta_T:C_*(P_Q,P_Q^T)\overset{\beta_T}{\longrightarrow} C_*(W_TP_Q)\overset{w}{\longrightarrow} C_*(wW_TP_Q).$$ Hence, we have a map
$$
\rho_*^w:H_*(P_Q,P_Q^T)\longrightarrow H_*(wW_TP_Q).
$$
When $T=S(w)=\{t_B\}\cup T_S$ where $T_S$ is a spherical,
$t_Bs=st_B$ for any $s\in T_S$ implies that $W_{T_S}<W_{B}$, i.e., $B\cap F_s\neq \emptyset$ for any $s\in T_S$. So $B'$ does not intersect any $F_s$, hence for $k>1$ we have $${H}_k(P_Q,P_Q^T)\cong {H}_{k-1}(P_Q^T)\cong {H}_{k-1}(P_Q^{T_S} \coprod B')\cong {H}_{k-1}(P_Q^{T_S})\cong {H}_k(P_Q,P_Q^{T_S})$$ where $P_Q$ and $B'$ are contractible simple polytopes.
Now put
$$\rho_*^w: H_k(P_Q,P_Q^T)\cong{H}_k(P_Q,P_Q^{T_S}) \overset{\beta_{T_S}}{\longrightarrow} H_*(W_{T_S}P_Q)\overset{i_*}{\longrightarrow} H_*(W_{T}P_Q)\overset{\times w}{\longrightarrow} H_*(wW_TP_Q).$$

\vskip .2cm
Next,
order the elements of $G$,
$$w_1,w_2,\cdots$$
so that $l(w_i)\leq l(w_{i+1})$. For each $n\geq 1$, put
$$X_n=\bigcup_{i=1}^{n}w_i P_Q.$$
To simplify notation, set $w=w_n$.

\begin{lem}\label{lm17}
$X_{n-1}\cap wP=wP^{S(w)}$.
\end{lem}
\begin{proof}
Notice that $X_{n-1}$ contains a subgraph of Cayley graph of $G$ associated with the generator set $\mathcal{S}$, where the length between each vertex and the unit element is less than or equal to $l(w)$.
Then
$$l(ws)=\begin{cases}l(w)-1, if~ s\in S(w),\\
l(w)+1, if~s\in \mathcal{S}-S(w).
\end{cases}$$
 A chamber $w_iP_Q \ (i<n)$ in $X_{n-1}$ intersects with $wP_Q$ in the facet $wF$ if and only if either  $w_i\cdot s_F=w$ for $F\in \mathcal{F}(P)-\mathcal{F}_B$ or $w_i\cdot t_F^{-1}=w$ for $F\in \mathcal{F}_B$ where $t_F$ is a torsion-free generator in $\mathcal{S}$; in other words, either $l(ws_F)=l(w)-1$ or $l(wt_F)=l(w)-1$. Therefore, $X_{n-1}\cap wP_Q=wP_Q^{S(w)}$.
\end{proof}

Finally let us finish the proof of  \autoref{thm4}.

\vskip .2cm

\noindent{\em Proof of  \autoref{thm4}.}
 We know from \autoref{lm17}  that $X_{n-1}\cap wP_Q=wP_Q^{S(w)}$. Hence, the excision theorem gives an isomorphism
$$
H_*(X_{n},X_{n-1})\overset{\cong}{\longrightarrow} H_{*}(wP_Q,wP_Q^{S(w)}).
$$
Consider the exact sequence of the pair $(X_n,X_{n-1})$
$$
\cdots\rightarrow H_*(X_{n-1})\overset{j_*}{\longrightarrow}H_*(X_n)\overset{k_*}{\longrightarrow}H_*(X_n,X_{n-1})\rightarrow\cdots
$$
We claim that the map $k_*$ is a split epimorphism, which is equivalent to that the map $k_*^w:H_*(X_n)\rightarrow H_*(P_Q,P_Q^{S(w)})$ is a split epimorphism, where $k_*^w$ denotes the composition of $k_*$ with the excision isomorphism and left translation by $w^{-1}$. Consider the map $\rho_*^w$ on $H_*(P_Q, P_Q^{S(w)})$ whose image is contained in $H_*(wW_{S(w)}P_Q)$.  For every $v\neq 1$ in $W_{S(w)}$, we have $l(wv)<l(w)$; hence, $wW_{S(w)}P_Q\subset X_n$. Hence the image of $\rho_*^w$ is contained in $H_*(X_n)$. All these can be seen from the following commutative diagram:
$$
\xymatrix{\ar @{} [dr] |{}
H_*(X_n,X_{n-1}) \ar[r]^{\cong} & H_*(wP_Q, wP_Q^{S(w)}) \ar[d]^{\times w^{-1}} \\
H_*(X_n) \ar @/^/[r]^{k_*^w}  \ar[u]^{k_*}  & \ar @/^/ [l]^{\rho_*^w}  H_*(P_Q, P_Q^{S(w)})\ar[d]^{ \beta_{*}}\\
H_*(wW_{S(w)}P_Q) \ar[u]^{i_*}&\ar[l]^{\times w} H_*(W_{S(w)}P_Q)
}
$$
where $\beta_{*}$ is induced by multiplication by $\beta_{S(w)}$ when $S(w)$ is a spherical set, and is the composition $$\beta_{T_S}\circ i_*: H_k(P_Q,P_Q^T)\cong{H}_k(P_Q,P_Q^{T_S}) \overset{\beta_{T_S}}{\longrightarrow} H_*(W_{T_S}P_Q)\overset{i_*}{\longrightarrow} H_*(W_{T}P_Q)$$ when $S(w)$ is the union of a $\{t_B\}$ and a spherical set $T_S$.
\vskip.2cm

Since $\widetilde{Q}$ is the universal cover of $Q$,
$H_1(\widetilde{Q})\cong 0$. For $*>1$,
it can be see that $k_*^w\circ \rho_*^w$ is the identity on $H_*(P_Q, P_Q^{S(w)})$ by above diagram. Hence there is the following splitting  short exact sequence:
$$0\rightarrow H_*(X_{n-1})\overset{j_*}{\longrightarrow}H_*(X_n)\overset{k_*^w}{\longrightarrow}H_*(P_Q, P_Q^{S(w)})\rightarrow 0.
$$
This implies that
$$H_*(X_n)\cong H_*(X_{n-1})\oplus H_*(P_Q, P_Q^{S(w)})$$
where $H_*(X_1)= H_*(P_Q)=0$.
Since $\widetilde{Q}$ is the increasing union of the $X_n$, we have
$$H_*(\widetilde{Q})=\lim_{n\rightarrow \infty}H_*(X_n)\cong\bigoplus_{w\in G}H_*(P_Q, P_Q^{S(w)}).$$
This completes the proof. $\hfill\Box$

\subsection{Proof of \hyperref[DJS-small]{Theorem A}}
Now let us give the proof of \hyperref[DJS-small]{Theorem A}.
Let $Q$ be a simple handlebody, and $q:P_Q\longrightarrow Q$ be the quotient map by gluing all paired facets in $\mathcal{F}_B$.
Then the orbifold universal cover $\pi:\widetilde{Q}\rightarrow Q$ of $Q$ can be constructed by (\ref{uni-cover}).

\vskip .2cm
Let $\mathcal{C}(P_Q)$ be the standard cubical cellular decomposition of $P_Q$. For each cube $c\in \mathcal{C}(P_Q)$, each component of $\pi^{-1}(c)$ is a cube in $\widetilde{Q}$. Then $\mathcal{C}(P_Q)$ determines a cubical  cellular decomposition of $\widetilde{Q}$, denoted by $\mathcal{C}(\widetilde{Q})$, such that the link of each point $v$ in $\mathcal{C}(\widetilde{Q})$ is exactly the nerve $\mathcal{N}(P_Q)$ of $P_Q$.
Hence, if $Q$ is flag, then $P_Q$ is flag, so is $\mathcal{N}(P_Q)$.  By Gromov lemma (\autoref{LMG}), $\widetilde{Q}$ is non-positively curved. In fact, $\widetilde{Q}$ is a CAT(0) space. Then by  Cartan-Hadamard Theorem, $\widetilde{Q}$ is aspherical. Therefore, $Q$ is orbifold-aspherical.

\vskip.2cm

On the contrary,
if $Q$ is orbifold-aspherical, then $\widetilde{Q}$ is contractible.
 Using an idea of Davis in \cite[Subsection 8.2]{D3}, we shall show that if $P_Q$ is not flag then $\widetilde{Q}$ is not contractible.
 Indeed,  if $P_Q$ is not flag, then $\mathcal{N}(P_Q)$ contains an empty $k$-simplex for some $k\geq 2$. The dual of this empty $k$-simplex gives an essential embedding sphere in $\widetilde{Q}$. By \autoref{nonempty}, the fundamental class of such a sphere is nontrivial in $H_k(\widetilde{Q})$, which contradicts that $\widetilde{Q}$ is contractible.
 This completes the proof.  $\hfill \square$



\section{Proof of \hyperref[FTT]{Theorem B}}\label{section6}

\vskip .2cm

The  purpose of this section is to characterize the rank two free abelian subgroup $\mathbb{Z}\oplus \mathbb{Z}$ in $\pi^{\text{orb}}_1(Q)$ in terms of an  $\square$-belt in $Q$.

\begin{theoremB}\label{TH0}
  Suppose that $Q$ is a simple $n$-handlebody. Then there is a rank two free abelian subgroup $\mathbb{Z}\oplus \mathbb{Z}$ in $\pi^{\text{orb}}_1(Q)$ if and only if $Q$ contains an $\square$-belt.
\end{theoremB}

\begin{rem}
The ``simple'' condition of a handlebody is necessary  in above proposition. In fact, it is easy to see that the orbifold fundamental group of a two-dimensional annulus as a right-angled Coxeter orbifold is isomorphic to $\mathbb{Z}\oplus (\mathbb{Z}_2*\mathbb{Z}_2)$, which contains a rank two free abelian subgroup $\mathbb{Z}\oplus\mathbb{Z}$.  Consider a right-angled Coxeter 3-handlebody $Q$ with an $\pi_1$-injective annulus-suborbifold $B$ such that $B$ is a $\pi_1$-injective suborbifold,  it provides a subgroup $\mathbb{Z}\oplus\mathbb{Z}$ in its orbifold fundamental group. Of course, such $Q$ is not simple.  All of these results are the generalization of   \cite[Lemma 5.22]{BH} which is related to the Flat Torus Theorem in \cite[Chapter II.7]{BH}.
\end{rem}

 \begin{example}[$\square$'s of \autoref{example-31}]  \label{Example-51}
 We show that each $\square$ in (c) and (d) of \autoref{example-31} determines a subgroup  $\mathbb{Z}\oplus \mathbb{Z}$  in $\pi^{\text{orb}}_1(Q)$, whereas the cases of (a) and (b) do not so.

\vskip .2cm
 On (a), the four facets $F_1,F_2,F_3,F_4$ correspond to a suborbifold $B$ which is a quadrilateral in $Q^*$,
 but it is not an $\square$-belt in $Q$. In fact,  $$i_*(\pi_1^{orb}(B))\cong W_\square/\langle (s_1s_3)^2\rangle\cong (\mathbb{Z}_2)^2\oplus (\mathbb{Z}_2*\mathbb{Z}_2)<\pi^{orb}_1(Q)$$ and
 $s_1s_3, s_2s_4$ generate a subgroup $\mathbb{Z}_2\oplus \mathbb{Z}$ in $i_*(\pi_1^{orb}(B))<\pi^{orb}_1(Q)$, where $i_*:\pi_1^{orb}(B)\rightarrow \pi_1^{orb}(Q)$ is induced by the inclusion $i:B\hookrightarrow Q$. Thus, there is no subgroup $\mathbb{Z}\oplus \mathbb{Z}$ in $i_*(\pi_1^{orb}(B))$.

\vskip.2cm

 On (b), $\{F_1,F_2,F_3,F_4\}$  does not determine a quadrilateral sub-orbifold. Without loss of generality, assume that $\{F_1,F_2,F_3,F_4\}$ bounds only one hole of $Q^*$. Then there are at least $5$ generators in $\pi^{orb}_1(Q)$ associated to five facets in $P_Q$, denoted by $\{F_1,F_2,F_3,F_4, F_1'\}$ with $F_1\cap B^+\sim F'_1\cap B^-$, where $B$ is the cutting belt of $Q$ and cut $F_1$ into two facets in $P_Q$. Thus, (b) induces a subgroup of $\pi^{orb}_1(Q)$ as follows:
 \begin{equation*}\begin{split}
 W_{b}:=& \langle s_1,s_2,s_3,s_4,s'_1,t\mid (s_i)^2=1, \forall i; (s_1s_2)^2=(s_2s_3)^2=(s_3s_4)^2=(s_4s'_1)^2=1; s'_1=ts_1t\rangle\\
&=\langle s_1,s_2,s_3,s_4,t\mid (s_i)^2=1, \forall i; (s_1s_2)^2=(s_2s_3)^2=(s_3s_4)^2=(s_4ts_1t^{-1})^2=1\rangle
 \end{split}
 \end{equation*}
which contains no subgroup $\mathbb{Z}\oplus \mathbb{Z}$.

\vskip.2cm
On (c) or (d),  $\{F_1,F_2,F_3,F_4\}$  determines an $\square$-belt $B_\square$ of $Q$.
If $B_\square$ does not intersect with any cutting belt, then $B_\square$ is kept in $P_Q$, so there is a subgroup $\mathbb{Z}\oplus \mathbb{Z}<W(P_Q,\mathcal{F}_B)<\pi_1^{orb}(Q)$.
If there are some cutting belts $B_1,B_2,\cdots, B_k$ intersecting transversely  with only a pair of disjoint edges of $B_\square$, without loss of generality, assume that $B_1,B_2,\cdots, B_k$ intersect with two disjoint edges $f_1$ and $ f_3$ of $B_\square$, where some cutting belts may cut $f_1$ and $f_3$ many times, see (c) in \hyperref[F7]{Figure 7}. Then there is also a subgroup $\mathbb{Z}\oplus \mathbb{Z}$ generated by $s_1s_3$ and $s_2t_1t_2\cdots t_k s_4 t_k^{-1}\cdots t^{-1}_2t^{-1}_1$ where each $t_i$ is one of $\{t^{\pm 1}_B\}$. Also see the following \hyperref[F9]{Figure 9}.
 \begin{figure}[h]\label{F9}
\centering
\def\svgwidth{0.55\textwidth}
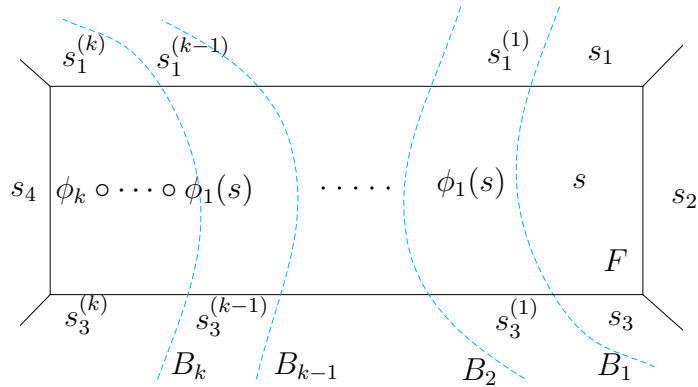
\caption{$\square$-belt and $\mathbb{Z}\oplus \mathbb{Z}$.}
\end{figure}
\end{example}

\subsection{The special case where $Q$ is a simple polytope} \label{spp}
First let us prove \hyperref[TH0]{Theorem B} when $Q$ is a simple polytope. This case can be followed by  Moussong's result~\cite{M} (see also~\cite[Corollary 12.6.3]{D2} in details).
Here we give an alternative proof as follows.
\vskip .2cm

Let $W_P=\langle S\ |R\rangle$ be the right-angled Coxeter group associated with a simple polytope $P$. We are going to show that there is a subgroup $\mathbb{Z}\oplus \mathbb{Z}$ in $W_P$ if and only if $P$ contains an $\square$-belt.

\vskip .2cm

Assume that $w=s_1\cdots s_m$ is a reduced word of length $m$ in $W_P$, and $t$ is a generator in $S$.
\begin{itemize}
\item If the length of $wt$ equals to $m-1$ after a sequence of elementary operations, then we call $t$ {\bf DIE} in $w$, i.e.,  there is a $s_i=t$ such that $(s_{>i}t)^2=1$ ($s_{>i}$ is a $s_j$ with $j>i$);
\item  If $wt=tw$ is reduced, then we call $t$ {\bf SUCCESS} for $w$, meaning that $t$ can commutate with all $s_i$;
\item If $wt$ is reduced and $wt\neq tw$, then we call $t$ {\bf FAIL} for $w$; in other words, there is a $s_i$ in $w$ such that $(s_it)^2\neq1$ and $(s_{>i}t)^2=1$.
\end{itemize}

\noindent{\em Proof of \hyperref[TH0]{Theorem B} for $Q$ to be a simple polytope $P$.}
Let $W_\square$ be the right-angled Coxeter group determined by a quadrilateral. Then two pairs of  disjoint edges of $\square$ provide two elements $s_1s_3$ and $s_2s_4$ which generate a subgroup $\mathbb{Z}\oplus \mathbb{Z}$ in $W_\square<W_P$.
Suppose that  there is subgroup $\mathbb{Z}\oplus \mathbb{Z}$ in $W_P$.
Then we need to find a required $\square$-belt in $P$.
We proceed as follows.

\vskip.2cm

\noindent{\bf  Claim-1} {\em There are two generators $x=s_1\cdots s_m$ and $y=t_1\cdots t_n$ of $\mathbb{Z}\oplus \mathbb{Z}$ such that all $t_i$ commutate with all $s_j$.}
\vskip.2cm

Assume that $x=s_1\cdots s_m$ and $y=t_1\cdots t_n$ are arbitrary two reduced expressions, which generate a  subgroup $\mathbb{Z}\oplus \mathbb{Z}$ in $W_P$. Then $xy=yx$, giving that
$$
s_1\cdots s_m\cdot t_1\cdots t_n=t_1\cdots t_n\cdot s_1\cdots s_m
$$
where, without loss of generality, assume $m\geq n$.
 Tits theorem (\autoref{Tits}) tells us that the word $w=s_1\cdots s_m\cdot t_1\cdots t_n$ can be turned into $t_1\cdots t_n\cdot s_1\cdots s_m$ by a series of elementary operations. We perform an induction with $xt_1$ as a starting point. On $xt_1$, there are the following three cases:
\begin{itemize}
\item[(A)] $t_1$ is DIE at $s_k$ in $x$. Then we may write $x=t_1s'_2\cdots s'_{m-1}t_1$ and $y=t_1t'_2\cdots t'_{n-1}t_1$. So we can take two shorter words $x'=s'_2\cdots s'_{m-1}$ and $y'=t'_2\cdots t'_{n-1}$  as generators of $\mathbb{Z}\oplus \mathbb{Z}$ as well. This returns back to the starting of our argument with two words of shorter word length.
\item[(B)]  $t_1$ is FAIL in $x$. Then we can take $s_1=t_1$, so $x=t_1s_2\cdots s_m$.
\item[(C)]  $t_1$ is SUCCESS. Then $t_1$ commutates with all $s_i$.
\end{itemize}

Consider $xt_2$ in the case (B),
if  $t_2$ is DIE, then $(t_1t_2)^2=1$, and one may write $x=t_2t_1s'_3\cdots s'_{m-1}t_2$ and $y=t_2t_1t'_3\cdots t'_{n-1}t_2$. In a similar way to the case (A), set $t_1s'_3\cdots s'_{m-1}$ and $t_1t'_3\cdots t'_{n-1}$ as new  generators of  $\mathbb{Z}\oplus \mathbb{Z}$.
If  $t_2$ is FAIL, then one may write $x=t_1t_2s_3\cdots s_m$.
If  $t_2$ is SUCCESS, then $t_2$ commutates with all $s_i$ (including $s_1=t_1$).
\vskip.2cm

Consider $xt_2$ in  the case (C), if  $t_2$ is DIE, then one can take $s_1=t_2$, so $x=t_2s_2\cdots s_m$ and $(t_1t_2)^2=1$.  Moreover, exchanging $t_1$ and $t_2$ in $y$ returns to the case (A), so we can take two shorter words as generators  of  $\mathbb{Z}\oplus \mathbb{Z}$.
If $t_2$ is FAIL, then $x=t_2s_2\cdots s_m$.
Otherwise, $t_2$ is SUCCESS, too.

\vskip.2cm
The above procedure can always be carried out by inductive  hypothesis. We can end this procedure after finite steps  of elementary operations until we have obtained a complete analysis for all $t_i$. Actually, each $t_i$ is either FAIL or SUCCESS for the final $x$ and $y$. There are only three possibilities as follows:
\begin{itemize}
\item[$(i)$] All $t_i$ are FAIL. In this case, we may write $x=t_1t_2\cdots t_ns_{n+1} \cdots s_m=ys_{n+1} \cdots s_m$, so that we can take $y^{-1}x$ and $y$ as new generators of $\mathbb{Z}\oplus \mathbb{Z}$. Of course, all $t_i$ can commutate with all $s_j$, as desired.
\item[$(ii)$] Some $t_i$'s are FAIL. In this case, we may write $x=t_{i_1}\cdots t_{i_k}s_{k+1}\cdots s_m$ such that each of those  $t_{j_u}\not=t_{i_1}, ..., t_{i_k}$ commutates with all $s_i$ and $t_{i_1},\cdots, t_{i_k}$. So one may write $x=t_1\cdots t_n\cdot t_{j_1}\cdots t_{j_{n-k}}\cdot s_{k+1}\cdots s_m$. Then $y^{-1}x$ removes those FAIL $t_i$'s in $x$.
      Furthermore,  $y^{-1}x$ and $y$ can be chosen as  new generators of $\mathbb{Z}\oplus \mathbb{Z}$ as desired.
\item[$(iii)$]  All $t_i$ are SUCCESS. In this case, $x$ and $y$ are naturally the required generators of Claim-1.
\end{itemize}

Thus we finish the proof of Claim-1.

\vskip.2cm
Now choose two generators $x$ and $y$ of $\mathbb{Z}\oplus \mathbb{Z}$ which satisfy the property in Claim-1.

\vskip .2cm

\noindent{\bf Claim-2:} {\em There are two letters $s,s'$ in $x$ and two letters $t,t'$ in $y$, which correspond to
four facets of $P$, denoted by $F_s,F_{s'},F_t,F_{t'}$, that form an $\square$ in $P^*$.}

\vskip.2cm

Since $x$ is free, there must exist two letters $s, s'$ in $x$ such that $F_{s}\cap F_{s'}=\emptyset$. Similarly,  there also exist two letters $t, t'$ in $y$ such that $F_{t}\cap F_{t'}=\emptyset$.
If   $\{s,s'\}\cap \{t,t'\}= \emptyset$, since $t, t'$ commute with $s,s'$, then clearly $F_s,F_{s'},F_t,F_{t'}$ determine an $\square$-belt in $P$.  Otherwise,  $\{s,s'\}\cap \{t,t'\}\neq \emptyset$.
Assume that $s=t$, then $(tt')^2=(st')^2=1$. This gives a contradiction since $y$ is a reduced word.

\vskip .2cm
Together with the above arguments, this completes the proof. $\hfill\Box$

\vskip .2cm
Next let us deal with the case of a simple handlebody.  Let $Q$ be a  simple handlebody of genus $\mathfrak{g}$, and $P_Q$ be the associated simple polytope obtained by cutting $Q$ open along  cutting belts $\{B_i | i=1, 2, \cdots, \mathfrak{g}\}$.

\subsection{Proof of the sufficiency  of \hyperref[TH0]{Theorem B}}
Assume that there is an $\square$-belt $B_\square$ given by $\{F_1,F_2,F_3,F_4\}$ in $\mathcal{N}(Q)$. After cutting $Q$ open along  cutting belts $B_i, i=1, 2, \cdots, \mathfrak{g}$,
  by \autoref{Lemma-34}, there are the following two cases.
\begin{itemize}
\item
The $B_\square$ is still kept in  $P_Q$. Then $B_\square$ gives a subgroup $ \mathbb{Z}\oplus\mathbb{Z}$ in $W(P_Q, \mathcal{F}_B)<\cdots<\pi^{orb}_1(Q)$, which is generated by $s_1s_3$ and $s_2s_4$.

\item
 The $B_\square$ is not kept in  $P_Q$. Then there is only one situation in which  some cutting belts $B_i$ intersect transversely with a pair of disjoint edges of $B_\square$, say   $F_1$ and $F_3$.
 If  $B_\square$ intersects transversely with cutting belts $B_1, B_2, \cdots, B_k$ in turn, then  $s_1s_3$ and $s_2t_1\cdots t_ks_4t_k^{-1}\cdots t_1^{-1}$ generate a subgroup  $\mathbb{Z}\oplus\mathbb{Z}$ in $\pi^{orb}_1(Q)$, as the cases of (c) or (b) on \autoref{example-31}.
 See also \autoref{Example-51}.
 $\hfill\Box$
 \end{itemize}

\subsection{Proof of the necessity  of \hyperref[TH0]{Theorem B}}

 Cutting $Q$ open along a cutting belt $B$, we get a    right-angled Coxeter $n$-handlebody of genus $\mathfrak{g}-1$, denoted by $Q_{\mathfrak{g}-1}$. Conversely, $Q$
can be recovered from $Q_{\mathfrak{g}-1}$ by gluing its two disjoint boundary facets associated with $B$, which implies that the orbifold fundamental group of $Q$ is an  HNN-extension on $\pi^{orb}_1(Q_{\mathfrak{g}-1})$.
 Write $G_{\mathfrak{g}-1}=\pi^{orb}_1(Q_{\mathfrak{g}-1})$, and let $W_{B^+}$ and $W_{B^-}$ be two isomorphic subgroups of $G_{\mathfrak{g}-1}$ determined by two copies of $B$. Then we have
\begin{equation}
\pi^{orb}_1(Q)\cong G_{\mathfrak{g}-1}*_{\phi}= \langle G_{\mathfrak{g}-1}, t\mid t^{-1}a t=\phi(a), a\in W_{B^-}\rangle
\end{equation}
 where $\phi:W_{B^-}\rightarrow W_{B^+}$ is an isomorphism by mapping $s'\in W_{B^-}$ into $s\in W_{B^+}$.
Generally, $\pi^{orb}_1(Q)$ is isomorphic to  $\mathfrak{g}$ times HNN-extensions on the right-angled Coxeter group $W(P_Q, \mathcal{F}_B)$ as we have seen in the proof of \autoref{HNN}:
$$\xymatrix@C=0.5cm{
 \pi^{\text{orb}}_1(Q)&& \ar[ll]G_{\mathfrak{g}-1}&& \ar[ll]\cdots&& \ar[ll] G_1 &&\ar[ll]  G_0=W(P_Q,\mathcal{F}_B)\\
  }$$
where  each $G_{k}$ is also an HNN-extension over $G_{k-1}$ for $1\leq k \leq \mathfrak{g}-1$, and $G_0=W(P_Q,\mathcal{F}_B)$ is a right-angled Coxeter group.

 \vskip.2cm

 According to the normal form theorem of HNN-extension (\autoref{TT2}), each element $x$ in $\pi^{orb}_1(Q)$ has a unique iterative normal form. First, write
   $$x=g_0t_\mathfrak{g}^{\epsilon_1}g_1t_\mathfrak{g}^{\epsilon_2}\cdots g_{n-1}t_\mathfrak{g}^{\epsilon_n}g_n$$
   as a normal form for $t_\mathfrak{g}$ where $g_i\in G_{\mathfrak{g}-1}$. Next  inductively each $g_i$ is also a normal form in $G_k$ for $1\leq k \leq \mathfrak{g}-1$. More generally,  $x$ has a unique form
 \begin{equation}\label{Reducedx}
x=g_0t_{1}g_1\cdots g_{m-1}t_{m}g_m
 \end{equation}
 where each $g_i$ is reduced in $G_0=W(P_Q,\mathcal{F}_B)$, and each $t_i$ is one of $\{t^{\pm 1}_B\}$ which determines an isomorphism of $\{\phi^{\pm 1}_B\}$ on some subgroups of $\pi_1^{orb}(Q)$. This expression of  $x$ is a normal form with respect to all possible $t_{B}$.
  The expression in (\ref{Reducedx}) is called a {\em reduced normal form} of $x$ in $\pi_1^{orb}(Q)$.
  The number $m$ is called the {\em (total) $t$-length of $x$}.

 \vskip.2cm

By applying the Tits Theorem (\autoref{Tits})  and the Normal Form Theorem of HNN-extension (\autoref{TT2}), we have the following conclusion.
\begin{lem}\label{Lemma-51}
Two reduced words $x,y$  are the same in $\pi^{orb}_1(Q)$  if and only if one of both $x$ and $y$ can be transformed into the other one by a sequence of commutations of right-angled Coxeter group and $t$-reductions of HNN-extension.
\end{lem}

Next, we prove two lemmas.
\begin{lem}\label{Lemma-52}
If there is a subgroup $\mathbb{Z}\oplus \mathbb{Z}$ in $\pi_1^{orb}(Q)$, then one  generator of  $\mathbb{Z}\oplus \mathbb{Z}$ can be presented as a cyclically reduced word in $W(P_Q,\mathcal{F}_B)$.
\end{lem}

\begin{proof}
Assume that there is a subgroup $\mathbb{Z}\oplus \mathbb{Z}$ in $\pi^{orb}_1(Q)$, which is generated by two reduced normal forms as in (\ref{Reducedx}):
  $$x=g_0t_{1}g_1\cdots g_{m-1}t_{m}g_m$$
  and
   $$y=h_0t'_{1}h_1\cdots h_{n-1}t'_{n}h_n.$$
   Then $xy=yx$ in $\pi^{orb}_1(Q)$.   By \autoref{Lemma-51}, $xy$ and $yx$ have the same reduced normal form as in (\ref{Reducedx}).

\vskip .2cm

We do $t$-reductions on
 $$xy=g_0t_{1}g_1\cdots g_{m-1}t_{m}g_m\cdot h_0t'_{1}h_1\cdots h_{n-1}t'_{n}h_n$$
 and
 $$yx=h_0t'_{1}h_1\cdots h_{n-1}t'_{n}h_n\cdot g_0t_{1}g_1\cdots g_{m-1}t_{m}g_m.$$

Since $x,y$ are reduced normal forms, $xy$ and $yx$ have the same tails.
Without  loss of generality,  assume that $m\geq n$. Write $\widetilde{y}=t'_{1}h_1\cdots h_{n-1}t'_{n}h_n=h_0^{-1}y$. Then
$x$ can be written as
$$
x=g_0t_{1}g_1\cdots t_{m-n}g_{m-n}\widetilde{y}=g_0t_{1}g_1\cdots t_{m-n}g_{m-n}\cdot h_0^{-1}y.
$$
Since $x$ and $y$ generate $\mathbb{Z}\oplus \mathbb{Z}$,  both $y$ and $xy^{-1}$ do so. The word $xy^{-1}$ has a shorter $t$-length. We further do $t$-reductions on $xy^{-1}$ to get a normal form, also denoted by $x$.
\vskip.2cm
We can always continue to do this algorithm,  so that we can take either $x$ or $y$ from $W(P_Q,\mathcal{F}_B)$. Suppose $y=h\in W(P_Q,\mathcal{F}_B)$.

\vskip.2cm

Furthermore,  we can assume that $h$ is a cyclically reduced word in $W(P_Q,\mathcal{F}_B)$. In fact, if $h$ is not  cyclically reduced, without  loss of generality, assume that $h$ is of the form $w^{-1}h'w$, where $w$ is an arbitrary word and $h'$ is a cyclically reduced word in $W(P_Q,\mathcal{F}_B)$. Then we replace $h$ by $h'$, such that  $h'$ and $wxw^{-1}$ generate a $\mathbb{Z}\oplus \mathbb{Z}$ in $\pi_1^{orb}(Q)$. This completes the proof.
\end{proof}

\begin{lem}\label{Lemma-54}
 Let $x=g_0\cdot t_1\cdots t_k$ be a reduced normal form, where $g_0\in W(P_Q,\mathcal{F}_B)$ and each $t_i$ is  one of $\{t^{\pm 1}_B\}$, and $h$ be a cyclically reduced word in $W(P_Q,\mathcal{F}_B)$.
Then $x,h$ cannot generate a $\mathbb{Z}\oplus \mathbb{Z}$ in $\pi_1^{orb}(Q)$.
\end{lem}
\begin{proof}
If $x,h$  generate a $\mathbb{Z}\oplus \mathbb{Z}$ in $\pi_1^{orb}(Q)$, then
$$x\cdot h=g_0\cdot t_1\cdots t_k \cdot h=g_0h' \cdot t_1\cdots t_k$$
  where $h'=\phi_1\circ\cdots \circ\phi_k(h)$ is the image of the composition of some $\phi_i$ on $h$.

\vskip .2cm
We first claim that $h'$ is reduced in $W(P_Q,\mathcal{F}_B)$, and  the word length of $h'$ and $h$ are equal.
In fact, for each $i$, $\phi_i$ is an isomorphism from some $W_{B^-}$ to $W_{B^+}$ which maps generators  to generators, and all $W_{B^+}$ and $W_{B^-}$ are subgroups of $W(P_Q,\mathcal{F}_B)$.

\vskip.2cm

Next, we claim that $h=h'$. In fact,
$xh=hx$ implies that $g_0h'=hg_0$, that is $h'=g_0^{-1}hg_0$.
Let $s$ be a letter in  $g_0$.
If $s$ is FAIL in $h$, then the length of $h'$ is greater than the length of $h$, which is a contradiction. If   $s$ is DIE  in $h$, then $h$ has a form $s\overline{h}s$, which contradicts that $h$ is cyclically reduced. Thus, all letters in $g_0$ are SUCCESS in $h$. In other words, $g_0h=hg_0=g_0h'$, Thus $h=h'=\phi_1\circ\cdots \circ \phi_k(h)$.

\vskip.2cm

If $\phi_1\circ\cdots \circ \phi_k= id$, then the associated sequence $t_1\cdots t_k=1$, which contradicts that $x$ is reduced.
If $\phi_1\circ\cdots \circ \phi_k\neq id$ has a fixed point in all letters in $h$, then there is a letter $s_0$ in $h$ such that $\phi_1\circ\cdots \circ \phi_k(s_0)=s_0$. Furthermore,   $s_0, \phi_k(s_0),\cdots, \phi_1\circ\cdots \circ \phi_{k-1}(s_0)$ determine a non-contractible facet in $Q$, which contradicts that $Q$ is simple.
More generally, if $\phi_1\circ\cdots \circ \phi_k\neq id$ has no any fixed point, then there is a generator $s_1$ as a letter in $h$, such that $s_2=\phi_1\circ\cdots \circ \phi_k(s_1)\neq s_1$. Continue this procedure,
one can get a sequence $s_1,s_2, s_3, ...$, such that each $s_i$ is a generator as a letter in $h$ and $s_i=\phi_1\circ\cdots \circ \phi_k(s_{i-1})$. However, the word length of $h$ is finite, thus there must be two same elements in the sequence. Geometrically, this means that there is a non-contractible facet in $Q$, which contradicts that $Q$ is simple.
This completes the proof.
\end{proof}

Now let us give the proof of the necessity  of \hyperref[TH0]{Theorem B} in the general case.

\vskip .2cm
\noindent {\em Proof of the necessity  of \hyperref[TH0]{Theorem B}.}  Suppose that there are two elements $x$ and $y$ in $\pi^{\text{orb}}_1(Q)$ which generate a rank two free abelian subgroup $\mathbb{Z}\oplus \mathbb{Z}$. Our arguments are divided into  the following steps.
 \vskip.2cm
 \noindent{\bf Step-1.} {\em Simplify two generators $x,y$ of $\mathbb{Z}\oplus \mathbb{Z}$ by doing $t$-reductions.}
\vskip.2cm
 \autoref{Lemma-52} tells us that one of $x,y$  can be chosen as a  cyclically reduced word $h$ in $W(P_Q,\mathcal{F}_B)$,  say $y=h$.
 Now if $x$ is also a word in $W(P_Q,\mathcal{F}_B)$ (i.e., the $t$-length of $x$ is zero), then by \hyperref[spp]{Subsection 5.1}, there is an $\square$-belt in $P_Q$ which can appear in $Q$, as desired.
\vskip.2cm
Next let us consider the case in which the  $t$-length of $x$ is greater than zero.
Let  $$x=g_0t_{1}g_1\cdots g_{m-1}t_{m}g_m$$
  be a reduced normal form in $\pi_1^{orb}(Q)$.
Then  $xh=hx$ implies that
  \begin{itemize}
\item $g_mh=hg_m; ~~t_m\cdot h=\phi_m(h)\cdot t_m$;
\item $g_{m-1}\cdot\phi_m(h)=\phi_m(h)\cdot g_{m-1}; ~~t_{m-1}\cdot \phi_m(h)=\phi_{m-1}\circ\phi_m(h)\cdot t_{m-1}$;\\
$\cdots$
\item $g_0\cdot\phi_1\circ\cdots\circ\phi_m(h)=\phi_1\circ\cdots\circ\phi_m(h)\cdot g_0$
  \end{itemize}
  where each $\phi_i: W_{B_i}\longrightarrow W_{B_i'}$ is an isomorphism determined by  some $B_i\in \mathcal{F}_B$,  each $\phi_i\circ\cdots\circ\phi_m(h)$ is an expression  in $W_{B'_i}\cap W_{B_{i-1}}$ for $i=2,\cdots,m$   and $\phi_1\circ\cdots\circ\phi_m(h)\in W_{B_1'}$,  $h\in W_{B_m}$. Here two $B_i$ and $B_j$ may correspond to the same $B\in \mathcal{F}_B$.

\vskip .2cm
\noindent{\bf Step-2. }{\em Find facets $F_1,F_3$ around $B$ or $B'$ in $P_Q$.}
\vskip .2cm
Without loss of generality,  $h\in W(P_Q,\mathcal{F}_B)$ is a cyclically reduced word.
Since $h$ is a free element in $W_{B_m}\cap W(P_Q,\mathcal{F}_B)$,  we can take two generators $s_1$ and $s_3$ in $h$ corresponding  to two disjoint facets $F_1$ and $F_3$ of $P_Q$ such that
$F_1$ and $F_3$ intersect with $B_m$. In particular,  $s_1s_3$ is a free element in $W_{B_m}<W(P_Q,\mathcal{F}_B)=G_0<\cdots<G_{\mathfrak{g}-1}<G_\mathfrak{g}=\pi^{orb}_1(Q)$.

\vskip .2cm

\noindent{\bf Step-3.~}{\em Find the facet $F_2$ which intersects with $F_1,F_3$.}
\vskip .2cm

If $g_m\neq 1$, since $x$ is a normal form, then $g_m$ is
a representative of a coset of $W_{B_m}$ in $\pi_1^{orb}(Q)$.
 Thus there is a generator  $s_2\notin S(W_{B_m})$ in $g_m$ such that $hs_2=s_2h$, where $S(W_{B_m})$ is the generator set of $W_{B_m}$. This generator $s_2$ determines  a facet  $F_2$ in $P_Q$, as desired.
\vskip.2cm

If $g_m=1$, then
$
x\cdot h=g_0t_{1}g_1\cdots t_{m} \cdot h=g_0t_{1}g_1\cdots t_{m-1}g_{m-1}\cdot\phi_m(h)\cdot t_{m}.
$
A similar argument shows that either there is a $s_2\notin S(W_{B_{m-1}})$ as desired, or
  $$x\cdot h=g_0t_{1}g_1\cdots t_{m-1}\cdot\phi_m(h)\cdot t_{m}=g_0t_{1}g_1\cdots t_{m-2}g_{m-2}\cdot\phi_{m-1}\circ\phi_m(h)\cdot t_{m-1} t_{m}.$$

  We can continuously carry out  the above procedure. Finally we can arrive at two possible cases:
  \begin{itemize}
\item There  exist  some $g_i\neq 1$ for $i>0$. Then  there must be a letter $s_2$ in $g_i$ which determines the required $F_2$;
\item  $x$ is of the form  $x=g_0\cdot t_1\cdots t_m$, where $g_0\in W(P_Q,\mathcal{F}_B)$ and $t_1\cdots t_m$ is a word formed by letters in $\{t^{\pm 1}_{B}\}$.
By \autoref{Lemma-54},  $x$ and $h$ cannot generate a subgroup $\mathbb{Z}\oplus \mathbb{Z}$ in $\pi_1^{orb}(Q)$. So $x=g_0\cdot t_1\cdots t_m$ is impossible.
\end{itemize}

\vskip.2cm

 Thus, we can always find a facet $F_2$ from a nontrivial $g_i$ in the reduced form (\ref{Reducedx}) of $x$ where $i>0$.

\vskip .2cm

\noindent{\bf Step-4. }{\em Find a facet $F_4$ such that  $F_1, F_2, F_3, F_4$ determine an $\square$-belt in $Q$.}

\vskip .2cm

We proceed our argument as follows.

\vskip.2cm

 (I). If there is only a $g_i\neq 1$ (i.e., $g_j=1$ for any $j\not=i$) in the expression of $x$, then $x=t_{1}\cdots t_{i}\cdot g_i \cdot t_{i+1}\cdots t_m$, where $i$ must be more than zero by \autoref{Lemma-54}.  Now   $xh=hx$ implies that $t_{1}\cdots t_{i}\cdot  t_{i+1}\cdots t_m=1$. Actually, if $t_{1}\cdots t_{i}\cdot  t_{i+1}\cdots t_m\neq 1$, then $\phi_1\circ \cdots\circ \phi_m(h)=h$ implies that there is a non-contractible facet in $Q$, which is impossible (also see the proof of \autoref{Lemma-54}).
  Thus, $t_{1}\cdots t_{i}=(t_{i+1}\cdots t_m)^{-1}$,
  so $x=t_{1}\cdots t_{i}\cdot  g_i\cdot t_{i+1}\cdots t_m=(t_{i+1}\cdots t_m)^{-1}g_i (t_{i+1}\cdots t_m)$.
Since $x,h$ generate a $\mathbb{Z}\oplus \mathbb{Z}$, we see that $g_i, \phi_{i+1}\circ\cdots\circ \phi_m(h)$ generate a $\mathbb{Z}\oplus \mathbb{Z}$ in $W(P_Q,\mathcal{F}_B)$. Then by  \hyperref[spp]{Subsection 5.1}, there is an $\square$-belt in $Q$.

  \vskip.2cm
(II). If there are at least two nontrivial $g_i, g_{j}\neq 1$ in $x$ where $0<j< i\leq m$ but $g_k=1$ for all $k>j$ and $k\neq i$, then one may write $x=\cdots t_j\cdot  g_j\cdot t_{j+1}\cdots t_{i}\cdot g_i \cdot  t_{i+1}\cdots t_m$. So we have
\begin{equation}
\begin{split}
xh&=\cdots  t_j\cdot  g_j\cdot t_{j+1}\cdots t_{i}\cdot g_i \cdot  t_{i+1}\cdots t_m \cdot h\\
&=\cdots  t_j\cdot  g_j\cdot t_{j+1}\cdots t_{i}\cdot g_i \cdot h' t_{i+1}\cdots t_m\\
&=\cdots  t_j\cdot  g_j\cdot t_{j+1}\cdots t_{i}\cdot h' g_i \cdot  t_{i+1}\cdots t_m\\
&=\cdots t_j\cdot  g_j\cdot h'' \cdot t_{j+1}\cdots t_{i}\cdot g_i \cdot  t_{i+1}\cdots t_m.
\end{split}
\end{equation}
where $h'=\phi_{i+1}\circ\cdots\circ\phi_m(h)$ and $h''=\phi_{j+1}\circ\cdots\circ\phi_m(h)$. Since $xh=hx$, we have that
$g_jh''=h''g_j$, so   we can take a generator  $s_4$ in $g_{j}$ (not in $S(W_{B_{j}})$) such that $h''s_4=s_4h''$. Similarly, here $s_4$  determines a facet $F_4$ of $P_Q$ such that $F_4\cap F_1''\neq \emptyset$ and $F_4\cap F_3''\neq \emptyset$ where $F_1''$ and $F_3''$ are two facets of $P_Q$ determined by the images of  $\phi_{j+1}\circ\cdots\circ \phi_m$ on $s_1,s_3$.
In particular, $F_2\neq F_4$ in $P_Q$. Otherwise, the intersection of $q(F_1)$ and $q(F_2)$ in $Q$ is disconnected where $q:Q\rightarrow P_Q$ is defined in (\ref{eq-31}) , which contradicts that $Q$ is simple.
Hence, we get an $\square$-belt in $Q$.
 \vskip.2cm

(III).  If there are only $g_0$ and $g_i$ that are non-trivial in $x$ where $i>0$, then one may write $x=g_0 t_1\cdots t_{i} \cdot g_i \cdot t_{i+1} \cdots t_m$. Without  loss of generality, assume that $g_0,g_i$ are two reduced words in $W(P_Q,\mathcal{F}_B)$.
 Now if $x=g_0 t_1\cdots t_i \cdot g_i \cdot t_{i+1} \cdots t_m=t_1\cdots t_i \cdot g'_0g_i \cdot t_{i+1} \cdots t_m$ where $g'_0=\phi^{-1}_i\circ\cdots\circ\phi^{-1}_1(g_0)$, then by the proof of \autoref{Lemma-54}, $xh=hx$ implies that $t_1\cdots t_i\cdot t_{i+1} \cdots t_m=1$.
 As in the first case (I),  $g_0'g_i$ and $h'= \phi_{i+1}\circ\cdots\circ \phi_m(h)$ generate a $\mathbb{Z}\oplus \mathbb{Z}$ in $W(P_Q,\mathcal{F}_B)$. Hence we can find  an $\square$ in $Q$.
 If $$x=g_0 t_1\cdots t_i \cdot g_i \cdot t_{i+1} \cdots t_m=g_0't_1\cdots t_jg''_0 t_{j+1}\cdots t_i \cdot g'''_0g_i \cdot t_{i+1} \cdots t_m$$ where $g''_0$ cannot cross $t_{j+1}$ and $g_0=g_0'\cdot \phi_1\circ\cdots\circ\phi_j(g_0'')\cdot \phi_1\circ\cdots\circ\phi_i(g_0''')$.
  As in the second case (II), there is a generator $s_4$ in $g''_0$ which is not in $S(W_{{B'_{j+1}}})$. Then $s_4$ determines a facet $F_4$ of $P_Q$ such that $F_4$ intersects with $F_1$ and $F_3$ in $Q$. So there is an $\square$-belt in $Q$.
\vskip.2cm
Together with  all arguments above, we  complete the proof.
$\hfill\Box$

\section{Applications}\label{section7}
Throughout the following, we always assume that $Q$ is a genus $\mathfrak{g}$ simple handlebody with $m$ facets and $M$ is the manifold double over $Q$.
 In this section, we shall show that some $B$-belts in $Q$ can play a role in the obstruction of the existence of Riemannian metrics on $M$.
First we can see that every $B$-belt of  $Q$ is a {\em $\pi$-injective} suborbifold in the sense of the following \hyperref[proposition-61]{Lemma 7.1}.

\begin{lem}\label{proposition-61}
Let $i:B\hookrightarrow Q$ be a belt of $Q$. Then  $i_*:\pi_1^{orb}(B)\rightarrow\pi_1^{orb}(Q)$ is an injection.
Moreover, if $B$ is not orbifold-aspherical, then $Q$ is not orbifold-aspherical.
\end{lem}
\begin{proof}
If $B$ is a cutting belt of $Q$, then by \autoref{lemma-34}, $B$ is $\pi_1$-injective.

\vskip .2cm If $B$ is a belt of $Q$ which is disjoint with any cutting belts of $Q$, then $B$ can be embedded into $P_Q$, so the induced map $\pi_1^{orb}(B)\rightarrow \pi_1^{orb}(P_Q)$ is an injection. Thus $i_*:\pi_1^{orb}(B)\rightarrow \pi_1^{orb}(P_Q)\rightarrow \pi_1^{orb}(Q)$ is an injection.

\vskip.2cm
If $B$ intersects with some cutting belts, then we can always do some deformation on $B$  such that it intersects transversely with those cutting belts. Thus, we may assume that  $B$ is split into $B_1,B_2,\cdots,B_k$ by  cutting belts $E_1,E_2,\cdots,E_{k-1}$ such that
for each $i$, $B_i$ and $B_{i+1}$ exactly intersect with $E_i$ since all $E_i$ are simple polytopes and $|B|$ is a ball.
In addition, it is also easy to see  that for each $i$, $\pi_1^{orb}(B_i)$ is a right-angled Coxeter group.     Now
$$
\pi_1^{orb}(B)=\pi_1^{orb}(B_1)*\pi_1^{orb}(B_2)*\cdots *\pi_1^{orb}(B_k)/\langle R(E_i)|i=1, ..., k-1\rangle
$$
where $R(E_i)$ is a relation set consisting of all equations $s=t$, each of which is associated with  $F_s\sim_{E_i} F_t$ where  $F_s\in\mathcal{F}(B_{i})$ and $F_t\in\mathcal{F}(B_{i+1})$.

\vskip.2cm

Now for $g\in \pi_1^{orb}(B_i)$, $i_*(g)=t_{i-1}^{-1}\cdots t_1^{-1}gt_1\cdots t_{i-1}$. Define $\kappa:\pi_1^{orb}(Q)\rightarrow \pi_1^{orb}(B)$,
$$
\kappa(s)=\begin{cases}
s, ~~~&for~F_s\in \cup\mathcal{F}(B_i)\\
1, ~~~&for~F_s\in \mathcal{F}(P_Q)-\cup\mathcal{F}(B_i)\\
\end{cases}
$$
where all torsion free generators are mapped to $1$.
Then it is clear that $\kappa$ is well-defined and $\kappa\circ i_*=id$. Hence, $i_*$ is an injection.

\vskip.2cm

Let $\widetilde{Q},\widetilde{B}$ be the universal cover of $Q,B$, respectively.
If $B$ is not orbifold-aspherical, then there is an integer $k\geq 2$ such that $\pi_k(\widetilde{B})\neq 0$ and $\pi_i(\widetilde{B})= 0$  for any $1\leq i<k$. By Hurewize theorem, $H_k(\widetilde{B})\cong\pi_k(\widetilde{B})\neq 0$.
Hence by \autoref{thm4} there is an $\Delta^k$-belt in $B$. So there is an $\Delta^k$-belt of $Q$. Hence $Q$ is not orbifold-aspherical.
\end{proof}

\begin{lem}\label{or}
The manifold double of  a  simple handlebody is orientable.
\end{lem}
\begin{proof}
Let $Q$ be a simple $n$-handlebody  with $m$ facets and cutting belts $\mathcal{F}_B$,  $P_Q$ be the associated simple polytope, and $M=Q\times (\mathbb{Z}_2)^m/\sim$ be the manifold double over $Q$, as defined in (\ref{E1}).
 It suffices to prove that
  $H_n(M; \mathbb{Z})\cong \mathbb{Z}.$
We shall   use the proof method of Nakayama and Nishimura \cite[Theorem 1.7]{NN} here.
  The combinatorial structure of $P_Q$ defines a natural cellular decomposition of $M$. We denote by $\{(C_k(M),\partial_k)\}$ the chain complex associated with this cellular decomposition. In particular, $C_n(M)$ and $C_{n-1}(M)$ are the free abelian groups generated by $\{P_Q\}\times (\mathbb{Z}_2)^m=\{(P_Q,g)\mid g\in (\mathbb{Z}_2)^m\}$ and $\mathcal{F}(P_Q)\times (\mathbb{Z}_2)^m/\sim'=\{[F,g]\mid F\in \mathcal{F}(P_Q), g\in(\mathbb{Z}_2)^m \}$, respectively, where the equivalence class of $\mathcal{F}(P_Q)\times (\mathbb{Z}_2)^m$ is defined by the equivalence relation
$$\begin{cases}
(F,g)\sim' (F,g\cdot e_F) & \text{ if } F\in \mathcal{F}(P_Q)-\mathcal{F}_B,\\
(B^+,g)\sim' (B^-, g) & \text{ if ~} B^+, B^-\in \mathcal{F}_B.
\end{cases}$$
It should be pointed out that actually  there is a coloring $\lambda: \mathcal{F}(Q)\longrightarrow (\mathbb{Z}_2)^m$ in the construction of  $M=Q\times (\mathbb{Z}_2)^m/\sim$ such that
$\{\lambda(F)=e_F | F\in \mathcal{F}(Q)\}$ is the standard basis $\{e_i| i=1, ..., m\}$ of $(\mathbb{Z}_2)^m$. For any facet $F'$ in $\mathcal{F}(P_Q)-\mathcal{F}_B$, there must be a facet $F$ in $\mathcal{F}(Q)$ such that $F'=F$ or $F'\subsetneqq F$, so $F'$ and $F$ are colored by the same element $e_F$ of $(\mathbb{Z}_2)^m$. For any $B$ in $\mathcal{F}_B$, since $\text{Int}B\subset \text{Int}(Q)$, we convention that $B$ is colored by the unit element ${\bf e}_0$ of $(\mathbb{Z}_2)^m$.
In other words, the coloring $\lambda: \mathcal{F}(Q)\longrightarrow (\mathbb{Z}_2)^m$ induces a compatible coloring $\lambda':  \mathcal{F}(P_Q)\longrightarrow (\mathbb{Z}_2)^m$ such that for any $F' \in \mathcal{F}(P_Q)-\mathcal{F}_B$,  $e_{F'}=\lambda'(F')=\lambda(F)=e_F$ where $F\in \mathcal{F}(Q)$ with $F'\subset F$,
and for $B$ in $\mathcal{F}_B$, $\lambda'(B)={\bf e}_0$.
\vskip .2cm

We give an orientation on each facet $F_i$ and $B^\pm_i$ such that the orientation of $B_i^+$ is exactly the inverse orientation of $B_i^-$, so
\begin{equation*}
\begin{split}
\partial P_Q &= \sum_{F\in \mathcal{F}(P_Q)}F=F_1+\cdots+F_{m'}+B^+_1+\cdots+B^+_\mathfrak{g}+B^-_1+\cdots+B^-_\mathfrak{g}\\
&=\sum_{F\in \mathcal{F}(P_Q)-\mathcal{F}_B}F\\
\end{split}
\end{equation*}
 where $m'$ is the number of all facets in $\mathcal{F}(P_Q)-\mathcal{F}_B$.

 \vskip .2cm

Let $c_n=\sum_{g\in(\mathbb{Z}_2)^m} n_g(P,g)$ be an $n$-cycle of $C_n(M)$ where $n_g\in \mathbb{Z}$. Then
\begin{equation*}
\begin{split}
\partial(c_n)&=[\sum_{g\in(\mathbb{Z}_2)^m}n_g \sum_{F\in \mathcal{F}(P_Q)}(F,g)]=\sum_{[F,g]\in ((\mathcal{F}(P_Q)-\mathcal{F}(B))\times (\mathbb{Z}_2)^m)/\sim'}(n_g+n_{ge_F})[F,g]=0
\end{split}
\end{equation*}
 which induces that $n_g=-n_{ge_F}$ for any facet $F\in \mathcal{F}(P_Q)-\mathcal{F}_B$ and $g\in( \mathbb{Z}_2)^m$. Let $l(g)$ denote the word length of $g$ presented by $\{e_F\}$. For any $g\in (\mathbb{Z}_2)^m$,
there exists a subset $\mathcal{I}_g=\{F_{i_1}, ..., F_{i_k}\}$ of $\mathcal{F}(P_Q)-\mathcal{F}_B$ such that $g=\prod_{F\in \mathcal{I}_g}e_F$.
Then we see easily that
$$n_g=-n_{ge_{F_{i_1}}}=n_{ge_{F_{i_1}}e_{F_{i_2}}}=\cdots=(-1)^{l(g)}n_{g\prod_{F\in \mathcal{I}_g}e_F}=(-1)^{l(g)}n_{{\bf e}_0}$$
so $c_n=n_{{\bf e}_0}\sum_{g\in(\mathbb{Z}_2)^m} (-1)^{l(g)}(P,g)$.
 Then we obtain that $H_n(M;\mathbb{Z})=\ker \partial_n\cong \mathbb{Z}$ is generated by $\sum_{g\in(\mathbb{Z}_2)^m}(-1)^{l(g)}(P,g)$, which follows that  $M$ is orientable.
\end{proof}

\subsection{Nonpositive curvature}\label{section-61}
Recall that a geodesic metric space $X$ is  {\em non-positively curved} if it is a locally CAT(0) space. In general, $X$ is said to be of {\em curvature $\leq k$} (in the sense of Alexandrov) if it is locally a CAT($k$) space. See \cite[1.2 Definition, Page-159]{BH}.
\begin{thm}[{\cite[1A.6 Theorem, Page-173]{BH}}]
A smooth Riemannian manifold $N$ is of curvature $\leq k$ in the sense of Alexandrov if and only if the sectional curvature  of $N$ is $\leq k$.
\end{thm}
\begin{prop}[{\cite[Corollary 6.2.4]{P16}}]If $(N,g)$ is a complete Riemannian manifold of nonpositive
curvature, then the fundamental group is torsion free.
\end{prop}

Let $Q$ be a simple handlebody of dimension $n\geq 3$ and genus $\mathfrak{g}$, and $M\longrightarrow Q$ be the manifold double over $Q$, as defined in (\ref{E1}).
Let  $P_Q$  be the simple polytope  obtained from $Q$ by cutting open along $\mathfrak{g}$ disjoint cutting belts $B_1, ..., B_\mathfrak{g}$ in $Q$.
More precisely, $P_Q$ can be obtained in the following way: For each belt $B_i$,  choose  a regular neighborhood $N(B_i)$ of  $B_i$ that is homeomorphic to $B_i\times [-1,1]$ as manifolds with corners. Clearly  $N(B_i)$ is identified with a simple polytope, and it can also be understood  as the disk $D^1$-bundle of the trivial normal bundle of $B_i$ in $Q$. Then we get $P_Q$ by removing the interiors of trivial $D^1$-bundles $B_i\times [-1,1]$ of all $B_i$.
\vskip .2cm
In order to use Gromov Lemma as above, we need  a cubical cellular structure of the manifold double $M$ over $Q$. For this,  we perform the following procedure:
 \begin{itemize}
\item[$(1)$] First we decompose $Q$ into more pieces
$$Q=P_Q \bigcup_{i=1}^\mathfrak{g} N^+(B_i)\cup N^-(B_i)$$
where $ N^+(B_i)=B_i\times [0,1]$ and  $N^-(B_i)=B_i\times [-1,0]$ satisfy  $N(B_i)=N^+(B_i)\cup N^-(B_i)$.
 \item[$(2)$]
Next, the standard cubical decompositions of $P_Q$ and all $N^{\pm}(B_i)$ determine a right-angled Coxeter cubical  cellular decomposition of $Q$, denoted by $\mathcal{C}(Q)$. Specifically, all cone points of  $P_Q$ and all $N^{\pm}(B_i)$ will be $0$-cells with trivial local group in $\mathcal{C}(Q)$.
There are two kinds of $k(>0)$-cubes in the cubical decompositions of $P_Q$ and all $N^{\pm}(B_i)$, each of which either intersects transversely  with an $(n-k)$-face $f^k=F_{i_1}\cap \cdots\cap F_{i_k}$ or intersects transversely  with an $(n-k)$-face $f^k=F_{i_1}\cap \cdots\cap F_{i_{k-1}}\cap B_i$.
The first type of cubes determine right-angled Coxeter cubical  cells of the form $e^k/(\mathbb{Z}_2)^k$ in $\mathcal{C}(Q)$, and the second type of cubes determine right-angled Coxeter cubical  cells of the form $e^k/(\mathbb{Z}_2)^{k-1}$.
Then, $\mathcal{C}(Q)$ is obtained by attaching each pair associated with $B_i$ of the second type of cubes together.
It is clear that $\mathcal{C}(Q)$ is a right-angled Coxeter cubical  cellular decomposition of $Q$.
\item[$(3)$] Finally, by pulling back $\mathcal{C}(Q)$ to $M$ via the covering map $p:M\longrightarrow Q$, one can obtain a cubical cellular decomposition of $M$, denoted by $\mathcal{C}(M)$, such that each cube in $\mathcal{C}(M)$ is  a connected component of  $p^{-1}(c)$ for $c$ in $\mathcal{C}(Q)$.
   In particular, all vertices in $\mathcal{C}(M)$ exactly consist of the liftings of all cone points in $\mathcal{C}(Q)$.
 \end{itemize}

 \begin{lem}\label{LP}
 Let $v$ be a vertex  in $\mathcal{C}(M)$. Then Lk$(v)$ in $\mathcal{C}(M)$ is combinatorially isomorphic to one of nerves
   $\mathcal{N}(P_Q)$,  $\mathcal{N}(N^{+}(B_i))$ and $\mathcal{N}(N^{-}(B_i))$.
 \end{lem}
 \begin{proof}
 In fact, if $p(v)$ is the cone point of $P_Q$, then each $k(>0)$-cube adjacent  to $v$ gives a $(k-1)$-simplex in  $\mathcal{N}(P_Q)$, which corresponds to an $(n-k)$-face of $P_Q$. Therefore, Lk$(v)\cong \mathcal{N}(P_Q)$.
 A same argument can be applied to the case where $p(v)$ is the cone point of $N^{+}(B_i)$ or $N^{-}(B_i)$.
 \end{proof}
 \begin{prop}\label{npc}
  $M$ is non-positively curved if and only if  $Q$ is flag.
 \end{prop}
 \begin{proof}
 Let $\mathcal{C}(M)$ be the cubical cellular decomposition of $M$ discussed above.
Gromov Lemma (\autoref{LMG}) tells us that  $M$ is non-positively curved if and only if the link of each vertex in the cubical cellular decomposition of $M$ is flag.
By \autoref{LP}, the latter of the above statement means that $\mathcal{N}(P_Q)$ and all $\mathcal{N}(N^{\pm}(B_i))$ are flag, so $P_Q$ and all $N^\pm(B_i)$ are flag simple polytopes. This is also equivalent to  saying that $Q$ is flag. Thus,  if a simple handlebody $Q$ is flag, then its manifold double $M$  is non-positively curved.
\vskip.1cm
Conversely, if $M$ is non-positively curved, then by Cartan-Hadamard Theorem, $M$ is aspherical. Then $Q$ is orbifold-aspherical. By \hyperref[DJS-small]{Theorem A}, $Q$ is flag.
\end{proof}
\subsection{Strictly negative curvature}
\begin{prop}[Preissmann, 1943 {\cite[ Theorem 6.2.6]{P16}}]\label{proposition-63}
If $(N,g)$ is a compact manifold of negative
curvature, then any abelian subgroup of the fundamental group is cyclic.
\end{prop}
\begin{prop}[Gromov {\cite{G}}]\label{proposition-64}
Let $\mathcal{C}$ be a cubical complex. Suppose  that the link of every vertex in  $\mathcal{C}$ is flag and contains no $\square$. Let us give $\mathcal{C}$ a $(N,-\epsilon)$ geometry, where each cube in $\mathcal{C}$ is isomorphic to the unit cube in the hyperbolic space of curvature $-\epsilon$. If $\epsilon$ is sufficiently small then $K(\mathcal{C})\leq -\epsilon$.
\end{prop}
\begin{lem}\label{lemma-64} Let $Q$ be a simple handlebody with $m$ facets, and $M$ be the manifold double over $Q$. Then
 $\mathbb{Z}\oplus\mathbb{Z}<\pi_1(M)$ if and only if  $\mathbb{Z}\oplus\mathbb{Z}<\pi^{orb}_1(Q)$.
\end{lem}
\begin{proof} If $\mathbb{Z}\oplus\mathbb{Z}<\pi_1(M)$, then the short exact sequence
$$1\longrightarrow \pi_1(M)\longrightarrow \pi_1^{orb}(Q)\overset{\lambda}{\longrightarrow} (\mathbb{Z}_2)^m \longrightarrow 1$$
 induces that  $\mathbb{Z}\oplus\mathbb{Z}<\pi_1(M)<\pi^{orb}_1(Q)$. Conversely, if $\mathbb{Z}\oplus\mathbb{Z}<\pi^{orb}_1(Q)$, then by \hyperref[FTT]{Theorem B}, there is an $\square$-belt in $Q$ and $\mathbb{Z}\oplus\mathbb{Z}<\pi^{orb}_1(Q)$ is generated by two pairs of disjoints facets in $\square$. Denote the generators of $\mathbb{Z}\oplus\mathbb{Z}$ by $x=s_1s_3$ and $y=s_2t_1t_2\cdots t_k s_4 t_k^{-1}\cdots t^{-1}_2t^{-1}_1$ where each $t_i$ is one of $\{t^{\pm 1}_B\}$, then $\lambda(x^2)=\lambda(y^2)=1$. Hence $x^2,y^2\in \ker \lambda\cong \pi_1(M)$. So  $\mathbb{Z}\oplus\mathbb{Z}<\pi_1(M)$.
\end{proof}

\begin{prop}\label{nc}
If $M$ admits a strictly negative curvature, then $Q$ is flag and contains no any $\square$-belt. Specially, if $Q$ is simple polytope, then $M$ admits a strictly negative curvature if and only if $Q$ is flag and contains no any $\square$-belt.
\end{prop}
\begin{proof}
If $M$ admits a strictly negative curvature, then by \autoref{npc}, $Q$ is flag. Furthermore, if there is an $\square$-belt in $Q$, then by  \hyperref[FTT]{Theorem B}, there is a rank-two abelian subgroup in $\pi_1^{orb}(Q)$. By \autoref{lemma-64}, there is a subgroup $\mathbb{Z}\oplus \mathbb{Z}$ in $\pi_1(M)$. By Preissmann's result (\autoref{proposition-63}), $M$ cannot admit a strictly negative curvature. Hence, $Q$ contains no any $\square$-belt.
\vskip.1cm
When $Q$ is a simple polytope, then the standard cubical cellular decomposition of $Q$ induces a cubical cellular decomposition of $M$, denoted by $\mathcal{C}(M)$, such that each point $v\in \mathcal{C}(M)$ has link $\mathcal{N}(Q)$. By a result of Gromov (\autoref{proposition-64}), if $\mathcal{N}(Q)$ is flag and contains no any $\square$-belt, then $M$ admits a strictly negative curvature.
\end{proof}
\begin{rem}
We are inclined to think that $M$ admits a strict negative curvature if a simple handlebody $Q$ is flag and contains no any $\square$-belt. But we cannot find a suitable cubical decomposition of $M$ so as to use Gromov's result. Of course, this is related to the weak hyperbolization conjecture: ``Let $N$ be a closed aspherical manifold. Then either $\pi_1(N)$ contains $\mathbb{Z}\oplus\mathbb{Z}$ or $\pi_1(N)$ is Gromov-hyperbolic". See \cite[Conjecture 20.12]{KAP}.
\end{rem}
\subsection{Hyperbolic curvature}

When $Q$ is a simple $3$-polytope, Pogorelov Theorem, which is right-angled case of Andreev Theorem \cite{And71,RHD},  states that $Q$ admits a right-angled hyperbolic structure if and only if it is flag and contains no any $4$-belt, where ``right-angled'' means that all dihedral angles are $\frac{\pi}{2}$. This  gives a combinatorial equivalent  description of the hyperbolicity of simple $3$-polytopes as a right-angled Coxeter orbifold.  Now $Q$ is also called a (right-angled) hyperbolic polyhedra in $\mathbb{H}^3$.
\vskip.2cm

 As a generalization of  $3$-dimensional hyperbolic polyhedra, a $3$-manifold with corners is {\em hyperbolic} if  its interior admits a hyperbolic metric which can be extended to the boundary such that its all faces are totally geodesic (or locally convex). Moreover we say that a $3$-manifold with corners is {\em right-angled hyperbolic} if its all dihedral angles are $\frac{\pi}{2}$.

\vskip.2cm

Notice that a hyperbolic $3$-manifold with corners is not right-angled in general. A  $3$-manifold with corners can be equipped with many different orbifold structures. A hyperbolic structure on a  $3$-manifold with corners should be compatible with an orbifold structure on it. Hence there may be different hyperbolic structures on a $3$-manifold with corners. However, the hyperbolic structure of a hyperbolic closed  $3$-orbifold (or 3-manifold) is unique by Mostow Rigidity Theorem~\cite{M73}.
 The hyperbolization of  $3$-manifold with corners corresponds to the generalization of Andreev Theorem~\cite{And71,RHD}.
 This question is still open now.
\vskip.2cm
Here we mainly consider the right-angled hyperbolicity of  simple $3$-manifolds with corners, where a simple 3-manifold with corners is given by forgetting the orbifold structure on a simple 3-orbifold.  Then  one can obtain  the same  understanding for  right-angled hyperbolicity from the following two geometric objects:
\begin{itemize}
\item[(1)] A right-angled hyperbolic simple $3$-manifold with corners;
\item[(2)] A hyperbolic simple 3-orbifold (as a  right-angled Coxeter $3$-orbifold).
\end{itemize}
Thus, a right-angled hyperbolic simple $3$-manifold with corners is a hyperbolic simple 3-orbifold, and vice versa.

\vskip.2cm

\begin{prop}\label{hc}
Let $Q$ be a  3-dimensional simple handlebody. Then
$M$ is hyperbolic if and only if $Q$ is flag and contains no any $\square$-belt.
\end{prop}
\begin{proof}
With Perelman's work together, Thurston's Hyperbolization Theorem implies that a closed orientable $3$-manifold is hyperbolic if and only if it is aspherical and atoroidal, where if there is no subgroup $\mathbb{Z}\oplus\mathbb{Z}$ in $\pi_1(M)$, then $M$ is atoroidal.  By \autoref{or}, $M$ is always orientable. Together with \hyperref[DJS-small]{Theorem A}, \hyperref[FTT]{Theorem B} and \autoref{lemma-64}, we know that $M$ is aspherical and atoroidal if and only if $Q$ is flag and contains no any $\square$-belt.
\end{proof}
By $3$-dimensional hyperbolic manifold theory (or see \cite{O98}), $M$ is hyperbolic if and only if $Q$ is hyperbolic. Hence, we have
\begin{cor}
$Q$ is right-angled hyperbolic as a $3$-manifold with corners if and only if $Q$ is flag and contains no any $\square$-belt.
\end{cor}





For higher dimensional simple handlebody, the fundamental group of a closed hyperbolic manifold is a discrete convex-cocompact subgroup of $\text{Isom}(\mathbb{H}^n)$, and contains no subgroup $\mathbb{Z}\oplus \mathbb{Z} $. Hence, if $M$ is hyperbolic, then $Q$ must be flag and contains no any $\square$-belt. By a result in \cite[Corollary 6.11.6]{D2} there must exist a triangle or quadrilateral face in a simple polytope of dimension greater than  $4$. So
\begin{itemize}
\item A simple handlebody with  dimension greater than $4$ cannot be hyperbolic.
\end{itemize}
In 4-dimension case, it is not clear whether only dodecahedron and  $120$-cell are  hyperbolic simple $4$-polytopes, see \cite[Page 115]{D2} or \cite{GS}.
\subsubsection{$3$-handlebodies with simplicial  nerve}
A {\em $3$-handlebody with simplicial  nerve} is just a simple $3$-orbifold whose underlying space is a $3$-handlebody.
Let $Q$ be a $3$-handlebody with simplicial  nerve and with genus $\mathfrak{g}>0$. Next we show that $Q$ can always be cut into a simple polytope along some
codimension-one B-belts.

\begin{lem}\label{lemma-1}
A $3$-handlebody with simplicial nerve is a simple $3$-handlebody.
\end{lem}
\begin{proof}
Let $Q$ be a  $3$-handlebody of genus $\mathfrak{g}\geq 0$ with simplicial nerve. If $\mathfrak{g}=0$, then it is easy to check that $Q$ is
a simple $3$-polytope by  Steinitz Theorem and \cite[Proposition 3.4]{L}.
 If $\mathfrak{g}>0$,  let $$\{(D_i^{2}, \partial D_i^{2})\hookrightarrow (|Q|, \partial |Q|)\mid i=1,2,\cdots,\mathfrak{g}\}$$ be some disjoint compressing $2$-disks in $|Q|$  such that $|Q|$ is cut into a connected $3$-ball along those compressing $2$-disks. Considering the facial structure determined by the triangulation $\mathcal{N}(Q)$ of $\partial |Q|$, we can always do some slight deformations for the boundaries of  compressing $2$-disks on  faces of $Q$, so that $\{D_i^{2}\}$ can be modified into  some embedded sub-orbifolds $\{B_i \}$ of preserving codimension in $Q$. Each $B_i$ is a polygon.
\vskip.1cm

 Given a $B_i$, by the definition of $B$-belts, we see that $B_i$ is not an $B$-belt if and only if there must exist two non-adjacent edges  $f_1$ and $f_2$ in $B_i$ such that
 \begin{itemize}
\item[(i)] $F_{f_1}\cap F_{f_2}\neq \emptyset$ (probably $F_{f_1}$ and $F_{f_2}$ can even be the same face of $Q$);
\item[(ii)] $F_{f_1}\cup F_{f_2}$ can deformatively retract onto  $B_i$ in $|Q|$ (in fact,  $F_{f_1}\cup F_{f_2}$ can deformatively retract onto $\partial B$ in $\partial |Q|$).
\end{itemize}
where  $F_{f_1}$ and $F_{f_2}$ are two 2-faces of $Q$ that contain $f_1, f_2$ respectively.
So, if $B_i$ is not an $B$-belt, then there is no any hole
in the area $A$ in $\partial |Q|$ surrounded by $F_{f_1}, F_{f_2}$ and $B_i$.
Then, we can modify the boundary of $B_i$ by pushing the retract of $F_{f_1}\cup F_{f_2}$ into $F_{f_1}$, $F_{f_2}$ and throwing some edges of $B_i$ away, as shown in ~\hyperref[f3]{Figure 10},
\begin{figure}[h]\label{f3} \centering \def\svgwidth{0.5\textwidth}
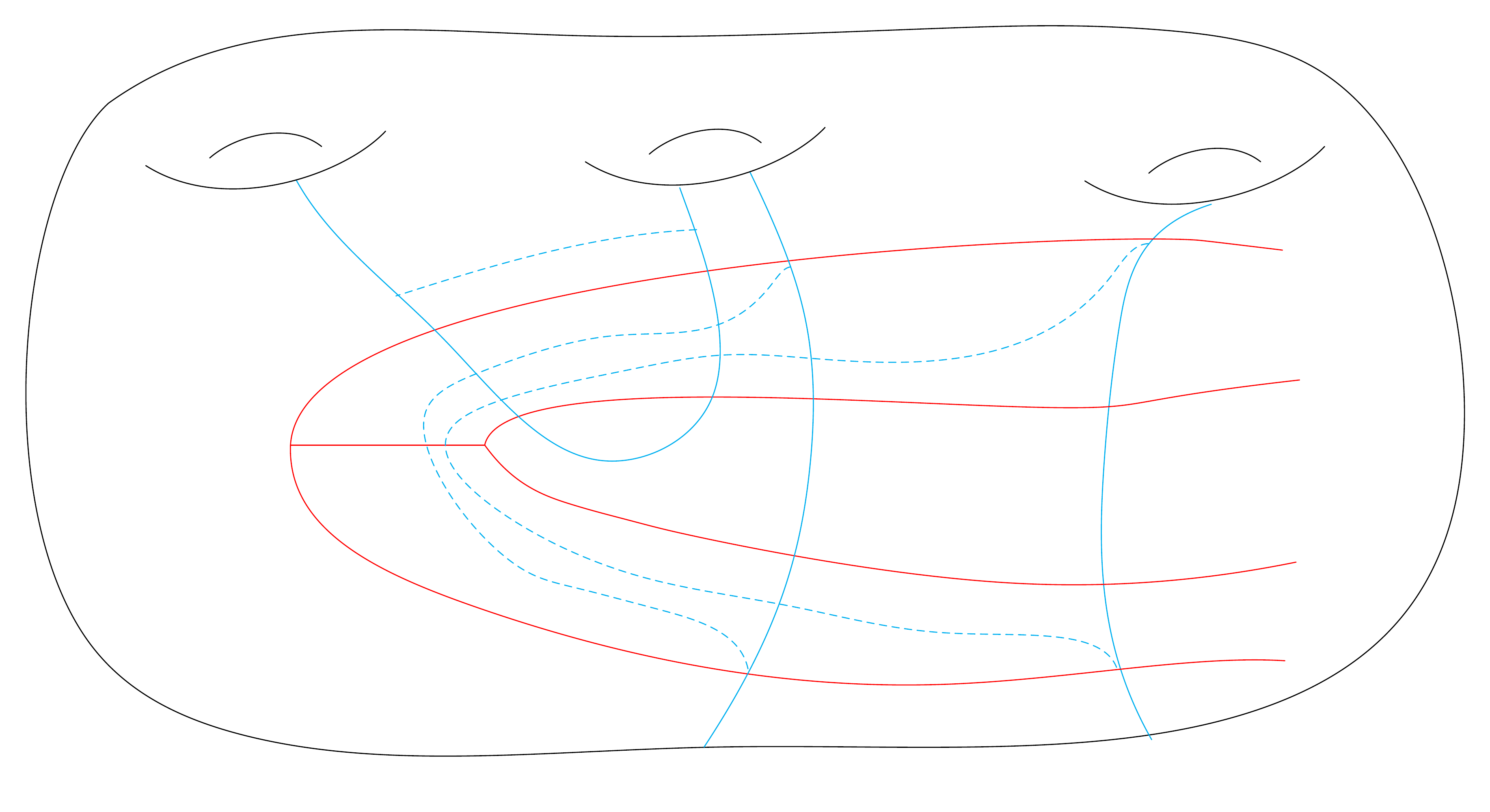 \caption{ Modifying the boundary of $B_i$. }
\end{figure}
 so that one can obtain a new $B'_i$ with fewer edges which  intersects transversely with  $F_{f_1}\cap F_{f_2}$. In particular, if $F_{f_1}=F_{f_2}$, then $f_1$ and $f_2$ will become the same edge in $B'_i$, and if $F_{f_1}\not=F_{f_2}$ then $f_1$ is adjacent to $f_2$ in $B'_i$.
In addition, if there is also another  sub-orbifold $B_j$ which intersects with the area $A$ in $\partial |Q|$, this  means that $B_j$ is not a belt, too. The above "pushing" process will move the boundary of $B_j$ out from the area $A$ and modify $B_j$ into $B'_j$ with fewer edges such that  $B'_i\cap B'_j=\emptyset$.
 Since $B_i$ is a polygon with finite edges, this process can end after a finite number of steps until one has modified $B_i$ into an  $B$-belt  which does not intersect with other $B_j$.
\vskip.1cm
We can perform the same procedure to other non $B$-belts in $\{B_j\}_{j\not=i}$. Finally one can obtain a set of disjoint cutting belts  such that $Q$ is cut open into a simple $3$-polytope along those cutting belts, implying that $Q$ is a simple $3$-handlebody.
\end{proof}
Hence,  \autoref{hc} still holds for $3$-handlebodies with simplicial  nerve.
\subsubsection{$3$-handlebodies with  ideal  nerve}\label{s-63}
We say that $Q$ is  a {\em $3$-handlebody with ideal nerve} if $Q$ is a right-angled Coxeter 3-orbifold such  that its underlying space $|Q|$ is a 3-handlebody and
  its nerve  is  an ideal triangulation of the boundary $\partial |Q|$.
\vskip.2cm

Now let $Q$  be a  $3$-handlebody with ideal nerve. Then, by the definition of ideal triangulation (\cite[Definition 2.6]{FST08}),  the interior of each face of $Q$ is also contractible.
  On the facial structure of $Q$, there are three possible cases:
\begin{itemize}
\item Some 2-faces of $Q$ are henagons (i.e., 2-faces with only one point of codimension 3 in $Q$, see (a) in \hyperref[Figure-100]{Figure 11}) or digons (i.e., 2-faces with only two points of codimension 3 in $Q$, see (b) in \hyperref[Figure-100]{Figure 11});
\item There may be some 2-faces with self-intersection (see (c) in \hyperref[Figure-100]{Figure 11});
\item The intersection of two 2-faces may be not connected (see (d) in \hyperref[Figure-100]{Figure 11}).
\end{itemize}

\begin{figure}[h]\label{Figure-100} \centering \def\svgwidth{0.6\textwidth} 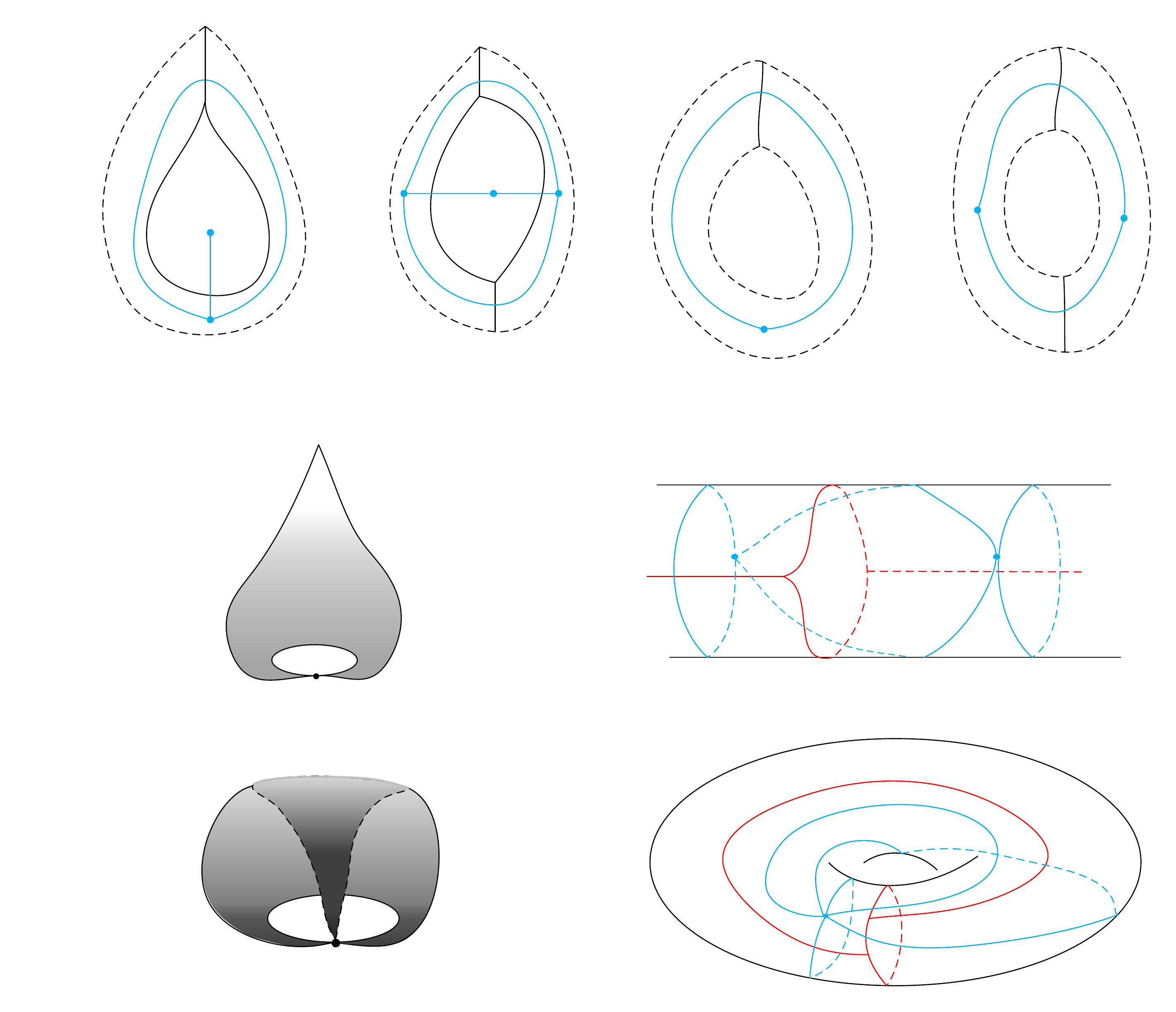 \caption{Ideal nerves.} \end{figure}

If there is a henagon 2-face of $Q$, then it gives a self-folded ideal triangle. For example, see the blue part of (a) in \hyperref[Figure-100]{Figure 11}.
In general, if there is a henagon 2-suborbifold in $Q$, then the nerve of associated faces may give some ideal triangles,  such as (e) and (f)  in \hyperref[Figure-100]{Figure 11}. In particular, the nerve of (f)   contains only one vertex and two ideal triangles glued along their three edges as shown in \hyperref[Figure-100]{Figure 11}.
All those cases agree with the definition of ideal triangulations in \cite[Definition 2.6]{FST08}.

\begin{lem}\label{Lemma-62}
Let $Q$ be a  $3$-handlebody with ideal nerve. Then $Q$ is very good if and only if it does not contain a henagon 2-suborbifold.
\end{lem}
\begin{proof}
By applying a theorem of Morgan or Kato~\cite[Theorem 6.14]{KAP}, each compact locally reflective $3$-orbifold that contains no bad  $2$-suborbifolds is very good. This means that if there is no  henagon 2-suborbifold in $Q$, then $Q$ is very good. Conversely, if there is a henagon 2-suborbifold in $Q$, then it is obvious that $Q$ is bad.
\end{proof}
Hence, if there is a henagon 2-suborbifold of $Q$, then $Q$  cannot be hyperbolic.

\vskip .2cm
Suppose that $Q$ contains no henagon 2-suborbifolds. Then, by \autoref{Lemma-62}, $Q$ can be covered finitely by a closed $3$-manifold $M$.
In general, $Q$ is not nice in the sense of Davis \cite[Page 180]{D2}, thus there is no natural manifold double defined as in (\ref{E1}) for $Q$.
\vskip.2cm
A digon $2$-suborbifold in $Q$ is said to be
{\em essential} if its two vertices are not contained in a unique edge of $Q$.
If there is an {essential}
digon 2-suborbifold in $Q$, then its nerve $\mathcal{N}(Q)$ will  contain two simplices  with  common vertices. See (d) in \hyperref[Figure-100]{Figure 9}.

\begin{lem}
Let $Q$ be a  $3$-handlebody with ideal nerve. Assume that there is no henagon suborbifold in $Q$, and $M$ is a covering manifold over $Q$. If $Q$ contains an {essential}  digon suborbifold,   then $M$ is reducible.
\end{lem}
\begin{proof}
Assume that two edges of a digon are contained in two 2-faces $F_1$ and $F_2$ of $Q$. Then we consider the double cover of $Q$, denoted by $D_Q$,  which is obtained by gluing two copies of $Q$ along $F_1$. At the same time, two copies of $F_2$
are also glued along $F_1\cap F_2$, giving  an annulus in $D_Q$.  Let $M'$ be a  manifold double over $D_Q$. Then $M'$ can be decomposed into the connected sum of some $3$-manifolds, which implies that $M'$ is reducible. Hence, $D_Q$ and $Q$ are reducible. So $M$ is reducible.
\end{proof}

 A digon 2-face of $Q$ will give an essential digon 2-suborbifold in $Q$ unless that $Q$ is a trihedron. Thus in this cases $Q$ is reducible as well.
Therefore, if there is a  henagon 2-suborbifold or an essential digon 2-suborbifold in $Q$, then $Q$ cannot be hyperbolic.

\vskip .2cm

Next, suppose that $Q$ is not a  trihedron and
contains no henagon and essential digon 2-suborbifolds. If there are   some 2-faces with self-intersection or
the intersection of two 2-faces is not connected,
then  we can always construct some simple orbifold covers of $Q$.
 In fact,  we can use some copies of $Q$ to construct a covering space of $Q$ as follows:  first we cut open each  of copies by using a fixed 2-suborbifold $B$, and then form a connected handlebody $\widehat{Q}$  by attaching them together along those new facets produced by  $B$.
 If necessary, we can choose enough copies of $Q$  so as to make sure that this connected handlebody is simple, and is exactly the required covering space of $Q$. Applying \hyperref[T1]{Theorem A} gives

\begin{cor}\label{bg}
A  $3$-handlebody with ideal nerve   is hyperbolic if and only if  it is not trihedron, tetrahedron and contains no $\triangle^2, \square$-belts and no henagon or essential digon 2-suborbifolds.
\end{cor}

\begin{rem}\label{ideal nerve}
 Let $Q$ be a  $3$-handlebody with ideal nerve.
We can define henagon $2$-suborbifolds and  essential digons $2$-suborbifolds in $Q$ as $1$- and $2$-belts of $Q$, respectively.
 Then by \autoref{Lemma-62}, $Q$ is very good if and only if $Q$ contains no $1$-belts. An easy argument gives that a very good  $Q$ is  flag if and only if  it is not  trihedron and  tetrahedron  (i.e.,  $S^3/(\mathbb{Z}_2)^3$ and $ S^3/(\mathbb{Z}_2)^4$) and contains no $2$- and $3$-belts (i.e.,  $\pi_1$-injective $S^2/(\mathbb{Z}_2)^2$- and $S^2/(\mathbb{Z}_2)^3$-suborbifolds).
Furthermore, a very good flag $Q$ is hyperbolic if and only if it contains no $4$-belts (i.e., $\pi_1$-injective $T^2/(\mathbb{Z}_2)^2$-suborbifolds).
Thus, a right-angled Coxeter $3$-handlebody with ideal nerve except trihedron and  tetrahedron  is hyperbolic if and only if it contains no $1,2,3,4$-belts. 
\end{rem}
\subsubsection{Example of non-simple 3-handlebody}\label{s-64}
Let $P$ be the product of a pentagon and $[0,1]$.
Gluing two opposite pentagons of $P$ together such that its diagonal vertices coincide with each other gives a right-angled
Coxeter 3-orbifold with its underlying space as a solid torus, denoted by $Q$. Then  $Q$ is a Seifert $3$-orbifold. Thus it cannot be hyperbolic. This is because each embedding annulus 2-facet is an obstruction.
\subsection{Positive (scalar) curvature}
A simple polytope $P$ is {\em two-neighborly} if any two facets of $P$ has a nonempty intersection.
So it is clear that $P$ is two-neighborly if and only if it contains no any $I$-belt.
 In addition, we also know from \cite[Proposition 2.1]{SMY} that $P$ is two-neighborly if and only if its manifold double $M$ is simply connected.  Thus, we have that

\begin{lem} \label{neib} Let $P$ be a simple polytope. Then the following statements are equivalent.
\begin{itemize}
\item $P$ is two-neighborly;
\item $P$ contains no any $I$-belt;
\item its manifold double $M$ is simply connected.
\end{itemize}
\end{lem}

\begin{rem}
Similarly, for a simple handlebody $Q$, we may define it to be {\em two-neighborly} if any two vertices in its nerve $\mathcal{N}(Q)$ are connected by an edge.
However,  if the genus of $Q$ is greater than zero, unlike the case of simple polytopes,  the existence of $I$-belts can not be used to detect whether $Q$ is two-neighborly since
we can always find an $I$-belt between two facets with nonempty intersection such that this $I$-belt can not be deformed onto the intersection of two facets in $Q$.
\end{rem}

\begin{lem}
Let $N$ be a triangulable closed $n$-manifold with  $n>1$. If $\pi_1(N)$ is nontrivial, then the $1$-skeleton of any triangulation of $N$ cannot be a complete graph.
\end{lem}
\begin{proof}
 Assume that $K$ is a triangulation of $N$ whose $1$-skeleton $K^1$ is a complete graph. Fixed a vertex $x$ in $K$, let $N(x)$ be the union of those $n$-simplices in $K$ which contain $x$. Then $N(x)\cong D^n$ whose boundary $\partial N(x)$ is a two-neighborly simplicial $(n-1)$-sphere.
\vskip.2cm
Let $\Delta^n$ be an  arbitrary simplex which does not contain $x$. Since $K^1$ is a complete graph, all vertices of $\Delta^n$ are contained in $\partial N(x)$. Hence, $K^1$ is a subcomplex of $N(x)$. So any closed loop in $K^1$ is contractible in $N(x)$.
This means that $N$ is simply connected, giving a contradiction.
\end{proof}
\begin{cor}
A (flag) simple handlebody with genus $>0$ cannot be two-neighborly. In other words, a two-neighborly simple handlebody must be a two-neighborly simple polytope as a manifold with corners.
\end{cor}


\begin{prop}[Hopf-Rinow, Myers, {\cite[Corollary 6.3.2 \& Theorem 6.3.3]{P16}}]\label{post-cur}
If an $n$-dimensional closed manifold $N$ admits a complete Riemannian metric of positive sectional curvature, then
$\pi_1(N)$ is finite. Specially, $\pi_1(N)$ is $0$ or $\mathbb{Z}_2$ if $n$ is even, and $N$ is orientable if $n$ is odd.
\vskip .1cm
If $N$ admits a complete Riemannian metric of positive Ricc curvature, then $\pi_1(N)$ is finite, too.
\end{prop}

 For a simple handlebody $Q$, by the short exact sequence,
$$1\longrightarrow \pi_1(M)\longrightarrow \pi_1^{orb}(Q)\overset{\lambda}{\longrightarrow} (\mathbb{Z}_2)^m \longrightarrow 1$$
if the genus of $Q$ is greater than zero, then any torsion-free generator $t$ determined by a cutting belt is  mapped to $1\in (\mathbb{Z}_2)^m$ via $\lambda$. Hence $t$ gives a torsion-free element in $\pi_1(M)$. So we have
\begin{lem}\label{lem63}
If the genus of $Q$ is greater than zero, then $\pi_1(M)$ is not finite.
\end{lem}
As a direct consequence of \autoref{post-cur}, \autoref{neib} and \autoref{lem63}, we have
\begin{cor}
If $M$ admits a complete Riemannian metric of positive sectional or Ricc curvature, then $Q$ must be a two-neighborly simple polytope, that is, there is no any $I$-belt in $Q$.
\end{cor}


Conversely, the existence of positive sectional curvature and positive Ricc curvature of a closed (or compact) $n$-manifold is a very hard question, which is involved in many conjectures and open questions. For example, we know that the real moment angle manifold over $\Delta^2\times \Delta^2$ is $S^2\times S^2$. However, it is well-known that the  existence of positive sectional curvature on $S^2\times S^2$ is just as Hopf Conjecture.

\vskip.2cm

On the other hand, there are some results about the existence of  positive scalar curvature. One can refer to some works of Gromov-Lawson, Schoen-Yau, Stolz \cite{SY,SY2,GL,GL2,GL3,S}.
By Gromov-Lawson \cite{GL3},  a compact manifold of nonpositive sectional curvature cannot carry a metric of positive sectional curvature. So
\begin{itemize}
\item If $M$ admits a positive scalar curvature, then $Q$ is not flag.
\end{itemize}
Moreover, it is reasonable to conjecture that:
\begin{itemize}
\item[{\textcolor{red}{$\star$}}] {\em If a simple polytope $Q$ is two-neighborly, then $M$ admits a positive scalar curvature.}
\end{itemize}

In 3-dimensional case, Wu-Yu in \cite{WY} gave a combinatorial description for the case of  real moment-angled manifolds with positive scalar curvature over  simple $3$-polytopes.
\begin{prop}[{\cite[Corollary 4.10]{WY}}]\label{PSC}
A  real moment-angle manifold (or small cover) over a simple
3-polytope $P$ can admit a Riemannian metric with positive scalar curvature if and only if
$P$ is combinatorially equivalent to a polytope obtained from $\Delta^3$ by a sequence of vertex
cuts.
\end{prop}
Let $P$ be a simple polytope obtained from $\Delta^3$ by a sequence of vertex cuts. Then except for $P=\Delta^3$, any $2$-dimensional belt in $P$ is an $\Delta^2$-belt. Conversely, assume that every 2-dimensional belt is only $\Delta^2$-belt  in a simple polytope $P$, then it is easy to see that $P$ is a tetrahedron or $P$ can be decomposed into the connected sum of some $\Delta^3$s. This is equivalent to saying that $P$ can be obtained from $\Delta^3$ by a sequence of vertex cuts. Hence, we get an equivalent description of \autoref{PSC} in terms of $\Delta^2$-belts and $\Delta^3$-belts.
\begin{cor}\label{psc}
Let $P$ be a simple $3$-polytope. Then its manifold double $M$ admit a Riemannian metric with positive scalar curvature if and only if every $2$-dimensional belt in $P$ is $\Delta^2$, or $P$ is just a tetrahedron. \end{cor}

 \subsection{}
We summarize what we have discussed in the below table.
\begin{table}[h]
    \centering
    \begin{tabular}{|p{2.4cm}| p{2.5cm}p{1cm}p{2cm}p{6.5cm}|}
        \hline
\diagbox[innerwidth=2.4cm]{$M$}{$Q$} & $2$-Neighborly &Flag&Pogorelov  & Description  \\
        \hline
\vskip 0pt Sec$<0$   &{\Large \begin{turn}{-90}$\Lsh$\end{turn}} {\tiny not} &{\Large \begin{turn}{-90}$\Lsh$\end{turn}}&   {\Large \begin{turn}{-90}$\Lsh$\end{turn}} {\Large \begin{turn}{180}$\Rsh$\end{turn}}{\tiny $\dim 3$} & 
 \autoref{nc}.        \\[12pt]
\vskip 0pt  Hyperbolic   &{\Large \begin{turn}{-90}$\Lsh$\end{turn}} {\tiny not}        &{\Large \begin{turn}{-90}$\Lsh$\end{turn}}&{\Large \begin{turn}{-90}$\Lsh$\end{turn}} {\Large \begin{turn}{180}$\Rsh$\end{turn}}{\tiny $\dim 3$}          & $\dim 3$: \autoref{hc};          $\dim 4$: Not clear; $\dim \geq 5$: None.\\[12pt]
\vskip 0pt Sec$\leq 0$, NPC   & {\Large \begin{turn}{-90}$\Lsh$\end{turn}} {\tiny not}    &{\Large \begin{turn}{-90}$\Lsh$\end{turn}}   {\Large \begin{turn}{180}$\Rsh$\end{turn}}     &{\Large \begin{turn}{180}$\Rsh$\end{turn}}
      &  
      \autoref{npc}       \\[12pt]
\vskip 0pt Flat    &{\Large \begin{turn}{-90}$\Lsh$\end{turn}} {\tiny not}      &{\Large \begin{turn}{-90}$\Lsh$\end{turn}}&{\Large \begin{turn}{-90}$\Lsh$\end{turn}} {\tiny not}          & \cite[Theorem 1.2]{SMY}.           \\[12pt]
\vskip 0pt Spherical    & {\Large \begin{turn}{-90}$\Lsh$\end{turn}}     &{\Large \begin{turn}{-90}$\Lsh$\end{turn}} {\tiny not} &{\Large \begin{turn}{-90}$\Lsh$\end{turn}} {\tiny not}            &   \cite[Theorem 1.2]{SMY}. \\[12pt]
\vskip 0pt Sec, Ric $>0$     &{\Large \begin{turn}{-90}$\Lsh$\end{turn}}       &{\Large \begin{turn}{-90}$\Lsh$\end{turn}}{\tiny not} &{\Large \begin{turn}{-90}$\Lsh$\end{turn}} {\tiny not}           & Not clear (e.g. Hopf Conjecture). \\[12pt]
\vskip 0pt Scalar$>0$   & {{\Large \begin{turn}{180}\color{red}$\Rsh$\end{turn}} {\tiny Conjecture~ {\color{red}$\star$}} }    &{\Large \begin{turn}{-90}$\Lsh$\end{turn}} {\tiny not} & {\Large \begin{turn}{-90}$\Lsh$\end{turn}} {\tiny not}         & Not clear except $\dim 3$.      \\
        \hline
    \end{tabular}
    \caption{A simple handlebody is called {\em Pogorelov} if it is flag and contains no $\square$-belt.}
\end{table}



\end{document}

%% file: G1208.pdf_tex
\begingroup%
  \makeatletter%
  \providecommand\color[2][]{%
    \errmessage{(Inkscape) Color is used for the text in Inkscape, but the package 'color.sty' is not loaded}%
    \renewcommand\color[2][]{}%
  }%
  \providecommand\transparent[1]{%
    \errmessage{(Inkscape) Transparency is used (non-zero) for the text in Inkscape, but the package 'transparent.sty' is not loaded}%
    \renewcommand\transparent[1]{}%
  }%
  \providecommand\rotatebox[2]{#2}%
  \newcommand*\fsize{\dimexpr\f@size pt\relax}%
  \newcommand*\lineheight[1]{\fontsize{\fsize}{#1\fsize}\selectfont}%
  \ifx\svgwidth\undefined%
    \setlength{\unitlength}{283.14349365bp}%
    \ifx\svgscale\undefined%
      \relax%
    \else%
      \setlength{\unitlength}{\unitlength * \real{\svgscale}}%
    \fi%
  \else%
    \setlength{\unitlength}{\svgwidth}%
  \fi%
  \global\let\svgwidth\undefined%
  \global\let\svgscale\undefined%
  \makeatother%
  \begin{picture}(1,0.46384962)%
    \lineheight{1}%
    \setlength\tabcolsep{0pt}%
    \put(0,0){\includegraphics[width=\unitlength,page=1]{G1208.pdf}}%
    \put(0.50823299,0.18230096){\color[rgb]{0,0,0}\makebox(0,0)[lt]{\lineheight{1.25}\smash{\begin{tabular}[t]{l}{\Large{$\simeq$}}\end{tabular}}}}%
    \put(0.18966666,0.18782524){\color[rgb]{0,0,0}\makebox(0,0)[lt]{\lineheight{1.25}\smash{\begin{tabular}[t]{l}$U_1$\end{tabular}}}}%
    \put(0.29186567,0.01288992){\color[rgb]{0,0,0}\makebox(0,0)[lt]{\lineheight{1.25}\smash{\begin{tabular}[t]{l}$U_2$\end{tabular}}}}%
    \put(0.44102105,0.27437212){\color[rgb]{0,0,0}\makebox(0,0)[lt]{\lineheight{1.25}\smash{\begin{tabular}[t]{l}$U_3$\end{tabular}}}}%
  \end{picture}%
\endgroup%

%% file: G2.pdf_tex
\begingroup%
  \makeatletter%
  \providecommand\color[2][]{%
    \errmessage{(Inkscape) Color is used for the text in Inkscape, but the package 'color.sty' is not loaded}%
    \renewcommand\color[2][]{}%
  }%
  \providecommand\transparent[1]{%
    \errmessage{(Inkscape) Transparency is used (non-zero) for the text in Inkscape, but the package 'transparent.sty' is not loaded}%
    \renewcommand\transparent[1]{}%
  }%
  \providecommand\rotatebox[2]{#2}%
  \newcommand*\fsize{\dimexpr\f@size pt\relax}%
  \newcommand*\lineheight[1]{\fontsize{\fsize}{#1\fsize}\selectfont}%
  \ifx\svgwidth\undefined%
    \setlength{\unitlength}{107.65945816bp}%
    \ifx\svgscale\undefined%
      \relax%
    \else%
      \setlength{\unitlength}{\unitlength * \real{\svgscale}}%
    \fi%
  \else%
    \setlength{\unitlength}{\svgwidth}%
  \fi%
  \global\let\svgwidth\undefined%
  \global\let\svgscale\undefined%
  \makeatother%
  \begin{picture}(1,0.7219431)%
    \lineheight{1}%
    \setlength\tabcolsep{0pt}%
\put(0,0){\includegraphics[width=\unitlength]{G2.pdf}}%
    \put(0.10004701,-0.00909856){\color[rgb]{0,0,0}\makebox(0,0)[lt]{\lineheight{1.25}\smash{\begin{tabular}[t]{l}$\Gamma=1$\end{tabular}}}}%
    \put(0.40232503,-0.00290044){\color[rgb]{0,0,0}\makebox(0,0)[lt]{\lineheight{1.25}\smash{\begin{tabular}[t]{l}$\Gamma=\mathbb{Z}_2$\end{tabular}}}}%
    \put(0.72056393,-0.00905161){\color[rgb]{0,0,0}\makebox(0,0)[lt]{\lineheight{1.25}\smash{\begin{tabular}[t]{l}$\Gamma=\mathbb{Z}_2^2$\end{tabular}}}}%
  \end{picture}%
\endgroup%

%% file: G0825.pdf_tex
\begingroup%
  \makeatletter%
  \providecommand\color[2][]{%
    \errmessage{(Inkscape) Color is used for the text in Inkscape, but the package 'color.sty' is not loaded}%
    \renewcommand\color[2][]{}%
  }%
  \providecommand\transparent[1]{%
    \errmessage{(Inkscape) Transparency is used (non-zero) for the text in Inkscape, but the package 'transparent.sty' is not loaded}%
    \renewcommand\transparent[1]{}%
  }%
  \providecommand\rotatebox[2]{#2}%
  \newcommand*\fsize{\dimexpr\f@size pt\relax}%
  \newcommand*\lineheight[1]{\fontsize{\fsize}{#1\fsize}\selectfont}%
  \ifx\svgwidth\undefined%
    \setlength{\unitlength}{320.88700104bp}%
    \ifx\svgscale\undefined%
      \relax%
    \else%
      \setlength{\unitlength}{\unitlength * \real{\svgscale}}%
    \fi%
  \else%
    \setlength{\unitlength}{\svgwidth}%
  \fi%
  \global\let\svgwidth\undefined%
  \global\let\svgscale\undefined%
  \makeatother%
  \begin{picture}(1,0.70828267)%
    \lineheight{1}%
    \setlength\tabcolsep{0pt}%
    \put(0,0){\includegraphics[width=\unitlength,page=1]{G0825.pdf}}%
    \put(0.00413863,0.06977785){\color[rgb]{0,0,0}\makebox(0,0)[lt]{\lineheight{1.25}\smash{\begin{tabular}[t]{l}$x_1$\end{tabular}}}}%
    \put(0.08007635,0.1917347){\color[rgb]{0,0,0}\makebox(0,0)[lt]{\lineheight{1.25}\smash{\begin{tabular}[t]{l}$x_2$\end{tabular}}}}%
    \put(0.38490361,0.07427964){\color[rgb]{0,0,0}\makebox(0,0)[lt]{\lineheight{1.25}\smash{\begin{tabular}[t]{l}$x_3$\end{tabular}}}}%
    \put(0.55222043,0.05552219){\color[rgb]{0,0,0}\makebox(0,0)[lt]{\lineheight{1.25}\smash{\begin{tabular}[t]{l}$x_1$\end{tabular}}}}%
    \put(0.5877952,0.17865584){\color[rgb]{0,0,0}\makebox(0,0)[lt]{\lineheight{1.25}\smash{\begin{tabular}[t]{l}$x_2$\end{tabular}}}}%
    \put(0.78368775,0.2127099){\color[rgb]{0,0,0}\makebox(0,0)[lt]{\lineheight{1.25}\smash{\begin{tabular}[t]{l}$x_3$\end{tabular}}}}%
    \put(0.02440492,0.52603221){\color[rgb]{0,0,0}\makebox(0,0)[lt]{\lineheight{1.25}\smash{\begin{tabular}[t]{l}$x_1$\end{tabular}}}}%
    \put(0.39465748,0.52370902){\color[rgb]{0,0,0}\makebox(0,0)[lt]{\lineheight{1.25}\smash{\begin{tabular}[t]{l}$x_3$\end{tabular}}}}%
    \put(0,0){\includegraphics[width=\unitlength,page=2]{G0825.pdf}}%
    \put(0.73591996,0.66266287){\color[rgb]{0,0,0}\makebox(0,0)[lt]{\lineheight{1.25}\smash{\begin{tabular}[t]{l}$x_3$\end{tabular}}}}%
    \put(0,0){\includegraphics[width=\unitlength,page=3]{G0825.pdf}}%
    \put(0.56047368,0.51620604){\color[rgb]{0,0,0}\makebox(0,0)[lt]{\lineheight{1.25}\smash{\begin{tabular}[t]{l}$x_1$\end{tabular}}}}%
    \put(0.92812039,0.51245453){\color[rgb]{0,0,0}\makebox(0,0)[lt]{\lineheight{1.25}\smash{\begin{tabular}[t]{l}$x_1$\end{tabular}}}}%
    \put(0.59081118,0.6280545){\color[rgb]{0,0,0}\makebox(0,0)[lt]{\lineheight{1.25}\smash{\begin{tabular}[t]{l}$x_2$\end{tabular}}}}%
    \put(0.87375858,0.62337073){\color[rgb]{0,0,0}\makebox(0,0)[lt]{\lineheight{1.25}\smash{\begin{tabular}[t]{l}$x_2$\end{tabular}}}}%
    \put(0.85967589,0.37626463){\color[rgb]{0,0,0}\makebox(0,0)[lt]{\lineheight{1.25}\smash{\begin{tabular}[t]{l}$x_2$\end{tabular}}}}%
    \put(0.5844447,0.39571911){\color[rgb]{0,0,0}\makebox(0,0)[lt]{\lineheight{1.25}\smash{\begin{tabular}[t]{l}$x_2$\end{tabular}}}}%
    \put(0.23817045,0.36420718){\color[rgb]{0,0,0}\makebox(0,0)[lt]{\lineheight{1.25}\smash{\begin{tabular}[t]{l}$x_2$\end{tabular}}}}%
    \put(0.19429812,0.67056409){\color[rgb]{0,0,0}\makebox(0,0)[lt]{\lineheight{1.25}\smash{\begin{tabular}[t]{l}$x_2$\end{tabular}}}}%
    \put(0.7352296,0.37295264){\color[rgb]{0,0,0}\makebox(0,0)[lt]{\lineheight{1.25}\smash{\begin{tabular}[t]{l}$x_3$\end{tabular}}}}%
  \end{picture}%
\endgroup%

%% file: G24a.pdf_tex
\begingroup%
  \makeatletter%
  \providecommand\color[2][]{%
    \errmessage{(Inkscape) Color is used for the text in Inkscape, but the package 'color.sty' is not loaded}%
    \renewcommand\color[2][]{}%
  }%
  \providecommand\transparent[1]{%
    \errmessage{(Inkscape) Transparency is used (non-zero) for the text in Inkscape, but the package 'transparent.sty' is not loaded}%
    \renewcommand\transparent[1]{}%
  }%
  \providecommand\rotatebox[2]{#2}%
  \newcommand*\fsize{\dimexpr\f@size pt\relax}%
  \newcommand*\lineheight[1]{\fontsize{\fsize}{#1\fsize}\selectfont}%
  \ifx\svgwidth\undefined%
    \setlength{\unitlength}{195.52640533bp}%
    \ifx\svgscale\undefined%
      \relax%
    \else%
      \setlength{\unitlength}{\unitlength * \real{\svgscale}}%
    \fi%
  \else%
    \setlength{\unitlength}{\svgwidth}%
  \fi%
  \global\let\svgwidth\undefined%
  \global\let\svgscale\undefined%
  \makeatother%
  \begin{picture}(1,0.55619701)%
    \lineheight{1}%
    \setlength\tabcolsep{0pt}%
    \put(0,0){\includegraphics[width=\unitlength,page=1]{G24a.pdf}}%
    \put(0.04166667,0.27734372){\color[rgb]{0,0,0}\makebox(0,0)[lt]{\lineheight{1.25}\smash{\begin{tabular}[t]{l}$F_1$\end{tabular}}}}%
    \put(0.87630209,0.27343749){\color[rgb]{0,0,0}\makebox(0,0)[lt]{\lineheight{1.25}\smash{\begin{tabular}[t]{l}$F_2$\end{tabular}}}}%
    \put(0.54166668,0.03645832){\color[rgb]{0,0,0}\makebox(0,0)[lt]{\lineheight{1.25}\smash{\begin{tabular}[t]{l}$F_3$\end{tabular}}}}%
  \end{picture}%
\endgroup%

%% file: G4.pdf_tex
\begingroup%
  \makeatletter%
  \providecommand\color[2][]{%
    \errmessage{(Inkscape) Color is used for the text in Inkscape, but the package 'color.sty' is not loaded}%
    \renewcommand\color[2][]{}%
  }%
  \providecommand\transparent[1]{%
    \errmessage{(Inkscape) Transparency is used (non-zero) for the text in Inkscape, but the package 'transparent.sty' is not loaded}%
    \renewcommand\transparent[1]{}%
  }%
  \providecommand\rotatebox[2]{#2}%
  \newcommand*\fsize{\dimexpr\f@size pt\relax}%
  \newcommand*\lineheight[1]{\fontsize{\fsize}{#1\fsize}\selectfont}%
  \ifx\svgwidth\undefined%
    \setlength{\unitlength}{148.18840027bp}%
    \ifx\svgscale\undefined%
      \relax%
    \else%
      \setlength{\unitlength}{\unitlength * \real{\svgscale}}%
    \fi%
  \else%
    \setlength{\unitlength}{\svgwidth}%
  \fi%
  \global\let\svgwidth\undefined%
  \global\let\svgscale\undefined%
  \makeatother%
  \begin{picture}(1,0.63675181)%
    \lineheight{1}%
    \setlength\tabcolsep{0pt}%
    \put(0,0){\includegraphics[width=\unitlength,page=1]{G4.pdf}}%
    \put(0.139748,0.30469642){\color[rgb]{0,0,0}\makebox(0,0)[lt]{\lineheight{1.25}\smash{\begin{tabular}[t]{l}(a)\end{tabular}}}}%
    \put(0.45131732,0.30927833){\color[rgb]{0,0,0}\makebox(0,0)[lt]{\lineheight{1.25}\smash{\begin{tabular}[t]{l}(b)\end{tabular}}}}%
    \put(0.78923255,0.30584192){\color[rgb]{0,0,0}\makebox(0,0)[lt]{\lineheight{1.25}\smash{\begin{tabular}[t]{l}(c)\end{tabular}}}}%
    \put(0.43642612,-0.01145472){\color[rgb]{0,0,0}\makebox(0,0)[lt]{\lineheight{1.25}\smash{\begin{tabular}[t]{l}(d)\end{tabular}}}}%
  \end{picture}%
\endgroup%

%% file: G0822.pdf_tex
\begingroup%
  \makeatletter%
  \providecommand\color[2][]{%
    \errmessage{(Inkscape) Color is used for the text in Inkscape, but the package 'color.sty' is not loaded}%
    \renewcommand\color[2][]{}%
  }%
  \providecommand\transparent[1]{%
    \errmessage{(Inkscape) Transparency is used (non-zero) for the text in Inkscape, but the package 'transparent.sty' is not loaded}%
    \renewcommand\transparent[1]{}%
  }%
  \providecommand\rotatebox[2]{#2}%
  \newcommand*\fsize{\dimexpr\f@size pt\relax}%
  \newcommand*\lineheight[1]{\fontsize{\fsize}{#1\fsize}\selectfont}%
  \ifx\svgwidth\undefined%
    \setlength{\unitlength}{1146bp}%
    \ifx\svgscale\undefined%
      \relax%
    \else%
      \setlength{\unitlength}{\unitlength * \real{\svgscale}}%
    \fi%
  \else%
    \setlength{\unitlength}{\svgwidth}%
  \fi%
  \global\let\svgwidth\undefined%
  \global\let\svgscale\undefined%
  \makeatother%
  \begin{picture}(1,0.54908377)%
    \lineheight{1}%
    \setlength\tabcolsep{0pt}%
    \put(0,0){\includegraphics[width=\unitlength,page=1]{G0822.pdf}}%
    \put(0.19107966,0.5097094){\color[rgb]{0,0,0}\makebox(0,0)[lt]{\lineheight{1.25}\smash{\begin{tabular}[t]{l}$B_1$\end{tabular}}}}%
    \put(0.183631,0.39306461){\color[rgb]{0,0,0}\makebox(0,0)[lt]{\lineheight{1.25}\smash{\begin{tabular}[t]{l}$B_2$\end{tabular}}}}%
    \put(0.33769257,0.48720803){\color[rgb]{0,0,0}\makebox(0,0)[lt]{\lineheight{1.25}\smash{\begin{tabular}[t]{l}$B_3$\end{tabular}}}}%
    \put(0.24105837,0.39410266){\color[rgb]{0,0,0}\makebox(0,0)[lt]{\lineheight{1.25}\smash{\begin{tabular}[t]{l}$B_4$\end{tabular}}}}%
    \put(0.14760256,0.3122634){\color[rgb]{0,0,0}\makebox(0,0)[lt]{\lineheight{1.25}\smash{\begin{tabular}[t]{l}$B_5$\end{tabular}}}}%
    \put(0.03712853,0.41541974){\color[rgb]{0,0,0}\makebox(0,0)[lt]{\lineheight{1.25}\smash{\begin{tabular}[t]{l}$B_\square$\end{tabular}}}}%
    \put(0.07750811,0.49231467){\color[rgb]{0,0,0}\makebox(0,0)[lt]{\lineheight{1.25}\smash{\begin{tabular}[t]{l}$F_1$\end{tabular}}}}%
    \put(0.36064263,0.3593704){\color[rgb]{0,0,0}\makebox(0,0)[lt]{\lineheight{1.25}\smash{\begin{tabular}[t]{l}$F_2$\end{tabular}}}}%
    \put(0.68754898,0.44094983){\color[rgb]{0,0,0}\makebox(0,0)[lt]{\lineheight{1.25}\smash{\begin{tabular}[t]{l}$B_\square$\end{tabular}}}}%
    \put(0.29553555,0.10341075){\color[rgb]{0,0,0}\makebox(0,0)[lt]{\lineheight{1.25}\smash{\begin{tabular}[t]{l}$B_\square$\end{tabular}}}}%
    \put(0.7213778,0.16756959){\color[rgb]{0,0,0}\makebox(0,0)[lt]{\lineheight{1.25}\smash{\begin{tabular}[t]{l}$f_1$\end{tabular}}}}%
    \put(0.64023108,0.06270095){\color[rgb]{0,0,0}\makebox(0,0)[lt]{\lineheight{1.25}\smash{\begin{tabular}[t]{l}$f_2$\end{tabular}}}}%
    \put(0.5931808,0.2865796){\color[rgb]{0,0,0}\makebox(0,0)[lt]{\lineheight{1.25}\smash{\begin{tabular}[t]{l}$B$\end{tabular}}}}%
    \put(0.5957421,0.00643539){\color[rgb]{0,0,0}\makebox(0,0)[lt]{\lineheight{1.25}\smash{\begin{tabular}[t]{l}$B$\end{tabular}}}}%
    \put(0.20580705,0.00802788){\color[rgb]{0,0,0}\makebox(0,0)[lt]{\lineheight{1.25}\smash{\begin{tabular}[t]{l}$B$\end{tabular}}}}%
    \put(0.18253232,0.262122){\color[rgb]{0,0,0}\makebox(0,0)[lt]{\lineheight{1.25}\smash{\begin{tabular}[t]{l}{ (a)}\end{tabular}}}}%
    \put(0.71170122,0.2752481){\color[rgb]{0,0,0}\makebox(0,0)[lt]{\lineheight{1.25}\smash{\begin{tabular}[t]{l}{(b)}\end{tabular}}}}%
    \put(0.70324924,-0.00542743){\color[rgb]{0,0,0}\makebox(0,0)[lt]{\lineheight{1.25}\smash{\begin{tabular}[t]{l}{ (d)}\end{tabular}}}}%
    \put(0.18421142,-0.02937594){\color[rgb]{0,0,0}\makebox(0,0)[lt]{\lineheight{1.25}\smash{\begin{tabular}[t]{l}{(c)}\end{tabular}}}}%
    \put(0,0){\includegraphics[width=\unitlength,page=2]{G0822.pdf}}%
    \put(0.86700256,0.20274068){\color[rgb]{0,0,0}\makebox(0,0)[lt]{\lineheight{1.25}\smash{\begin{tabular}[t]{l}$f_1$\end{tabular}}}}%
    \put(0.93308329,0.0953188){\color[rgb]{0,0,0}\makebox(0,0)[lt]{\lineheight{1.25}\smash{\begin{tabular}[t]{l}$f_2$\end{tabular}}}}%
  \end{picture}%
\endgroup%

%% file: G22.pdf_tex
\begingroup%
  \makeatletter%
  \providecommand\color[2][]{%
    \errmessage{(Inkscape) Color is used for the text in Inkscape, but the package 'color.sty' is not loaded}%
    \renewcommand\color[2][]{}%
  }%
  \providecommand\transparent[1]{%
    \errmessage{(Inkscape) Transparency is used (non-zero) for the text in Inkscape, but the package 'transparent.sty' is not loaded}%
    \renewcommand\transparent[1]{}%
  }%
  \providecommand\rotatebox[2]{#2}%
  \newcommand*\fsize{\dimexpr\f@size pt\relax}%
  \newcommand*\lineheight[1]{\fontsize{\fsize}{#1\fsize}\selectfont}%
  \ifx\svgwidth\undefined%
    \setlength{\unitlength}{263.80197144bp}%
    \ifx\svgscale\undefined%
      \relax%
    \else%
      \setlength{\unitlength}{\unitlength * \real{\svgscale}}%
    \fi%
  \else%
    \setlength{\unitlength}{\svgwidth}%
  \fi%
  \global\let\svgwidth\undefined%
  \global\let\svgscale\undefined%
  \makeatother%
  \begin{picture}(1,0.38271551)%
    \lineheight{1}%
    \setlength\tabcolsep{0pt}%
    \put(0,0){\includegraphics[width=\unitlength,page=1]{G22.pdf}}%
    \put(0.23477476,0.37564441){\color[rgb]{0,0,0}\makebox(0,0)[lt]{\lineheight{1.25}\smash{\begin{tabular}[t]{l}$F$\end{tabular}}}}%
    \put(0.65314686,0.3748252){\color[rgb]{0,0,0}\makebox(0,0)[lt]{\lineheight{1.25}\smash{\begin{tabular}[t]{l}$F'$\end{tabular}}}}%
    \put(0.3765436,0.09143874){\color[rgb]{0,0,0}\makebox(0,0)[lt]{\lineheight{1.25}\smash{\begin{tabular}[t]{l}$B^+$\end{tabular}}}}%
    \put(0.47090205,0.0923077){\color[rgb]{0,0,0}\makebox(0,0)[lt]{\lineheight{1.25}\smash{\begin{tabular}[t]{l}$B^-$\end{tabular}}}}%
    \put(0.12502574,0.07926097){\color[rgb]{0,0,0}\makebox(0,0)[lt]{\lineheight{1.25}\smash{\begin{tabular}[t]{l}$x_0$\end{tabular}}}}%
    \put(0.73875358,0.07120519){\color[rgb]{0,0,0}\makebox(0,0)[lt]{\lineheight{1.25}\smash{\begin{tabular}[t]{l}$x_0$\end{tabular}}}}%
    \put(0.1780129,0.23240575){\color[rgb]{0,0,0}\makebox(0,0)[lt]{\lineheight{1.25}\smash{\begin{tabular}[t]{l}$s_F$\end{tabular}}}}%
    \put(0.71699637,0.23240575){\color[rgb]{0,0,0}\makebox(0,0)[lt]{\lineheight{1.25}\smash{\begin{tabular}[t]{l}$s_{F'}$\end{tabular}}}}%
    \put(0.33426867,0.21658236){\color[rgb]{0,0,0}\makebox(0,0)[lt]{\lineheight{1.25}\smash{\begin{tabular}[t]{l}$t_{B^+}$\end{tabular}}}}%
    \put(0.55480684,0.21658236){\color[rgb]{0,0,0}\makebox(0,0)[lt]{\lineheight{1.25}\smash{\begin{tabular}[t]{l}$t_{B^-}$\end{tabular}}}}%
  \end{picture}%
\endgroup%

%% file: G25.pdf_tex
\begingroup%
  \makeatletter%
  \providecommand\color[2][]{%
    \errmessage{(Inkscape) Color is used for the text in Inkscape, but the package 'color.sty' is not loaded}%
    \renewcommand\color[2][]{}%
  }%
  \providecommand\transparent[1]{%
    \errmessage{(Inkscape) Transparency is used (non-zero) for the text in Inkscape, but the package 'transparent.sty' is not loaded}%
    \renewcommand\transparent[1]{}%
  }%
  \providecommand\rotatebox[2]{#2}%
  \newcommand*\fsize{\dimexpr\f@size pt\relax}%
  \newcommand*\lineheight[1]{\fontsize{\fsize}{#1\fsize}\selectfont}%
  \ifx\svgwidth\undefined%
    \setlength{\unitlength}{411.75bp}%
    \ifx\svgscale\undefined%
      \relax%
    \else%
      \setlength{\unitlength}{\unitlength * \real{\svgscale}}%
    \fi%
  \else%
    \setlength{\unitlength}{\svgwidth}%
  \fi%
  \global\let\svgwidth\undefined%
  \global\let\svgscale\undefined%
  \makeatother%
  \begin{picture}(1,0.68852459)%
    \lineheight{1}%
    \setlength\tabcolsep{0pt}%
    \put(0,0){\includegraphics[width=\unitlength,page=1]{G25.pdf}}%
    \put(0.87028052,0.07210083){\color[rgb]{0,0,0}\makebox(0,0)[lt]{\lineheight{1.25}\smash{\begin{tabular}[t]{l}$B_1$\end{tabular}}}}%
    \put(0.66570224,0.06472278){\color[rgb]{0,0,0}\makebox(0,0)[lt]{\lineheight{1.25}\smash{\begin{tabular}[t]{l}$B_2$\end{tabular}}}}%
    \put(0.38167265,0.07329989){\color[rgb]{0,0,0}\makebox(0,0)[lt]{\lineheight{1.25}\smash{\begin{tabular}[t]{l}$B_{k-1}$\end{tabular}}}}%
    \put(0.22595915,0.0733785){\color[rgb]{0,0,0}\makebox(0,0)[lt]{\lineheight{1.25}\smash{\begin{tabular}[t]{l}$B_k$\end{tabular}}}}%
    \put(0.83298741,0.35475844){\color[rgb]{0,0,0}\makebox(0,0)[lt]{\lineheight{1.25}\smash{\begin{tabular}[t]{l}$s$\end{tabular}}}}%
    \put(0.62776951,0.35084033){\color[rgb]{0,0,0}\makebox(0,0)[lt]{\lineheight{1.25}\smash{\begin{tabular}[t]{l}$\phi_1(s)$\end{tabular}}}}%
    \put(0.05219786,0.33978156){\color[rgb]{0,0,0}\makebox(0,0)[lt]{\lineheight{1.25}\smash{\begin{tabular}[t]{l}$\phi_k\circ\cdots\circ\phi_1(s)$\end{tabular}}}}%
    \put(0.88034194,0.2296532){\color[rgb]{0,0,0}\makebox(0,0)[lt]{\lineheight{1.25}\smash{\begin{tabular}[t]{l}$F$\end{tabular}}}}%
    \put(0,0){\includegraphics[width=\unitlength,page=2]{G25.pdf}}%
    \put(0.98247151,0.33232163){\color[rgb]{0,0,0}\makebox(0,0)[lt]{\lineheight{1.25}\smash{\begin{tabular}[t]{l}$s_2$\end{tabular}}}}%
    \put(-0.01609482,0.34304547){\color[rgb]{0,0,0}\makebox(0,0)[lt]{\lineheight{1.25}\smash{\begin{tabular}[t]{l}$s_4$\end{tabular}}}}%
    \put(0.8562938,0.54920407){\color[rgb]{0,0,0}\makebox(0,0)[lt]{\lineheight{1.25}\smash{\begin{tabular}[t]{l}$s_1$\end{tabular}}}}%
    \put(0.70406753,0.54480835){\color[rgb]{0,0,0}\makebox(0,0)[lt]{\lineheight{1.25}\smash{\begin{tabular}[t]{l}$s_1^{(1)}$\end{tabular}}}}%
    \put(0.20460979,0.53586976){\color[rgb]{0,0,0}\makebox(0,0)[lt]{\lineheight{1.25}\smash{\begin{tabular}[t]{l}$s_1^{(k-1)}$\end{tabular}}}}%
    \put(0.06212802,0.53998772){\color[rgb]{0,0,0}\makebox(0,0)[lt]{\lineheight{1.25}\smash{\begin{tabular}[t]{l}$s_1^{(k)}$\end{tabular}}}}%
    \put(0.06559384,0.14652558){\color[rgb]{0,0,0}\makebox(0,0)[lt]{\lineheight{1.25}\smash{\begin{tabular}[t]{l}$s_3^{(k)}$\end{tabular}}}}%
    \put(0.26311681,0.14269021){\color[rgb]{0,0,0}\makebox(0,0)[lt]{\lineheight{1.25}\smash{\begin{tabular}[t]{l}$s_3^{({k-1})}$\end{tabular}}}}%
    \put(0.71606169,0.14054465){\color[rgb]{0,0,0}\makebox(0,0)[lt]{\lineheight{1.25}\smash{\begin{tabular}[t]{l}$s_3^{(1)}$\end{tabular}}}}%
    \put(0.88636549,0.14729797){\color[rgb]{0,0,0}\makebox(0,0)[lt]{\lineheight{1.25}\smash{\begin{tabular}[t]{l}$s_3$\end{tabular}}}}%
  \end{picture}%
\endgroup%

%% file: G0920.pdf_tex
\begingroup%
  \makeatletter%
  \providecommand\color[2][]{%
    \errmessage{(Inkscape) Color is used for the text in Inkscape, but the package 'color.sty' is not loaded}%
    \renewcommand\color[2][]{}%
  }%
  \providecommand\transparent[1]{%
    \errmessage{(Inkscape) Transparency is used (non-zero) for the text in Inkscape, but the package 'transparent.sty' is not loaded}%
    \renewcommand\transparent[1]{}%
  }%
  \providecommand\rotatebox[2]{#2}%
  \newcommand*\fsize{\dimexpr\f@size pt\relax}%
  \newcommand*\lineheight[1]{\fontsize{\fsize}{#1\fsize}\selectfont}%
  \ifx\svgwidth\undefined%
    \setlength{\unitlength}{1056.75bp}%
    \ifx\svgscale\undefined%
      \relax%
    \else%
      \setlength{\unitlength}{\unitlength * \real{\svgscale}}%
    \fi%
  \else%
    \setlength{\unitlength}{\svgwidth}%
  \fi%
  \global\let\svgwidth\undefined%
  \global\let\svgscale\undefined%
  \makeatother%
  \begin{picture}(1,0.54506742)%
    \lineheight{1}%
    \setlength\tabcolsep{0pt}%
    \put(0,0){\includegraphics[width=\unitlength,page=1]{G0920.pdf}}%
    \put(0.78982251,0.22727307){\color[rgb]{0,0,0}\makebox(0,0)[lt]{\lineheight{1.25}\smash{\begin{tabular}[t]{l}Push\end{tabular}}}}%
    \put(0.68489587,0.21484374){\color[rgb]{0,0,0}\makebox(0,0)[lt]{\lineheight{1.25}\smash{\begin{tabular}[t]{l}$B_i$\end{tabular}}}}%
    \put(0.47395835,0.21744793){\color[rgb]{0,0,0}\makebox(0,0)[lt]{\lineheight{1.25}\smash{\begin{tabular}[t]{l}$B_j$\end{tabular}}}}%
    \put(0.79231775,0.33007812){\color[rgb]{0,0,0}\makebox(0,0)[lt]{\lineheight{1.25}\smash{\begin{tabular}[t]{l}$F_{f_1}$\end{tabular}}}}%
    \put(0.81298861,0.12429612){\color[rgb]{0,0,0}\makebox(0,0)[lt]{\lineheight{1.25}\smash{\begin{tabular}[t]{l}$F_{f_2}$\end{tabular}}}}%
    \put(0,0){\includegraphics[width=\unitlength,page=2]{G0920.pdf}}%
    \put(0.59800243,0.20255662){\color[rgb]{0,0,0}\makebox(0,0)[lt]{\lineheight{1.25}\smash{\begin{tabular}[t]{l}$A$\end{tabular}}}}%
  \end{picture}%
\endgroup%

%% file: G0913.pdf_tex
\begingroup%
  \makeatletter%
  \providecommand\color[2][]{%
    \errmessage{(Inkscape) Color is used for the text in Inkscape, but the package 'color.sty' is not loaded}%
    \renewcommand\color[2][]{}%
  }%
  \providecommand\transparent[1]{%
    \errmessage{(Inkscape) Transparency is used (non-zero) for the text in Inkscape, but the package 'transparent.sty' is not loaded}%
    \renewcommand\transparent[1]{}%
  }%
  \providecommand\rotatebox[2]{#2}%
  \newcommand*\fsize{\dimexpr\f@size pt\relax}%
  \newcommand*\lineheight[1]{\fontsize{\fsize}{#1\fsize}\selectfont}%
  \ifx\svgwidth\undefined%
    \setlength{\unitlength}{819bp}%
    \ifx\svgscale\undefined%
      \relax%
    \else%
      \setlength{\unitlength}{\unitlength * \real{\svgscale}}%
    \fi%
  \else%
    \setlength{\unitlength}{\svgwidth}%
  \fi%
  \global\let\svgwidth\undefined%
  \global\let\svgscale\undefined%
  \makeatother%
  \begin{picture}(1,0.86080586)%
    \lineheight{1}%
    \setlength\tabcolsep{0pt}%
    \put(0,0){\includegraphics[width=\unitlength,page=1]{G0913.pdf}}%
    \put(0.16829015,0.52292461){\color[rgb]{0,0,0}\makebox(0,0)[lt]{\lineheight{1.25}\smash{\begin{tabular}[t]{l}(a)\end{tabular}}}}%
    \put(0.40944159,0.51647672){\color[rgb]{0,0,0}\makebox(0,0)[lt]{\lineheight{1.25}\smash{\begin{tabular}[t]{l}(b)\end{tabular}}}}%
    \put(0.63926328,0.50308867){\color[rgb]{0,0,0}\makebox(0,0)[lt]{\lineheight{1.25}\smash{\begin{tabular}[t]{l}(c)\end{tabular}}}}%
    \put(0.89798452,0.50567175){\color[rgb]{0,0,0}\makebox(0,0)[lt]{\lineheight{1.25}\smash{\begin{tabular}[t]{l}(d)\end{tabular}}}}%
    \put(0.73839577,0.25991513){\color[rgb]{0,0,0}\makebox(0,0)[lt]{\lineheight{1.25}\smash{\begin{tabular}[t]{l}(e)\end{tabular}}}}%
    \put(0.74481125,-0.01779095){\color[rgb]{0,0,0}\makebox(0,0)[lt]{\lineheight{1.25}\smash{\begin{tabular}[t]{l}(f)\end{tabular}}}}%
    \put(0.16172781,-0.01337252){\color[rgb]{0,0,0}\makebox(0,0)[lt]{\lineheight{1.25}\smash{\begin{tabular}[t]{l}Ideal triangles \end{tabular}}}}%
    \put(0,0){\includegraphics[width=\unitlength,page=2]{G0913.pdf}}%
  \end{picture}%
\endgroup%